\documentclass[11pt]{llncs}

\usepackage{amssymb,amsmath,mathtools,url}
\usepackage[dvips]{graphics}

\usepackage{enumerate}
\usepackage{mleftright}
\usepackage[symbol]{footmisc}
\usepackage[affil-it]{authblk}
\usepackage{enumitem}   

\usepackage{tabularx, booktabs}	
\newcolumntype{Y}{>{\centering\arraybackslash}X}

\makeatletter
\newcommand*{\myfnsymbolsingle}[1]{%
  \ensuremath{%
    \ifcase#1% 0
    \or % 1
      *%   
    \or % 2
      \dagger
    \or % 3  
      \ddagger
    \or % 4   
      \mathsection
    \or % 5
      \mathparagraph
    \else % >= 6
      \@ctrerr  
    \fi
  }%   
}   
\makeatother

\usepackage{alphalph}
\newalphalph{\myfnsymbolmult}[mult]{\myfnsymbolsingle}{}

\newcommand{\F}{{\mathbb F}}

\newcommand{\Z}{{\mathbb Z}}

\newcommand{\Q}{{\mathbb Q}}
\renewcommand{\vec}[1]{\mathbf{#1}}

\newcommand{\ie}{i.e., }
\newcommand{\eg}{e.g., }
\newcommand{\ea}{{\em et al. }}

\usepackage{vmargin}
\setpapersize{A4}
\setmarginsrb{1.65cm}{1.65cm}{1.65cm}{1.65cm}{0pt}{1.0cm}{0pt}{1.0cm}

\begin{document}

\title{On the Enumeration of Irreducible Polynomials over $\text{GF}(q)$\\
with Prescribed Coefficients}

\author{Robert Granger}

\institute{
Secure Systems Research Group\\
Department of Computer Science\\
University of Surrey\\ 
United Kingdom\\
\email{r.granger@surrey.ac.uk}
}

%\institute{  
%Laboratory for Cryptologic Algorithms\\
%\'Ecole polytechnique f\'ed\'erale de Lausanne\\
%Switzerland\\
%\email{robert.granger@epfl.ch}
%}

\maketitle

\begin{center}
{\em Dedicated to the memory of Oonagh N\'i Ch\'eileachair.\\ 
Sonas ort.}
\end{center}

\begin{abstract}
We present an efficient deterministic algorithm which outputs exact expressions in terms of $n$ for the number of monic degree $n$ 
irreducible polynomials over $\F_{q}$ of characteristic $p$ for which the first $l < p$ coefficients are prescribed, provided that 
$n$ is coprime to $p$. Each of these counts is $\frac{1}{n}(q^{n-l} + \mathcal{O}(q^{n/2}))$. The main idea behind the algorithm 
is to associate to an equivalent problem a set of Artin-Schreier curves defined over $\F_q$ whose number of $\F_{q^n}$-rational affine 
points must be combined. This is accomplished by computing their zeta functions using a $p$-adic algorithm due to Lauder and Wan. 
Using the computational algebra system Magma one can, for example, compute the zeta functions of the arising curves for $q=5$ and $l=4$ 
very efficiently, and we detail a proof-of-concept demonstration. Due to the failure of Newton's identities in positive 
characteristic, the $l \ge p$ cases are seemingly harder. Nevertheless, we use an analogous algorithm to compute example curves for $q = 2$ and $l \le 7$, 
and for $q = 3$ and $l = 3$. Again using Magma, for $q = 2$ we computed the relevant zeta functions for $l = 4$ and $l = 5$, obtaining explicit formulae 
for these open problems for $n$ odd, as well as for subsets of these problems for all $n$, while for $q = 3$ we obtained explicit formulae for $l = 3$ and 
$n$ coprime to $3$. We also discuss some of the computational challenges and theoretical questions arising from this approach in the general case 
and propose some natural open problems.


\end{abstract}

%--------------------------------------

\section{Introduction}\label{sec:intro}

For $q = p^r$ a prime power let $\F_q$ denote the finite field of $q$ elements, and let $I_q(n)$ denote the number of monic 
irreducible polynomials in $\F_q[x]$ of degree $n$. A classical result due to Gauss~\cite[pp. 602-629]{gauss} states that
\[
I_q(n) = \frac{1}{n} \sum_{d \mid n} \mu(d) q^{n/d}.
\]
A natural problem is to determine the number of monic irreducible polynomials in $\F_q[x]$ of degree $n$ for which 
certain coefficients are prescribed, which we refer to in general as the 
{\em prescribed coefficients problem}. As Panario has stated~\cite[p. 115]{Panarioquote},
{\em ``The long-term goal here is to provide existence and counting results for
irreducibles with any number of prescribed coefficients to any given values. This
goal is completely out of reach at this time. Incremental steps seem doable, but
it would be most interesting if new techniques were introduced to attack these
problems.''}

Regarding existence, building upon work of Bourgain~\cite{Bourgain} and Pollack~\cite{Pollack}, the best result to date is due 
to Ha~\cite{Ha}, who in 2016 proved that there exists a monic irreducible polynomial in $\F_q[x]$ of degree $n$ with up to 
$\lfloor{(1/4 - \epsilon)n}\rfloor$ coefficients prescribed in any positions, for any $\epsilon > 0$ and $q$ sufficiently large.

In contrast to Ha's important progress towards the above long-term goal on the existence side, the corpus of counting results 
is far less well developed. Nearly all such research has focused on the subproblem of determining the number of monic irreducible polynomials 
in $\F_q[x]$ of degree $n$ for which the first $l$ coefficients have the prescribed values $t_1,\ldots,t_l$, which we denote by 
$I_q(n,t_1,\ldots,t_l)$. Although asymptotics for such subproblems have been obtained by Cohen~\cite{Cohen}, very few 
exact results are known. In 1952 Carlitz gave formulae for $I_q(n,t_1)$~\cite{Carlitz}, while
in 1990 Kuz'min gave formulae for $I_q(n,t_1,t_2)$~\cite{kuzmin1,kuzmin2}; Cattell~\ea later reproduced Kuz'min's results
for the base field $\F_2$, in 1999~\cite{cattell}. In 2001 the three coefficient case $I_2(n,t_1,t_2,t_3)$ was solved, by Yucas and Mullen 
for $n$ even~\cite{yucasmullen} and by Fitzgerald and Yucas for $n$ odd~\cite{fitzyucas}. 
Formulae for $I_{2^r}(n,t_1,t_2)$ for all $r \ge 1$ were given again in 2013 by Ri \ea~\cite{RMR}. 
Most recently, in 2016 Ahmadi \ea gave formulae for $I_{2^r}(n,0,0,0)$ for all $r \ge 1$~\cite{AGGMY}.

Rather than study the above subproblem instances directly, the papers~\cite{cattell,yucasmullen,fitzyucas,RMR,AGGMY} all study a set of equivalent 
problems, namely counting the number of elements of $\F_{q^n}$ with correspondingly prescribed traces, which we refer to in general as the 
{\em prescribed traces problem}. In particular, for $a \in \F_{q^n}$ the 
characteristic polynomial of $a$ with respect to the extension $\F_{q^n}/\F_q$ is defined to be: 
\begin{equation*}
\prod_{i = 0}^{n-1} (x - a^{q^i}) = x^n - T_1(a)x^{n-1} +  T_2(a)x^{n-2} - \cdots + (-1)^{n-1} T_{n-1}(a)x + (-1)^n T_n(a),
\end{equation*}
with $T_l: \F_{q^n} \rightarrow \F_q$, $1 \le l \le n$ the successive trace functions
\begin{eqnarray*}
T_1(a) &=& \sum_{i = 0}^{n-1} a^{q^i},\\
T_2(a) &=& \sum_{0 \le i_1 < i_2 \le n-1} a^{q^{i_1} + q^{i_2}},\\
T_3(a) &=& \sum_{0 \le i_1 < i_2 < i_3 \le n-1} a^{q^{i_1} + q^{i_2} + q^{i_3}},\\
       &\vdots&\\
T_l(a) &=& \sum_{0 \le i_1 < \cdots < i_l \le n-1} a^{q^{i_1} + \cdots + q^{i_l}},\\
       &\vdots&\\
T_n(a) &=& a^{1 + q + q^2 + \cdots + q^{n-1}}.
\end{eqnarray*}
For any $n \ge l$ and $t_1,\ldots,t_l \in \F_q$, let $F_q(n,t_1,\ldots,t_l)$ denote the number of elements $a \in \F_{q^n}$ for which $T_1(a) = t_1$, \ldots, $T_l(a) = t_l$. 
If for a given $q = p^r$ and $l \ge 1$ one determines $F_q(n,t_1,\ldots,t_l)$ for all $(t_1,\ldots,t_l) \in (\F_q)^l$, then an application of the multinomial theorem and a 
generalised M\"obius inversion-type argument gives $I_q(n,t_1,\ldots,t_l)$ for all $(t_1,\ldots,t_l) \in (\F_q)^l$, and vice versa, hence the equivalence. 
The formulae for the equivalence follow from the approach of Miers and Ruskey~\cite{miersruskey2}, 
extending the subcases already proven for binary fields~\cite{cattell,yucasmullen,fitzyucas,RMR,AGGMY}. 
For our main case of interest for which $l < p$, each $I_q(n,t_1,\ldots,t_l)$ can be expressed in terms of only one $F_q(n,t_{1}',\ldots,t_{l}')$, 
and vice versa, see~\S\ref{sec:transform0}.

As well as proving formulae for $I_{2^r}(n,0,0,0)$ for all $r \ge 1$, the work~\cite{AGGMY} also explained an intriguing phenomenon, 
which is that the formulae for $F_2(n,t_1,t_2)$ and $F_2(n,t_1,t_2,t_3)$ proven in~\cite{cattell} and~\cite{yucasmullen,fitzyucas} 
depend on $n \bmod 8$ and on $n \bmod 24$, respectively. In particular, by Fourier analysing the formulae the present author showed 
that they are related to the number of $\F_{2^n}$-rational affine points of certain genus one and two supersingular curves defined 
over $\F_2$. A simple argument gave a new derivation of the formulae for the $t_1 = 0$ cases and the $n$ odd cases, and since the 
curves featured are supersingular their normalised Weil numbers are roots of unity, which explains the observed periodicity.

In this paper we greatly extend the curve-based approach of~\cite{AGGMY} in order to prove the following theorem.
Let $\overline{\Q}$ be an algebraic closure of $\Q$ and let $\overline{\Z}$ be the integral closure of $\Z$ in $\overline{\Q}$.

\begin{theorem}\label{thm:maintheorem}
Let $q = p^r$, $l < p$ and let
\begin{equation*}\label{eq:N}
N = \begin{cases}
\begin{array}{ll}
(q^l(ql - q -l)+q)  &  \ \text{if}  \ r = 1 \\
(q^l(ql - q -l)+q)(p-1)  &  \ \text{if}  \ r > 1.\\
\end{array}
\end{cases}
\end{equation*}
Then for each $\overline{n} \in \{1,\ldots,p-1\}$ there exist $\omega_1,\ldots,\omega_N \in \overline{\Z}$ (not necessarily distinct), which are all of norm $\sqrt{q}$, such that for all $(t_1,\ldots,t_l) \in (\F_q)^l$ there exist explicitly determined $\upsilon_1,\ldots, \upsilon_N \in \{0,\pm1\}$
such that for all $n \equiv \overline{n} \pmod{p}$ with $n \ge l$ one has
\begin{equation}\label{eq:maintheorem}
F_q(n,t_1,\ldots,t_l) = \frac{1}{q^l}\Big( q^n + \frac{q-1}{p-1} \sum_{k=1}^N \upsilon_k \omega_{k}^n \Big) = q^{n-l} + \mathcal{O}(q^{n/2}).
\end{equation}
Moreover, there exists an explicit deterministic algorithm which for each $\overline{n}$ computes $\omega_1,\ldots,\omega_N$ 
in $q^{l}\cdot\tilde{\mathcal{O}}(l^5 p^4)$ bit operations when $r=1$, and $q^{l+1}\cdot\tilde{\mathcal{O}}(l^5 p^4 r^3)$ bit operations when $r>1$.
\end{theorem}
Here we use the soft-oh notation $\tilde{\mathcal{O}}$, which ignores factors which are logarithmic in the argument. Note that
Theorem~\ref{thm:maintheorem} implies an asymptotic equidistribution result for the prescribed traces problem, as one expects.

The main idea behind the algorithm is to associate to the prescribed traces problem a fibre product of $l$ Artin-Schreier curves defined over $\F_q$,
whose number of $\F_{q^n}$-rational affine points solves the problem. In order to count this number we propose two methods.
The first -- which we refer to as the {\em direct method} -- computes a single $F_q(n,t_1,\ldots,t_l)$ by replacing the fibre product
by an intersection. However, this results in curves of relatively large genus for which we do not know of an efficient point counting algorithm.
The second -- which we refer to as the {\em indirect}, or {\em `batching' method} -- determines all of the $q^l$ counts 
$F_q(n,t_1,\ldots,t_l)$ simultaneously, by transforming these counting problems into a set of $q^l$ equivalent counting problems.
All but one of the latter problems requires computing the zeta function of an Artin-Schreier curve defined over $\F_q$,
which is accomplished using a $p$-adic algorithm due to Lauder and Wan~\cite{LauderWan}. 
While obtaining the relevant curves is immediate, determining their zeta functions is a fundamentally computational problem, so one 
should not expect to be able to simply write down general formulae for $F_q(n,t_1,\ldots,t_l)$. 
There may of course exist faster algorithms for determining the exact formulae than the approach we take, particularly when only one
$F_q(n,t_1,\ldots,t_l)$ is required, but we emphasise that the one presented here constitutes the first algorithmic approach to solving 
the prescribed traces problem exactly, which therefore represents a shift in perspective with regard to its study.

Although the indirect method is very efficient for computing all $q^l$ counts $F_q(n,t_1,\ldots,t_l)$ simultaneously,
a potential disadvantage is that the number $N$ of algebraic integers featuring in Theorem~\ref{thm:maintheorem} may be larger than for
the direct method, even accounting for multiplicities. There may, therefore, be some redundancy in the resulting formulae. 
Such redundancy can be eliminated with some postcomputation by identifying, for each $\overline{n}$, cancellations between linear 
combinations of $n$-th powers of the $\omega_k$'s which are valid for all $n \equiv \overline{n} \pmod{p}$, see~\S\ref{sec:ternaryzetas} for an example.
One should therefore attempt to use the direct method first, and if it is not computationally feasible then use the indirect method, 
noting that some redundancy elimination may be required.

It should be no surprise that the formulae for $F_q(n,t_1,\ldots,t_l)$ arise from the number of points on curves. As noted by Voloch~\cite{voloch}, 
a method used by Hayes~\cite{hayes}, Hsu~\cite{hsu}, Voloch himself and others to estimate $I_q(n,t_1,\ldots,t_l)$ relates these counts to the number 
of points over $\F_{q^n}$ of certain curves defined over $\F_q$ whose function fields are subfields of the so-called cyclotomic function fields. The equivalence to $F_q(n,t_1,\ldots,t_l)$ then implies that these counts are related to the number of points on curves as well.

The prescribed traces problem that we partially solve in this work is similar to a problem studied by Miers and Ruskey~\cite{miersruskey1}. 
In particular, using combinatorial techniques, expressions were given for the number of strings over finite rings with prescribed symmetric 
function evaluations. However, while this problem is similar the techniques used do not seem applicable to the prescribed traces problem as we 
have defined it. More relevant to our case, but still seemingly inapplicable are the expressions given by Miers and Ruskey for the number of equivalence 
classes of aperiodic strings over finite rings with prescribed symmetric function evaluations, under rotation~\cite{miersruskey2}. However, 
as already mentioned the formulae for the equivalence between the prescribed coefficients problem and the prescribed traces problem that we study 
follow easily from this work.

A fundamental limitation of our proposed algorithm is that it breaks for $l \ge p$ due to the failure of Newton's identities in positive characteristic.
Nevertheless, by employing a slightly different but likely limited technique to bypass the $l < p$ constraint, we computed curves for $q = 2$, $l \le 7$
and $n$ odd whose number of $\F_{2^n}$-rational affine points solve the $F_2(n,t_1,\ldots,t_l)$ counting problems.
In fact, the ingredients of our main algorithm were initially developed for $q = 2$, since even the $l = 4$ 
cases were open problems with the only previous result being an approximation to the counts~\cite{Koma,Komathesis} (there is however a small 
error in the transforms used, see~\S\ref{sec:transformPCPPTP}). 
We used the computational algebra system Magma~\cite{magma} to compute the corresponding zeta functions for $l = 4$ and $l = 5$, obtaining explicit formulae for these open problems for $n$ odd. Interestingly, for these cases the characteristic polynomial of the Frobenius 
endomorphism of the featured curves have factors that arise from ordinary -- rather than supersingular -- abelian varieties, and a simple argument 
implies that the set of formulae $\{ F_2(n,t_1,t_2,t_3,t_4) \}_{n \ge 4}$ and $\{ F_2(n,t_1,t_2,t_3,t_4,t_5) \}_{n \ge 5}$ for each
$(t_1,\ldots,t_4)$ and $(t_1,\ldots,t_5)$ respectively, can not be periodic in $n$, as in the $l = 2$ and $l = 3$ cases. 
This perhaps explains why there had been little progress in the four coefficients problem for the past $15$ years, since there 
is not a finite set of cases to enumerate.
We also give formulae for a subset of $F_2(n,t_1,t_2,t_3,t_4)$ and $F_2(n,t_1,t_2,t_3,t_4,t_5)$ which are valid {\em for all}
$n \ge 4$ and $n \ge 5$, respectively. Somewhat surprisingly, the formulae for $F_2(n,0,0,0,*,t_5)$ for $n \ge 5$ -- where the asterisk 
means that we do not mind what the value of the fourth trace function is -- are periodic in $n$ with period $120$. It is therefore conceivable 
that these formulae could have been discovered without using our curve-based approach, although no doubt far less efficiently. 
The method for $q = 2$ extends to one for any $q = 2^r$ and $l \le 7$, and also extends to $q = 3^r$ and $l \le 3$, which features
non-supersingular Weil numbers for just three prescribed coefficients when $r=1$.

The sequel is organised as follows. 
In~\S\ref{sec:generalalg} we present our main algorithm for arbitrary $q = p^r$, $l < p$ and $n$ coprime to $p$, and
present details of a very efficient proof-of-concept demonstration for $q=5$ and $l=4$.
In~\S\ref{sec:char2} we present direct and indirect methods for solving the prescribed traces problem for $q = 2$ and $n$ odd and discuss
its possible limitations.
Then in~\S\ref{sec:binaryzetas} we apply these methods to compute curves for $l \le 7$, some for $n$ odd and some for all relevant $n$, provide explicit formulae for $l = 4$ and $l = 5$, and for $l=4$ detail a connection to binary Kloosterman sums and provide the correct transform from
the prescribed traces problems to the prescribed coefficients problem.
In~\S\ref{sec:ternaryzetas} we present some formulae for the $q = l = 3$ case. 
Finally, in~\S\ref{sec:summary} we discuss some of the computational challenges and theoretical questions arising from our
approach in the general case and propose some natural open problems.

We assume the reader is familiar with curves and their zeta functions. Should they not be, the relevant definitions 
may be found in~\S3 of the antecedent to this work~\cite{AGGMY}, which we recommend to the reader. 
The code used for all interesting computations performed with Magma and Maple~\cite{maple} 
is openly available from \url{https://github.com/robertgranger/CountingIrreducibles}, with the relevant files indicated in footnotes in the text.

%----------------------------------------------------------------------------------
%----------------------------------------------------------------------------------

\section{The Main Algorithm}\label{sec:generalalg}

In this section we present an algorithm for solving the prescribed traces problem -- and thus the prescribed coefficients problem -- exactly for any 
prime power $q = p^r$, any $1 \le l < p$ and any $n \ge l$ coprime to $p$. The input traces whose values are prescribed are $T_{1},\ldots,T_{l}$, 
\ie the problem is to compute
\begin{equation}\label{eq:mainalg1}
F_q(n,t_1,\ldots,t_l) = \#\{a \in \F_{q^n} \mid T_1(a) = t_1,\ldots, T_l(a) = t_l\},
\end{equation}
for any $t_1,\ldots,t_l \in \F_q$. We begin with what we refer to as the direct method in~\S\ref{sec:maindirect} and present the indirect, 
or batching method in~\S\ref{sec:mainindirect}.

\subsection{Direct method}\label{sec:maindirect}

We begin with the following extremely simple lemma, whose proof follows from~\cite[Theorem 2.25]{lidl}.

\begin{lemma}\label{lem:paramq}
\begin{enumerate}[label={(\arabic*)}]
\item For $a \in \F_{q^n}$ the condition $T_1(a) = 0$ is equivalent to $a = a_{1}^q - a_{1}$, for $q$ different $a_1 \in \F_{q^n}$.
\item For $a \in \F_{q^n}$ and $n \not\equiv 0 \pmod{p}$, the condition $T_1(a) = t_1$ is equivalent to $a = a_{1}^q - a_{1} + t_1/n $, 
for $q$ different $a_1 \in \F_{q^n}$.
\end{enumerate}
\end{lemma}
We now recall Newton's identities over $\Z$ (see \eg\cite[Theorem 1.75]{lidl}) with indeterminates $\alpha_1,\ldots,\alpha_n$. 
Abusing notation slightly, we refer to the elementary symmetric polynomials in $\alpha_1,\ldots,\alpha_n$ as 
$T_1(\alpha),T_2(\alpha),\ldots$, and to the power sum symmetric polynomials $\alpha_{1}^j + \cdots + \alpha_{n}^j$
as $T_1( \alpha^j)$ for $j \ge 1$, \ie we work in the ring of symmetric functions, suppressing the dependence on $n$.
We use the convention that $T_0(\alpha) = 1$.

\begin{lemma}
For all $k \ge 1$ and $n \ge k$ we have
\begin{equation}\label{eq:NI}
k\,T_k(\alpha) = \sum_{j=1}^k (-1)^{j-1} T_{k-j}(\alpha)T_1(\alpha^j).
\end{equation}
\end{lemma}
Letting $\alpha = a$ and substituting $T_j(a) = t_j$ for $1 \le j \le k$, Eq.~(\ref{eq:NI}) becomes
\begin{eqnarray}
\nonumber k t_k &=& \sum_{j=1}^k (-1)^{j-1} t_{k-j} T_1(a^j)\\
\label{eq:recurrencespecialised} &=& (-1)^{k-1} T_1(a^k) + \sum_{j=1}^{k-1}(-1)^{j-1}t_{k-j} T_1(a^j),
\end{eqnarray}
since by convention we have $T_0(\cdot)=1$ and so $t_0 = 1$. Eq.~(\ref{eq:recurrencespecialised}) allows one to express each $T_1(a^k)$ as a polynomial
in $t_1,\ldots,t_k$ only, which we refer to as $p_k(t_1,\ldots,t_k)$. In particular, $p_1(t_1) = t_1$, $p_2(t_1,t_2) = t_{1}^2 - 2t_2$, 
$p_3(t_1,t_2,t_3) = 3t_3 + t_{1}^3 - 3t_1t_2$, and in general we have
\begin{equation}\label{eq:def:p_k}
p_k(t_1,\ldots,t_k) = T_1(a^k) = (-1)^{k-1}\Big( kt_k - \sum_{j=1}^{k-1}(-1)^{j-1}t_{k-j}p_j(t_1,\ldots,t_j) \Big).
\end{equation}
Since $1 \le k \le l < p$, each $t_k$ features in the condition on $T_1(a^k)$ in Eq.~(\ref{eq:def:p_k}) and does so linearly.
Therefore for each $(t_{1}',\ldots,t_{l}') \in (\F_q)^l$ there is a unique $(t_1,\ldots,t_l) \in (\F_q)^l$ such that for all $1 \le k \le l$ we have 
$t_{k}' = p_k(t_1,\ldots,t_k)$, and vice versa. Eq.~(\ref{eq:mainalg1}) thus becomes
\begin{equation}\label{eq:newmain}
\# \{ a \in \F_{q^n} \mid T_{1}(a) = t_{1}', T_1(a^2) = t_{2}',\ldots,T_1(a^l) = t_{l}'\}.
\end{equation}
To evaluate Eq.~(\ref{eq:newmain}) with what we call the direct method, let $\overline{n} \in \{1,\ldots,p-1\}$ and suppose 
$n \equiv \overline{n} \pmod{p}$. Then by introducing variables $a_1,\ldots,a_l$ 
and repeatedly applying Lemma~\ref{lem:paramq}, Eq.~(\ref{eq:newmain}) equals
\begin{equation}\label{eq:main_direct}
\frac{1}{q^l} \# \{ (a,a_1,\ldots,a_l) \in (\F_{q^n})^{l+1} \mid a_{1}^q - a_{1} = a -  t_{1}'/\overline{n}, \ldots, 
a_{l}^q - a_{l} = a^l - t_{l}'/\overline{n} \}.
\end{equation}
Note that if $t_{1}' = \cdots = t_{l}' = 0$ then one need not introduce $\overline{n}$ at all, and the count~(\ref{eq:main_direct}) is valid for 
all $n \ge l$. Rather than work with this intersection one can alternatively define the curves 
$C_k/\F_{q}: a_{1}^q - a_{1} = a^k - t_{k}'/\overline{n}$ for $1 \le k \le l$ and consider their fibre product, \`a la~\cite[\S2]{vandegeer}. 
In~\S\ref{sec:mainindirect} we take a similar but more elementary approach, working only with an associated set of Artin-Schreier curves 
of much smaller genus, since there exists a practical algorithm for computing their zeta functions due to Lauder and Wan~\cite{LauderWan}. 
We do so by evaluating the count~(\ref{eq:newmain}) indirectly, which allows one to solve the prescribed traces problem for all $q^l$
such $F_q(n,t_1,\ldots,t_l)$ simultaneously.

%-------------------------------------------------------------------------------------------------

\subsection{A transform of the prescribed traces problem}\label{sec:transform1p}

In this subsection we transform the problem of counting the number of elements of $\F_{q^n}$ with prescribed traces to the 
problem of counting the number of elements for which linear combinations of the trace functions evaluate to $1$. 
The transform is more general than is required for the target problem of interest and may be applied to any number of functions from $\F_{q^n}$
to $\F_q$. Therefore let $m \ge 1$ be the number of such functions. We first fix some notation.

We require a bijection from the integers $\{0,\ldots,q^m-1\}$ to $(\F_{q})^m$, the image of an input $i$ being denoted by $\vec{i}$. 
One can for instance take the base-$q$ expansion of $i$ to give $i_{m-1}^{'}q^{m-1} + \cdots + i_{0}^{'}$ and then set 
$\vec{i} = (i_{m-1},\ldots,i_0) = (\tau(i_{m-1}^{'}),\ldots,\tau(i_{0}^{'}))$, where $\tau: \{0,\ldots,q-1\} \rightarrow \F_q$ is defined by 
fixing a degree $r$ monic irreducible $f \in \F_p[x]$ and a polynomial basis for $\F_{p^r}/\F_p$, and mapping the base-$p$ expansion of an 
integer in $\{0,\ldots,q-1\}$ to the polynomial with those coefficients. Note that according to this definition, $\vec{0} = (0,\ldots,0)$ is 
the all-zero vector in $(\F_{q})^m$.

Let $f_0,\ldots,f_{m-1}: \F_{q^n} \rightarrow \F_q$ be any functions and let $\vec{f} = (f_{m-1},\ldots,f_0)$.
For $\vec{i} = (i_{m-1},\ldots,i_0),\vec{j} = (j_{m-1},\ldots,j_0) \in (\F_{q})^m$ let $\vec{i} \cdot \vec{j}$ denote the usual inner product.
For any $\vec{i} \in (\F_q)^m$, let $\vec{i}\cdot \vec{f}$ denote the function 
\[
\sum_{k=0}^{m-1} i_k f_k: \F_{q^n} \rightarrow \F_q,
\]
and let $V_1(\vec{i}\cdot \vec{f})$ denote the number of elements of $\F_{q^n}$ for which $\vec{i} \cdot \vec{f}$ evaluates to $1$. We define 
$V_1(\vec{0} \cdot \vec{f})$ to be $q^n$. Finally, let $N(\vec{j}) = N(j_{m-1},\ldots,j_{0})$ denote the number of $a \in \F_{q^n}$ such that 
$f_k(a) = j_k$, for $k = 0,\ldots,m-1$. 

%The reason we normalise to $1$ and avoid counting zeros is so that we avoid repeated relations (up to scalar multiples in $\F_{q}^{\times}$), 
%which would mean that one could not express the $q^m$ $N(\vec{j})$'s in terms of the $q^m$ $V_1(\vec{i} \cdot \vec{f})$'s by a 
%linear transformation. 
Our goal is to express any $N(\vec{j})$ in terms of the $V_1(\vec{i}\cdot \vec{f})$, but we begin by first solving the inverse problem, 
\ie expressing any $V_1(\vec{i}\cdot \vec{f})$ in terms of the $N(\vec{j})$.

\begin{lemma}
With the notation as above, for $\vec{i} \in (\F_q)^m \setminus \{\vec{0}\}$ we have
\begin{equation}\label{eq:basicrelationcharp}
V_1(\vec{i}\cdot \vec{f}) = \sum_{\vec{i} \cdot \vec{j} \, = \, 1} N(\vec{j}).
\end{equation}
\end{lemma}
\begin{proof}
By definition, we have $V_1(\vec{i}\cdot \vec{f}) = \#\{ a \in \F_{q^n} \mid \vec{i}\cdot \vec{f} (a) = 1\} = 
\#\{ a \in \F_{q^n} \mid \sum_{k=0}^{m-1} i_k f_k(a) = 1\}$.
Since $N(\vec{j})$ counts precisely those $a \in \F_{q^n}$ such that $f_k(a) = j_k$, we must count over all those $\vec{j}$ for which
$\sum_{k=0}^{m-1} i_k f_k(a) = 1$, \ie those such that $\sum_{k=0}^{m-1} i_k j_k = 1$. \qed
\end{proof}
Writing Eq.~(\ref{eq:basicrelationcharp}) in matrix form, for $i,j \in \{0,\ldots,q^m-1\}$ we have
\[ 
\begin{bmatrix}
V_1(\vec{i}\cdot\vec{f})
\end{bmatrix}^T
=
S_{q,m}
\cdot
\begin{bmatrix}
N(\vec{j})
\end{bmatrix}^T,
\] 
where 
\begin{equation*}\label{eq:spm}
(S_{q,m})_{i,j} = \begin{cases}
\begin{array}{ll}
1 &  \ \text{if}  \ \vec{i}\cdot\vec{j} = 1 \ \text{or if} \ \vec{i} = \vec{0} \\
%1 &  \ \text{if}  \ i \ne 0 \ \text{and} \ \\
0 &  \ \text{otherwise}
\end{array}.
\end{cases}
\end{equation*}
We have the following lemma.

\begin{lemma}\label{lem:spm1}
For all prime powers $q = p^r$ and $m \ge 1$, the $q^m \times q^m$ matrix $S_{q,m}$ is invertible over $\Q$.
\end{lemma}
\begin{proof}
Indexing the rows and columns by $i$ and $j$ for $0 \le i,j \le q^m - 1$, the $0$-th row of $S_{q,m}$ consists of $1$'s only, 
while besides the initial $1$, the $0$-th column consists of $0$'s only. Therefore no $\Q$-linear combination of rows $1$ to $q^m-1$ can cancel the $1$
in position $(0,0)$. Hence if one shows that the submatrix 
\begin{equation*}
S_{1 \le i,j \le q^m - 1} = \begin{cases}
\begin{array}{ll}
1 &  \ \text{if}  \ \vec{i}\cdot\vec{j} = 1\\
0 &  \ \text{otherwise}
\end{array}
\end{cases}
\end{equation*}
of $S_{q,m}$ is invertible then we are done, since this implies that $S_{q,m}$ has full rank. We claim that the inverse of $S$ is 
\begin{equation*}\label{eq:spm}
S_{1 \le i,j \le q^m - 1}^{\text{inv}} = \frac{1}{q^{m-1}} \begin{cases}
\begin{array}{rl}
1 &   \ \text{if}  \ \vec{i}\cdot\vec{j} = 1\\
-1 &  \ \text{if}  \ \vec{i}\cdot\vec{j} = 0\\
0 &  \ \text{otherwise}
\end{array}
\end{cases}.
\end{equation*}
Let $R = S \cdot S^{\text{inv}}$. Then for $1 \le i,j \le q^m-1$ one has 
\begin{equation}\label{eq:rij}
R_{i,j} = \sum_{k = 1}^{q^m - 1} S_{i,k} \cdot S_{k,j}^{\text{inv}} = 
\sum_{\substack{k = 1,\\
\vec{i}\cdot\vec{k} = 1, \ \vec{j}\cdot\vec{k} = 1}}^{q^m - 1} 1 
\ -\sum_{\substack{k = 1,\\
\vec{i}\cdot\vec{k} = 1, \ \vec{j}\cdot\vec{k} = 0}}^{q^m - 1} 1 \ .
\end{equation}
If $i = j$ then the second term of the r.h.s. of Eq.~(\ref{eq:rij}) is zero, while the first is $q^{m-1}$, since if one chooses a non-zero 
component $i_l$ of $\vec{i}$ (of which there is at least one as $i \ne 0$), then one can freely choose the $m-1$ coefficients
of $\vec{k}$ other than $k_l$, while the condition $\vec{i}\cdot\vec{k} = 1$ entails that
\[
k_l = (1 - \sum_{\substack{w = 0,\\w \ne l}}^{m-1} i_w k_w)/i_l,
\]
which is well-defined because $i_l$ is invertible.

Now assume $i \ne j$. If possible choose $l,l'$ such that $l \ne l'$ and $i_l \ne 0$ and $j_{l'} \ne 0$.
Then considering the set of all $k$ for which $\vec{i}\cdot\vec{k} = 1$ as described above, as $k_{l'}$ varies over $\F_q$, 
so does $k_{l'} j_{l'}$. Hence both terms of the r.h.s. of Eq.~(\ref{eq:rij}) are precisely $q^{m-2}$, since there are $m-2$ free components of
$\vec{k}$. For the remaining case where $\vec{i}$ and $\vec{j}$ both have only one non-zero component, in position $l$ say, then both terms  
of the r.h.s. of Eq.~(\ref{eq:rij}) are zero, since for the first there is no $k_l$ for which $i_l k_l = 1$ and $j_l k_l = 1$ since 
$i_l \ne j_l$, while for the second the condition $\vec{j}\cdot\vec{k} = 0$ implies $k_l = 0$, in which case $i_l k_l = 1$ can not hold.
Therefore, $R$ is the $q^m-1 \times q^m-1$ identity matrix.\qed
\end{proof}

To compute $N(\vec{0})$, we claim that 
\begin{equation}\label{eq:N0}
N(\vec{0}) = V_1( \vec{0}\cdot\vec{f}) - \frac{1}{q^{m-1}}\sum_{i=1}^{q^m-1} V_1( \vec{i}\cdot\vec{f} ).
\end{equation}
Since $V_1( \vec{0}\cdot\vec{f}) = q^n = \sum_{j = 0}^{q^m-1} N(\vec{j})$ by definition, Eq.~(\ref{eq:N0}) is equivalent to
\begin{equation}\label{eq:N01}
\sum_{j = 1}^{q^m-1} N(\vec{j}) = \frac{1}{q^{m-1}}\sum_{i=1}^{q^m-1} V_1( \vec{i}\cdot\vec{f} ) = 
\frac{1}{q^{m-1}} \sum_{i=1}^{q^m-1} \sum_{\vec{i} \cdot \vec{j} \, = \, 1} N(\vec{j}),
\end{equation}
with the latter equality given by Eq.~(\ref{eq:basicrelationcharp}). Considering the r.h.s. of Eq.~(\ref{eq:N01}), the number of occurrences 
of $N(\vec{j})$ is $q^{m-1}$ since as we argued previously if one fixes an $l$ for which $j_l \ne 0$, then one can choose the components of 
$\vec{i}$ other than $i_l$ freely, while $i_l$ is then fixed by the condition $\vec{i}\cdot\vec{j} = 1$.
We have therefore proven the following:
\begin{equation}\label{eq:Sminus1}
S_{q,m}^{-1} = \frac{1}{q^{m-1}}
\renewcommand\arraystretch{1.3}
\mleft[
\begin{array}{c|ccc}
  q^{m-1} & -1 & \cdots & - 1 \\
  \hline
  0 & & & \\
  \vdots & & S_{1 \le i,j \le q^m - 1}^{inv} =\begin{cases}
\begin{array}{rl}
1 &   \ \text{if}  \ \vec{i}\cdot\vec{j} = 1\\
-1 &  \ \text{if}  \ \vec{i}\cdot\vec{j} = 0\\
0 &  \ \text{otherwise}
\end{array}
\end{cases} & \\
  0 & & & \\
\end{array}
\mright]
%\begin{bmatrix*}[r]
%q^{m-1} & -1 & -1 & \cdots & 1\\
%0 & -1 & 1 & -1 & 1 & -1 & 1 & -1\\
%\vdots
%\end{bmatrix*}
\end{equation}
Thus in order to compute any of the $q^m$ possible outputs $N(\vec{j})$ of any set of $m$ functions $\vec{f}$, it is sufficient to count the number 
of evaluations to $1$ of all the $q^m-1$ non-zero $\F_q$-linear combinations of the functions, and then apply $S_{q,m}^{-1}$. 
In particular, one may choose the $f_{0},\ldots,f_{m-1}$ to be any subset of the trace functions $T_1,\ldots,T_{n}$, or in our case of 
interest, the trace functions $T_1,\ldots,T_{l}$.

%---------------------------------------------------------------------------

\subsection{Indirect method}\label{sec:mainindirect}

As per the transform of the previous subsection let $\vec{f} = (T_1(a^l),T_1(a^{l-1}),\ldots,T_1(a))$ and for 
$1 \le i \le q^{l} - 1$ let $\vec{i} = (i_{l-1},\ldots,i_0)$. Computing formulae for $V_1(\vec{i} \cdot \vec{f})$ for each 
such $i$ produces formulae for all $q^l$ possible counts in Eq.~(\ref{eq:newmain}), which then uniquely determine $F_q(n,t_1,\ldots,t_l)$.
For $\overline{n} \in \{1,\ldots,p-1\}$ and all $n \equiv \overline{n} \pmod{p}$, applying Lemma~\ref{lem:paramq} we have
\begin{eqnarray}
\nonumber V_1(\vec{i} \cdot \vec{f}) &=& \#\{ a \in \F_{q^n} \mid \sum_{k=0}^{l-1} i_k f_k(a) = 1\}\\
\label{eq:artin} &=& \frac{1}{q} \#\{ (a,a_1) \in (\F_{q^n})^2 \mid a_{1}^q - a_1 = -\frac{1}{\overline{n}} + \sum_{k=0}^{l-1} i_{k} a^{k+1} \big)\}.
\end{eqnarray}
In order to compute the zeta functions of the Artin-Schreier curves in~(\ref{eq:artin}) we use a $p$-adic point counting algorithm due to 
Lauder and Wan that is well suited to this purpose~\cite[Algorithm 25]{LauderWan}.
Let $f \in \F_q[a]$ be of degree $d$ not divisible by $p$, let $C_f/\F_q$ be the affine curve with equation $a_{1}^p - a_1 = f(a)$ and let 
$\tilde{C}_f$ denote the unique smooth projective curve that is birational to $C_f$. 
By Weil~\cite{weil}, we know that the zeta function $Z(\tilde{C}_f,T)$ of $\tilde{C}_f$ satisfies
\begin{equation}\label{zeta}
Z(\tilde{C}_f,T) = \frac{P(\tilde{C}_f,T)}{(1-T)(1-qT)},
\end{equation}
where $P$ is polynomial of degree $2g$ with $g = (p-1)(d-1)/2$ the genus of $\tilde{C}_f$, and which factorises as 
$\prod_{k=1}^{2g} (1 - \omega_k t)$ with $|\omega_k| = \sqrt{q}$. Since $\tilde{C}_f$ is birational to $C_f$ their
respective zeta functions differ only by a constant factor and thus to compute $Z(C_f,T)$ it suffices to compute $Z(\tilde{C}_f,T)$.
An exponential sum approach also allows one to deduce the form of $Z(C_f,T)$ directly~\cite[\S2.2]{LauderWan}.
We have the following.

\begin{theorem}\cite[Theorem 1]{LauderWan}
\label{thm:LW}
The zeta function of the smooth projective curve $\tilde{C}_f$ (and thus the affine curve $C_f$) may be computed deterministically in 
$\tilde{\mathcal{O}}(d^5 p^4 r^3)$ bit operations.
\end{theorem}
We note that one does not need to compute the unique smooth projective curve $\tilde{C}_f$ in order to apply the Lauder-Wan 
algorithm. 

Let the right hand side of the curves in~(\ref{eq:artin}) be denoted by $\overline{p}_{\vec{i},\overline{n}}(a)$ and denote each curve by
\begin{equation}\label{eq:maincurve}
C_{\vec{i},\overline{n}}/\F_{q}: a_{1}^q - a_{1} = \overline{p}_{\vec{i},\overline{n}}(a).
\end{equation}
The degree of $\overline{p}_{\vec{i},\overline{n}}(a)$ is 
$k'+1$ where $k'$ is the largest $0 \le k' \le l-1$ such that $i_{k'} \neq 0$. Since $k'+1 \le l < p$, for all $1 \le i \le q^{l}-1$ 
the degree of $\overline{p}_{\vec{i},\overline{n}}(a)$ is less than $p$ and thus not divisible by $p$, and so this precondition
of the Lauder-Wan algorithm is satisfied.
However, the Lauder-Wan algorithm can only be applied when the left hand side of~(\ref{eq:maincurve}) is $a_{1}^p - a_1$; we therefore need to 
reduce to this case. This issue was already addressed in~\cite[Corollary 1]{AGGMY} for the case $p = 2$ (albeit with a sign error) using a theorem of 
Kani-Rosen~\cite{kanirosen} and by studying quotients arising from the action of automorphisms of the left hand side only. The correct generalisation 
is as follows, where for an affine curve $C/\F_q$ we denote by $\#C(\F_{q^n})$ the number of $\F_{q^n}$-rational affine points of $C$.

\begin{lemma}\label{thm:extensionreduction}
Let $p \ge 2$ be prime, for $r \ge 1$ let $q = p^r$, let $f \in \F_q[x]$, let $C/\F_q$ be the Artin-Schreier curve $y^q - y = f(x)$ and for 
$\alpha \in \F_{q}^{\times}$ let $C_{\alpha}/\F_q$ be the Artin-Schreier curve $y^p - y = \alpha f(x)$. Then for all $n \ge 1$ we have
\begin{equation}\label{eqn:extensionreduction}
(p-1)\#C(\F_{q^n}) - \sum_{\alpha \in \F_{q}^{\times}} \#C_{\alpha}(\F_{q^n}) = (p - q)q^n.
\end{equation}
\end{lemma}
Thus computing the zeta function of $C/\F_q$ reduces to computing the zeta functions of the $q-1$ curves $C_{\alpha}/\F_q$.
Lemma~\ref{thm:extensionreduction} may be proven using an exponential sum argument and such a proof can be found (in Persian) in~\cite{omranproof}. 
For completeness we include it here and thank Omran Ahmadi for communicating to us its translation. 

\begin{proof}
For $m,r \ge 1$ positive integers let $\text{Tr}_{m}:\F_{p^m} \rightarrow \F_p: x \mapsto x + x^p + x^{p^2} + \cdots + x^{p^{m-1}}$ denote the absolute 
trace function, and let $\text{Tr}_{rm|r}:\F_{p^{rm}} \rightarrow \F_{p^r}: x \mapsto x + x^{p^r} + x^{p^{2r}} + \cdots + x^{p^{(m-1)r}}$ denote
the trace. By Lemma~\ref{lem:paramq}~(1) if $\text{Tr}_{rn}(\alpha f(a))=0$, then there are $p$ points on $C_\alpha$ with $x$-coordinates equal to $a$, 
and otherwise none since $\text{Tr}_{rn}(y^p - y) = 0$ for all $y \in \F_{p^{rn}}$. Therefore the number of points on $C_{\alpha}$ with 
$x$-coordinate equal to $a$ is $\sum \psi( \text{Tr}_{rn}(\alpha f(a)))$
where the sum is over the additive characters of $\F_p$. Denoting the set of all additive characters of $\F_p$ by $\Psi$ and its trivial character 
by $\psi_0$ we have 
\[
\#C_\alpha(\F_{p^{rn}})= \sum_{a \in \F_{p^{rn}}} \sum_{\psi \in \Psi} \psi( \text{Tr}_{rn}(\alpha f(a))) = p^{rn} + 
\sum_{a \in \F_{p^{rn}}} \sum_{\psi \neq \psi_0} \psi( \text{Tr}_{rn}(\alpha f(a))).
\]
Summing over all $\alpha \in \F_{p^r}^{\times}$ we obtain
\begin{eqnarray}
\nonumber S = \sum_{\alpha \in \F_{p^r}^{\times}} \#C_\alpha(\F_{p^{rn}}) &=& p^{rn}(p^r-1) + 
\sum_{\alpha \in \F_{p^r}^{\times}} \sum_{a \in \F_{p^{rn}}} \sum_{\psi \neq \psi_0} \psi( \text{Tr}_{rn}(\alpha f(a)))\\
\label{Total-Sum}  &=& p^{rn}(p^r-1) + 
 \sum_{a \in \F_{p^{rn}}} \sum_{\alpha \in \F_{p^r}^{\times}} \sum_{\psi \neq \psi_0} \psi( \text{Tr}_{rn}(\alpha f(a))).
\end{eqnarray}
Since $\alpha \in \F_{p^r}$, we have
\[
\text{Tr}_{rn}(\alpha f(a))=\text{Tr}_{r}(\text{Tr}_{rn|r}(\alpha f(a))) = \text{Tr}_{r}(\alpha\text{Tr}_{rn|r}(f(a))).
\]
This implies that if $\text{Tr}_{rn|r}(f(a)) = 0$ for an $a \in \F_{p^{rn}}$, then
\begin{equation}\label{eq-T}
\sum_{\alpha \in \F_{p^r}^{\times}} \sum_{\psi \neq \psi_0} \psi( \text{Tr}_{rn}(\alpha f(a))) =
\sum_{\alpha \in \F_{p^r}^{\times}} \sum_{\psi \neq \psi_0} \psi( \text{Tr}_{r}(\alpha\text{Tr}_{rn|r}(f(a)))) = (p-1)(p^r - 1),
\end{equation}
and otherwise 
\begin{equation}\label{eq-U}
\sum_{\alpha \in \F_{p^r}^{\times}} \sum_{\psi \neq \psi_0} \psi( \text{Tr}_{rn}(\alpha f(a))) =
\sum_{\alpha \in \F_{p^r}^{\times}} \sum_{\psi \neq \psi_0} \psi( \text{Tr}_{r}(\alpha\text{Tr}_{rn|r}(f(a)))) = 1 - p.
\end{equation}
If we let $$T=\#\left\{a \in |\text{Tr}_{rn|r}(f(a)) = 0 \right\}$$ and $$U=\#\left\{a \in |\text{Tr}_{rn|r}(f(a)) \neq 0\right\},$$ 
then $T+U=p^{rn}$. Using this fact and combining Equations~\eqref{Total-Sum},~\eqref{eq-T} and ~\eqref{eq-U}, we obtain
\begin{eqnarray}
\nonumber S &=& p^{rn}(p^r-1) + (p^r-1)(p-1)T + (1-p)U\\ 
\nonumber &=& p^{rn}(p^r-1) + (p^r-1)(p-1)T + (1-p)(p^{rn} - T)\\
\label{eq:S} &=& (p-1)p^r T + p^{rn}(p^r - p)
\end{eqnarray}
Applying Lemma~\ref{lem:paramq}~(1) again we have
\[
\#C(\F_{p^{rn}})=p^r T,
\] 
and hence by Eq.~(\ref{eq:S}) we have
\[
(p-1)\#C(\F_{p^{rn}}) - \sum_{\alpha \in \F_{p^r}^{\times}} \#C_\alpha(\F_{p^{rn}}) = (p - p^r)p^{rn}. 
\]
\flushright \qed
\end{proof}

\noindent {\it Proof of Theorem~\ref{thm:maintheorem}:}
Fix $\overline{n} \in \{1,\ldots,p-1\}$. By transform~(\ref{eq:Sminus1}), in order to compute all $F_q(n,t_1,\ldots,t_l)$  
one need only compute $V_1(\vec{i} \cdot \vec{f})$ for all $1 \le i \le q^{l}-1$.
For each such $i$ let $d(\vec{i})$ denote the degree of $\overline{p}_{\vec{i},\overline{n}}(a)$, 
which we recall is $(k'+1)$ where $k'$ is the largest $0 \le k' \le l-1$ such that $i_{k'} \neq 0$. 
The number of $i$'s for which $d(\vec{i})$ is $l,l-1,\ldots,1$ is therefore $(q-1)q^{l-1},(q-1)q^{l-2},\ldots,q-1$, respectively.

We first count the number of roots -- counted with multiplicities, since they may not all be distinct -- which appear in the numerator of the zeta function of the curve $C_{\vec{i},\overline{n}}/\F_q$ featured in $V_1(\vec{i} \cdot \vec{f})$.
For $q = p$ the genus of $C_{\vec{i},\overline{n}}/\F_p$ is $(p-1)(d(\vec{i})-1)/2$ and so there are $(p-1)(d(\vec{i})-1)$ roots. 
The total number of roots over all $1 \le i \le p^l-1$ is therefore
\begin{equation*}\label{eq:numberofroots}
(p-1)\sum_{i = 1}^{p^l-1} (d(\vec{i}) - 1) = (p-1) \sum_{d=1}^{l} (p-1) p^{d-1} (d-1) = (p-1)^2 \sum_{d=1}^{l-1} dp^d = p^l(pl - p - l) + p.
\end{equation*}
For $q = p^r$ with $r > 1$, as per Lemma~\ref{thm:extensionreduction} we adapt the curves in Eq.~(\ref{eq:maincurve}), so for $\alpha \in \F_{q}^{\times}$ let 
\begin{equation}\label{eq:Calpha}
C_{\alpha,\vec{i},\overline{n}}/\F_{q}: a_{1}^p - a_{1} = \alpha\, \overline{p}_{\vec{i},\overline{n}}(a).
\end{equation}
Since the number of roots of $C_{\vec{i},\overline{n}}/\F_q$ is the sum over $\alpha \in \F_{q}^{\times}$ of the number of roots of 
$C_{\alpha,\vec{i},\overline{n}}/\F_{q}$, the total number of roots is therefore
\begin{eqnarray*}\label{eq:numberofrootsq}
\nonumber (q-1)(p-1)\sum_{i = 1}^{q^l-1} (d(\vec{i}) - 1) &=& (q-1)(p-1) \sum_{d=1}^{l} (q-1) q^{d-1} (d-1) = (q-1)^2(p-1) \sum_{d=1}^{l-1} dq^d\\
&=& (p-1)(q^l(ql - q - l) + q).
\end{eqnarray*}
Thus $N$ is as stated in the theorem and the (not necessarily distinct) roots $\omega_1,\ldots,\omega_N \in \overline{\Z}$ all have norm $\sqrt{q}$.
The $\upsilon_1,\ldots,\upsilon_N$ corresponding to each $\vec{t} = (t_1,\ldots,t_l)$ follow immediately from transform~(\ref{eq:Sminus1}).
In particular, for $q=p$ we have 
\begin{equation*}
F_p(n,\vec{0}) = \frac{1}{p^{l-1}} \Big( p^{n + l - 1} - \sum_{i=1}^{p^l-1} V_1(\vec{i} \cdot \vec{f}) \Big)
= p^n - \frac{1}{p^l} \sum_{i=1}^{p^l-1} \Big( p^n - \sum_{k=1}^{(p-1)(d(\vec{i}) - 1)} \omega_{i,k}^n \Big)
= \frac{1}{p^l} \Big(p^n + \sum_{k=1}^{N} \omega_{k}^n \Big),
\end{equation*}
while for $\vec{t} \ne \vec{0}$ we have
\begin{equation*}
F_p(n,\vec{t}) = \frac{1}{p^{l-1}} \Big( \sum_{\vec{i} \cdot \vec{t} = 1} V_1(\vec{i} \cdot \vec{f}) 
- \sum_{\vec{i} \cdot \vec{t} = 0, \ \vec{i} \ne \vec{0}} V_1(\vec{i} \cdot \vec{f})\Big)
=  \frac{1}{p^{l}} \Big( p^n + \sum_{k=1}^N \upsilon_k \omega_{k}^n \Big),
\end{equation*}
where the $\upsilon_{k}$ are in $\{0,\pm 1 \}$ and the $p^n$ arises as there are $p^{l-1}$ indices $i$ for which $\vec{i} \cdot \vec{t} = 1$, 
while there are $p^{l-1}-1$ non-zero indices $i$ for which $\vec{i} \cdot \vec{t} = 0$.
For $q = p^r$ with $r > 1$, we have 
\[
\#C_{\alpha,\vec{i},\overline{n}}(\F_{q^n}) = q^n - \sum_{k=1}^{(p-1)(d(\vec{i})-1)} \omega_{\alpha,i,k}^n,
\]
and hence
\[
\sum_{\alpha \in \F_{q}^{\times}} \#C_{\alpha,\vec{i},\overline{n}}(\F_{q^n}) = (q-1)q^n - \sum_{k=1}^{(q-1)(p-1)(d(\vec{i})-1)} \omega_{i,k}^n,
\]
and by Lemma~\ref{thm:extensionreduction} we have
\begin{eqnarray*}
\#C_{\vec{i},\overline{n}}(\F_{q^n}) &=& \frac{p-q}{p-1}\,q^n + \frac{q-1}{p-1}\Big( q^n - \sum_{k=1}^{(q-1)(p-1)(d(\vec{i})-1)} \omega_{i,k}^n \Big)\\
&=& q^n - \frac{q-1}{p-1} \sum_{k=1}^{(q-1)(p-1)(d(\vec{i})-1)} \omega_{i,k}^n,
\end{eqnarray*}
and
\[
V_1(\vec{i} \cdot \vec{f}) = \frac{1}{q} \Big( q^n - \frac{q-1}{p-1} \sum_{k=1}^{(q-1)(p-1)(d(\vec{i})-1)} \omega_{i,k}^n \Big).
\]
As before, by~(\ref{eq:Sminus1}) we have
\begin{eqnarray*}
F_q(n,\vec{0}) = \frac{1}{q^{l-1}} \Big( q^{n + l - 1} - \sum_{i=1}^{q^l-1} V_1(\vec{i} \cdot \vec{f}) \Big)
&=& q^n - \frac{1}{q^l} \sum_{i=1}^{q^l-1} \Big( q^n - \frac{q-1}{p-1} \sum_{k=1}^{(q-1)(p-1)(d(\vec{i}) - 1)} \omega_{i,k}^n \Big)\\
&=& \frac{1}{q^l} \Big(q^n + \frac{q-1}{p-1}\sum_{k=1}^{N} \omega_{k}^n \Big),
\end{eqnarray*}
while for $\vec{t} \ne \vec{0}$ we have
\begin{equation*}
F_q(n,\vec{t}) = \frac{1}{q^{l-1}} \Big( \sum_{\vec{i} \cdot \vec{t} = 1} V_1(\vec{i} \cdot \vec{f}) 
- \sum_{\vec{i} \cdot \vec{t} = 0, \ \vec{i} \ne \vec{0}} V_1(\vec{i} \cdot \vec{f})\Big)
=  \frac{1}{q^{l}} \Big( q^n + \frac{q-1}{p-1}\sum_{k=1}^N \upsilon_k \omega_{k}^n \Big).
\end{equation*}
In all cases, for fixed $p$ and $q$, as $n \rightarrow \infty$ with $n \equiv \overline{n} \pmod{p}$ we have
\begin{equation*}
F_q(n,t_1,\ldots,t_l) = q^{n-l} + \mathcal{O}(q^{n/2}).
\end{equation*}

Regarding the complexity claim, for $q = p$, by Theorem~\ref{thm:LW} the cost of computing all the relevant zeta functions in terms of the
number of bit operations is
\begin{equation*}
\sum_{i = 1}^{p^l - 1} \tilde{\mathcal{O}}(d(\vec{i})^5 p^4) =
\sum_{j=1}^{l} (p-1)p^{j-1} \cdot \tilde{\mathcal{O}}(j^5 p^4 ) = p^l \cdot \tilde{\mathcal{O}}(l^5 p^4).
\end{equation*}
For $q = p^r$ with $r > 1$ we employ Lemma~\ref{thm:extensionreduction} and similarly obtain a cost in terms of the number of bit operations of
\begin{equation*}
(q-1) \sum_{i = 1}^{q^l - 1} \tilde{\mathcal{O}}(d(\vec{i})^5 p^4 r^3) =
(q-1) \sum_{j=1}^{l} (q-1)q^{j-1} \cdot \tilde{\mathcal{O}}(j^5 p^4 r^3 ) = q^{l+1} \cdot \tilde{\mathcal{O}}( l^5 p^4 r^3 ).
\end{equation*} \qed

\subsubsection{Some remarks on the algorithm.}
Firstly, note that the `cost per $F_q(n,t_1,\ldots,t_l)$' is just $\tilde{\mathcal{O}}( l^5 p^4 )$ and $q \cdot \tilde{\mathcal{O}}( l^5 p^4 r^3)$ bit operations for $r = 1$ and $r > 1$, respectively. The algorithm is thus very efficient at computing all such $F_q(n,t_1,\ldots,t_l)$.
Regarding practical efficiency, it should be possible to re-use much of the computation in the $q^l-1$ runs of the Lauder-Wan algorithm. 
It may also be possible for $r > 1$ to work with the curve~(\ref{eq:maincurve}) rather than the curves~(\ref{eq:Calpha}), 
thus bypassing Lemma~\ref{thm:extensionreduction} and possibly reducing the complexity by a factor of $q-1$ in this case. 
Furthermore, if only one $F_q(n,t_1,\ldots,t_l)$ is required then there may be a way to exploit the direct method which is more 
efficient than computing either $2q^{l-1} - 1$ or $q^l - 1$ zeta functions, as is required by the indirect method.

Secondly, as already noted there may be many repeats amongst the $\omega_k$'s and so the number of summands in~(\ref{eq:maintheorem}) may be far 
less than $N$, even for $F_q(n,\vec{0})$ to which they all contribute. Indeed, using the direct approach leads to formulae with less redundancy 
than the indirect method.
For example, for $q = 5$ and $l = 4$, for the direct method the curve~(\ref{eq:main_direct}) is absolutely irreducible and has genus $860$, producing 
$1720$ roots (not necessarily distinct), whereas the indirect method produces $6880$ roots for $F_q(n,\vec{0})$ (not necessarily distinct)
which is a factor of $p-1$ more.
The advantage of the indirect method is that the genus of the arising curves is far smaller ($0,2,4$ and $6$ in this example), 
making the computation of their zeta functions extremely efficient. One may be able to eliminate the redundancy: we provide an example of this 
in~\S\ref{sec:ternaryzetas}. 

Finally, note that in contrast to the direct method, for the indirect method even if $\vec{t} = \vec{0}$, 
one needs to use $1/\overline{n}$ as it features in~(\ref{eq:artin}), so the obtained formulae are only valid for $n$ coprime to $p$.

\subsection{Reducing the prescribed coefficients problem to the prescribed traces problem}\label{sec:transform0}

We now show how to express our titular problem in terms of the prescribed traces problem. For the present case of interest for which $l < p$, 
the transform follows almost immediately from the work of Miers and Ruskey~\cite{miersruskey2}; only a small change in the relevant 
definitions is needed.

For $a \in \F_{q^n}$ denote by $\overline{a}$ the string $(a,a^q,a^{q^2},\ldots,a^{q^{n-1}})$. Such a string is called periodic if there is a substring whose repeated concatenation 
gives $\overline{a}$; otherwise the string is called aperiodic. Observe that $\overline{a}$ is aperiodic if and only if $a$ does not belong to any proper
subfield of $\F_{q^n}$.
Let $A_q(n,t_1,\ldots,t_l)$ denote the number of aperiodic strings corresponding to elements of $\F_{q^n}$ in the above manner for which 
$T_1(a) = t_1,\ldots,T_l(a) = t_l$. If $\overline{a}$ is counted by $A_q(n,t_1,\ldots,t_l)$ then all rotations of $\overline{a}$ are distinct and 
produce the same trace values. By this and the above observation we thus have
\[
I_q(n,t_1,\ldots,t_l) = \frac{1}{n}A_q(n,t_1,\ldots,t_l).
\]

Now let $f$ denote the minimum polynomial of $a$ over $\F_q$, which has degree $n/d$ for some $d \mid n$.
Note that $T_k(a)$ is the coefficient of $x^{n-k}$ in $f^d$~\cite[Lemma 2]{cattell}, so abusing notation slightly we also write $T_k(a)$ as 
$T_k(f^d)$. The multinomial theorem (cf.~\cite[Lemma 2.1]{miersruskey2}) gives the following.

\begin{lemma}\label{lemma:multinomial}
For all $k \ge 1$ and $d \ge 1$ we have
\begin{equation*}\label{eq:multinomial}
T_k(f^d) = \sum_{\nu_1 + 2\nu_2 + \cdots + k\nu_k = k} \binom{d}{\nu_1,\ldots,\nu_k,d - (\nu_1 + \cdots + \nu_k)} T_1(f)^{\nu_1} T_2(f)^{\nu_2}\cdots T_k(f)^{\nu_k}.
\end{equation*}
\end{lemma}
This motivates the following definition. For a positive integer $d$ and $\vec{t} = (t_1,\ldots,t_l) \in (\F_q)^l$ define the map 
$\theta_d: (\F_q)^l \rightarrow (\F_q)^l: \vec{t} \rightarrow \vec{u} = (u_1,\ldots,u_l)$ by
\begin{eqnarray*}\label{eq:theta}
u_k &=& \sum_{\nu_1 + 2\nu_2 + \cdots + k\nu_k = k} \binom{d}{\nu_1,\ldots,\nu_k,d - (\nu_1 + \cdots + \nu_k)} t_{1}^{\nu_1} t_{2}^{\nu_2}\cdots t_{k}^{\nu_k}\\
&=&  \sum_{\nu_1 + 2\nu_2 + \cdots + k\nu_k = k} d^{(\nu_1 + \nu_2 + \cdots + \nu_k)} \frac{t_{1}^{\nu_1}}{\nu_1 !} \frac{t_{2}^{\nu_2}}{\nu_2 !}\cdots
 \frac{t_{k}^{\nu_k}}{\nu_k !},
\end{eqnarray*}
where $d^{(m)} = d(d-1)\cdots(d-m+1)$ is the falling factorial.

For a propostion $P$ let $[P]$ denote its truth value. Since every periodic string is the repeated concatenation of an aperiodic string, we have
(cf.~\cite[Eq. (2.5)]{miersruskey2}):
\[
F_q(n,\vec{u}) = \sum_{d \mid n} \sum_{\vec{t} \in (\F_{q})^l} [\theta_d(\vec{t}) = \vec{u}] A_q\big(\frac{n}{d},\vec{t}\big).
\]
For the present case of interest there is precisely one $\vec{t}$ for each $\vec{u}$. Indeed, we have the following theorem -- in which
$\mu$ denotes the usual M\"obius function -- whose proof is identical to that given by Miers and Ruskey~\cite[Theorem 6.1]{miersruskey2}.

\begin{theorem}
Let $q = p^r$ with $p$ an odd prime and let $\vec{t} = (t_1,\ldots,t_l)$ with $l < p$. Then
\begin{equation}\label{eq:transform0}
I_q(n,\vec{t}) = \begin{dcases} 
\frac{1}{n} \sum_{\substack{d \mid n\\ p \nmid d}} \mu(d) \Big(F_q\Big(\frac{n}{d},\vec{0}\Big) - [pd \mid n]q^{n/(pd)}\Big)   \ &\text{if} \ \vec{t} = \vec{0},\\
\frac{1}{n} \sum_{\substack{d \mid n\\ p \nmid d}} \mu(d)F_q\Big(\frac{n}{d},\theta_{d^{-1}}(\vec{t})\Big) \ &\text{otherwise}.
\end{dcases}
\end{equation}
\end{theorem}
First note that the inverse of $d$ is computed mod $p$. Also note that since our algorithm does not address the $n \equiv 0 \pmod{p}$ cases, 
the condition $p \nmid d$ in the summations is vacuous and the second term in each summand of the $\vec{t} = \vec{0}$ case is always zero. 
Therefore, for all $\vec{t}$ we have
\[
I_q(n,\vec{t}) = \frac{1}{n} \sum_{d \mid n} \mu(d)F_q\Big(\frac{n}{d},\theta_{d^{-1}}(\vec{t})\Big) = \frac{1}{n}\big(q^{n-l} + \mathcal{O}(q^{n/2})\Big).
\]
Finally, we note that for $q$ even only the first trace is prescribed, and this case was solved by Carlitz~\cite{Carlitz}.

\subsection{Example: $F_5(n,t_1,t_2,t_3,t_4)$}
As a proof-of-concept example, for each $\overline{n} \in \{1,2,3,4\}$ we computed the zeta functions of all $5^4 - 1$ 
curves~(\ref{eq:maincurve}) that are required by the indirect method in order to compute each $F_5(n,t_1,t_2,t_3,t_4)$ for all 
$n \equiv \overline{n} \pmod{5}$ with $n \ge 4$. The computations for each $\overline{n}$ took under a minute using Magma V22.2-3 
on a 2.0GHz AMD Opteron computer.
Note that Magma uses brute-force point counting over successive extension fields
in order to compute the zeta functions. For this small example in which the genera are $0,2,4$ or $6$ this is clearly not a concern.
For much larger $p$ the Lauder-Wan algorithm should be used. 

For the sake of space, in the following theorem we only give the formula for $F_5(n,0,0,0,0)$, which turns out to be the same for each 
$\overline{n} \in \{1,2,3,4\}$. As per the transform~(\ref{eq:Sminus1}) this count combines all $5^4-1$ zeta functions 
and thus has the greatest weight amongst all of the $F_5(n,t_1,t_2,t_3,t_4)$ formulae.
For a polynomial $\gamma(X) \in \Z[X]$ let $\rho_n(\gamma)$ denote the sum of the $n$-th 
powers of the (complex) roots of $\gamma(X)$. 

\begin{theorem}\label{thm:q5l4}
For $n \ge 4$ and $(n,5) = 1$ we have
\begin{eqnarray*}
F_5(n,0,0,0,0) &=& 5^{n-4} - \frac{1}{5^4}\big(
160\rho_n(\gamma_{2,1})
+16\rho_n(\gamma_{2,2})
+164\rho_n(\gamma_{4,1})
+16\rho_n(\gamma_{4,2})
+116\rho_n(\gamma_{4,3})
+25\rho_n(\gamma_{8,1})\\
&+&20\rho_n(\gamma_{8,2})
+16\rho_n(\gamma_{8,3})
+18\rho_n(\gamma_{8,4})
+20\rho_n(\gamma_{8,5})
+20\rho_n(\gamma_{8,6})
+20\rho_n(\gamma_{8,7})
+16\rho_n(\gamma_{8,8})\\
&+& 24\rho_n(\gamma_{8,9})
+16\rho_n(\gamma_{8,10})
+21\rho_n(\gamma_{8,11})
+17\rho_n(\gamma_{8,12})
+16\rho_n(\gamma_{8,13})
+16\rho_n(\gamma_{8,14})\\
&+& 12\rho_n(\gamma_{8,15})
+10\rho_n(\gamma_{8,16})
+8\rho_n(\gamma_{8,17})
+13\rho_n(\gamma_{8,18})
+20\rho_n(\gamma_{12,1})
+20\rho_n(\gamma_{12,2})\\
&+& 20\rho_n(\gamma_{12,3})
+16\rho_n(\gamma_{12,4})
+16\rho_n(\gamma_{12,5})
+16\rho_n(\gamma_{12,6})
+16\rho_n(\gamma_{12,7})
+16\rho_n(\gamma_{12,8})\\
&+& 16\rho_n(\gamma_{12,9})
+16\rho_n(\gamma_{12,10})
+16\rho_n(\gamma_{12,11})
+16\rho_n(\gamma_{12,12}) 
+12\rho_n(\gamma_{12,13}) 
+12\rho_n(\gamma_{12,14})\\
&+& 12\rho_n(\gamma_{12,15}) 
\big),
\end{eqnarray*}
where 
{\small
\begin{flalign*}
\gamma_{2,1} &= X^2 - 5, &\\
\gamma_{2,2} &= X^2 + 5,\\
\gamma_{4,1} &= X^4 - 5X^3 + 15X^2 - 25X + 25,\\
\gamma_{4,2} &= X^4 + 5X^2 + 25,\\
\gamma_{4,3} &= X^4 + 5X^3 + 15X^2 + 25X + 25,\\
\gamma_{8,1} &= X^8 - 10X^7 + 45X^6 - 130X^5 + 305X^4 - 650X^3 + 1125X^2 - 1250X + 625,\\
\gamma_{8,2} &= X^8 - 5X^7 + 10X^6 - 25X^5 + 75X^4 - 125X^3 + 250X^2 - 625X + 625,\\
\gamma_{8,3} &= X^8 - 5X^7 + 10X^6 + 5X^5 - 45X^4 + 25X^3 + 250X^2 - 625X + 625,\\
\gamma_{8,4} &= X^8 - 5X^7 + 15X^6 - 35X^5 + 80X^4 - 175X^3 + 375X^2 - 625X + 625,\\
\gamma_{8,5} &= X^8 - 5X^7 + 20X^6 - 65X^5 + 155X^4 - 325X^3 + 500X^2 - 625X + 625,\\
\gamma_{8,6} &= X^8 - 15X^6 + 105X^4 - 375X^2 + 625,\\
\gamma_{8,7} &= X^8 - 10X^6 + 55X^4 - 250X^2 + 625,\\
\gamma_{8,8} &= X^8 - 5X^6 + 25X^4 - 125X^2 + 625,\\
\gamma_{8,9} &= X^8 + 30X^4 + 625,\\
\gamma_{8,10} &= X^8 + 5X^6 - 20X^5 + 5X^4 - 100X^3 + 125X^2 + 625,\\
\end{flalign*}
}
\vspace{-10mm}
{\small
\begin{flalign*}
\gamma_{8,11} &= X^8 + 5X^6 - 10X^5 + 5X^4 - 50X^3 + 125X^2 + 625,\\
\gamma_{8,12} &= X^8 + 5X^6 + 10X^5 + 5X^4 + 50X^3 + 125X^2 + 625,\\
\gamma_{8,13} &= X^8 + 5X^6 + 20X^5 + 5X^4 + 100X^3 + 125X^2 + 625,&\\
\gamma_{8,14} &= X^8 + 5X^7 + 10X^6 - 5X^5 - 45X^4 - 25X^3 + 250X^2 + 625X + 625,\\
\gamma_{8,15} &= X^8 + 5X^7 + 10X^6 + 25X^5 + 75X^4 + 125X^3 + 250X^2 + 625X + 625,\\
\gamma_{8,16} &= X^8 + 5X^7 + 15X^6 + 35X^5 + 80X^4 + 175X^3 + 375X^2 + 625X + 625,\\
\gamma_{8,17} &= X^8 + 5X^7 + 20X^6 + 65X^5 + 155X^4 + 325X^3 + 500X^2 + 625X + 625,\\
\gamma_{8,18} &= X^8 + 10X^7 + 45X^6 + 130X^5 + 305X^4 + 650X^3 + 1125X^2 + 1250X + 625,\\
\gamma_{12,1} &= X^{12} - 5X^{11} + 5X^{10} + 5X^9 + 5X^8 + 75X^7 - 425X^6 + 375X^5 + 125X^4 + 625X^3 + 3125X^2 - 15625X + 15625,\\
\gamma_{12,2} &= X^{12} - 5X^{11} + 10X^{10} + 5X^9 - 20X^8 - 125X^7 + 575X^6 - 625X^5 - 500X^4 + 625X^3 + 6250X^2 - 15625X + 15625,\\
\gamma_{12,3} &= X^{12} - 5X^{11} + 15X^{10} - 45X^9 + 80X^8 - 125X^7 + 325X^6 - 625X^5 + 2000X^4 - 5625X^3 + 9375X^2 - 15625X + 15625,\\
\gamma_{12,4} &= X^{12} - 10X^{10} - 5X^9 + 80X^8 - 425X^6 + 2000X^4 - 625X^3 - 6250X^2 + 15625,\\
\gamma_{12,5} &= X^{12} - 10X^{10} + 10X^9 + 55X^8 - 25X^7 - 175X^6 - 125X^5 + 1375X^4 + 1250X^3 - 6250X^2 + 15625,\\
\gamma_{12,6} &= X^{12} - 5X^{10} - 15X^9 + 5X^8 + 50X^7 + 75X^6 + 250X^5 + 125X^4 - 1875X^3 - 3125X^2 + 15625,\\
\gamma_{12,7} &= X^{12} - 5X^{10} - 10X^9 - 20X^8 + 25X^7 + 325X^6 + 125X^5 - 500X^4 - 1250X^3 - 3125X^2 + 15625,\\
\gamma_{12,8} &= X^{12} - 15X^9 + 30X^8 - 50X^7 + 75X^6 - 250X^5 + 750X^4 - 1875X^3 + 15625,\\
\gamma_{12,9} &= X^{12} + 15X^9 - 20X^8 - 50X^7 + 75X^6 - 250X^5 - 500X^4 + 1875X^3 + 15625,\\
\gamma_{12,10} &= X^{12} + 5X^{10} - 5X^9 + 5X^8 + 50X^7 + 75X^6 + 250X^5 + 125X^4 - 625X^3 + 3125X^2 + 15625,\\
\gamma_{12,11} &= X^{12} + 10X^{10} + 20X^9 + 55X^8 + 175X^7 + 325X^6 + 875X^5 + 1375X^4 + 2500X^3 + 6250X^2 + 15625,\\
\gamma_{12,12} &= X^{12} + 15X^{10} - 10X^9 + 130X^8 - 75X^7 + 825X^6 - 375X^5 + 3250X^4 - 1250X^3 + 9375X^2 + 15625,\\
\gamma_{12,13} &= X^{12} + 5X^{11} + 10X^{10} + 25X^9 + 80X^8 + 225X^7 + 575X^6 + 1125X^5 + 2000X^4 + 3125X^3 + 6250X^2 + 15625X\\
 &+ 15625,\\
\gamma_{12,14} &= X^{12} + 5X^{11} + 15X^{10} + 25X^9 + 5X^8 - 125X^7 - 425X^6 - 625X^5 + 125X^4 + 3125X^3 + 9375X^2 + 15625X + 15625,\\
\gamma_{12,15} &= X^{12} + 5X^{11} + 20X^{10} + 75X^9 + 230X^8 + 600X^7 + 1450X^6 + 3000X^5 + 5750X^4 + 9375X^3 + 12500X^2 + 15625X\\
 &+ 15625.
\end{flalign*}
}
\end{theorem}
Each of the above polynomials is the characteristic polynomial of the Frobenius endomorphism of an abelian variety. 
Regarding the potential periodicity in $n$ of the formula in Theorem~\ref{thm:q5l4}, whether the abelian varieties are supersingular or not can be determined by the following theorem of Stichtenoth and Xing~\cite[Prop.~1]{Stichtenoth}.

\begin{theorem}\label{thm:SS}
Let $A$ be an abelian variety of dimension $g$ over $\F_q = \F_{p^r}$ and let $P(X) = X^{2g} + a_1X^{2g-1} + \cdots + a_g X^g + \cdots + q^g$
be the characteristic polynomial of the Frobenius endomorphism on $A$. Then $A$ is supersingular if and only if for all $1 \le k \le g$ one has
$p^{\lceil kn/2 \rceil} \mid a_k$.
\end{theorem}
Applying Theorem~\ref{thm:SS} we see that $\gamma_{2,1},\gamma_{2,2},\gamma_{4,1},\gamma_{4,2},\gamma_{4,3},\gamma_{8,2},\gamma_{8,8}$ and $\gamma_{8,15}$ are the only supersingular cases.
It is possible in principle to apply Kronecker's theorem to the phases of the non-supersingular Weil numbers to determine whether or not the formula 
in Theorem~\ref{thm:q5l4} is periodic in $n$. Note that the product of the powers of these polynomials as per 
Theorem~\ref{thm:q5l4} produces (a redundant version of) the characteristic polynomial of Frobenius of the intersection curve of genus $860$
given by the direct method. Although it is not true for this polynomial that $a_{12} \equiv 0 \pmod{5^6}$, as would be required by Theorem~\ref{thm:SS} 
for supersingularity, it is possible (although seemingly unlikely) that the formula for the non-redundant method is periodic in $n$, which would
require that all of the non-supersingular contributions cancel.

%----------------------------------------------------------------------------------
%----------------------------------------------------------------------------------

\section{An Algorithm for the Prescribed Traces Problem for $q=2$}\label{sec:char2}

When $l \ge p$ the approach of~\S\ref{sec:generalalg} breaks due to the failure of Newton's identities in positive characteristic.
For example, for $l = p$ we have $T_1(a^p) = t_{1}^p$ which does not feature $t_p$ and thus the system of linear trace 
conditions~(\ref{eq:newmain}) is underdetermined.
In this section we present an algorithm that obviates this issue and solves the
prescribed traces problem for binary base fields for $l \le 7$ and $n$ odd, which is potentially extendable to larger $l$. We focus here on
the $q = 2$ case, noting that the general $q = 2^r$ case follows {\em mutatis mutandis}.

In contrast to the main algorithm we now permit any subset of the first $l$ traces to be prescribed: for subscripts $1 \le l_0 < \cdots < l_{m-1} \le l$,
$n \ge l_{m-1}$ and $t_{l_0},\ldots,t_{l_{m-1}} \in \F_2$ we denote by $F_2(n,t_{l_0},\ldots,t_{l_{m-1}})$ the number of elements $a \in \F_{2^n}$
for which $T_{l_0}(a) = t_{l_0}$, \ldots, $T_{l_{m-1}}(a) = t_{l_{m-1}}$. 

After presenting a specialised version of the transform from~\S\ref{sec:transform1p} in~\S\ref{sec:transform1char2}, 
we present the indirect method, in~\S\ref{subsec:indirectbinary}, as it is slightly simpler than the direct method, 
which we present in~\S\ref{subsec:directbinary}. The algebraic property upon which both methods rely is presented in~\S\ref{sec:Tlab}.
The values $T_l(1) = \binom{n}{l}$ feature repeatedly in the ensuing methods. We therefore begin with two preliminary results,
the first of which is an immediate corollary of a theorem of Zabek~\cite{zabek}.

\begin{lemma}\label{lem:period}
For $p$ a prime and $j \ge 1$ the period of the sequence $\big(\binom{0}{j},\binom{1}{j},\binom{2}{j},\ldots\big)$ mod $p$ is $p^{1 + \lfloor \log_p{j} \rfloor}$.  
\end{lemma}
We shall use the following corollary repeatedly.

\begin{corollary}\label{cor:period}
For $p$ a prime and $j \ge 1$ the period in $n$ of the set of vectors $\big(\binom{n}{j},\binom{n}{j-1},\ldots,\binom{n}{1}\big)$ mod $p$ is $p^{1 + \lfloor \log_p{j} \rfloor}$. 
\end{corollary}
\begin{proof}
By Lemma~\ref{lem:period} the period in $n$ of $\binom{n}{j}$ mod $p$ is $p^{1 + \lfloor \log_p{j} \rfloor}$. Therefore the period in $n$ 
of the stated set of vectors is the LCM of periods of each, which is also $p^{1 + \lfloor \log_p{j} \rfloor}$.
\end{proof}

\subsection{A transform of the prescribed traces problem for $q=2$}\label{sec:transform1char2}

We now transform the problem of counting field elements with prescribed traces to the problem of counting the number of zeros 
of linear combinations of the trace functions. Note that counting zeros rather than evaluations to one means that one does not 
necessarily need to introduce a $1/\overline{n}$ term when defining the relevant curves and as a consequence some results are obtainable 
for all $n$, rather than just $n$ odd. We first fix some notation, which is virtually the same as that defined in~\S\ref{sec:transform1p}.

Let $f_0,\ldots,f_{m-1}: \F_{2^n} \rightarrow \F_2$ be any functions and let $\vec{f} = (f_{m-1},\ldots,f_0)$. 
For $i,j \in \{0,\ldots,2^m-1\}$ let $\vec{i} = (i_{m-1},\ldots,i_0)$ and $\vec{j} = (j_{m-1},\ldots,j_0)$ 
denote the binary expansions of $i$ and $j$ respectively, and let $\vec{i} \cdot \vec{j}$ denote their inner product mod $2$.  
For any $i \in \{0,\ldots,2^m-1\}$, let $\vec{i}\cdot \vec{f}$ denote the function 
\[
\sum_{k=0}^{m-1} i_k f_k: \F_{2^n} \rightarrow \F_2,
\]
and let $V(\vec{i}\cdot \vec{f})$ denote the number of zeros in $\F_{2^n}$ of $\vec{i}\cdot \vec{f}$. We interpret $V(\vec{0} \cdot \vec{f})$ to be 
the number of zeros of the empty function, which we define to be $2^n$.
Furthermore, let $N(\vec{j}) = N(j_{m-1},\ldots,j_{0})$ denote the number of $a \in \F_{2^n}$ such that $f_k(a) = j_k$, for 
$k = 0,\ldots,m-1$. 

As before, our goal is to express any $N(\vec{j})$ in terms of the $V(\vec{i}\cdot \vec{f})$, but we begin by first solving the inverse problem, 
\ie expressing any $V(\vec{i}\cdot \vec{f})$ in terms of the $N(\vec{j})$.

\begin{lemma}
With the notation as above, for $0 \le i \le 2^m-1$ we have
\begin{equation}\label{eq:basicrelation}
V(\vec{i}\cdot \vec{f}) = \sum_{\vec{i} \cdot \vec{j} \, = \, 0} N(\vec{j}).
\end{equation}
\end{lemma}
\begin{proof}
By definition, we have $V(\vec{i}\cdot \vec{f}) = \#\{ a \in \F_{2^n} \mid \vec{i}\cdot \vec{f} (a) = 0\} = 
\#\{ a \in \F_{2^n} \mid \sum_{k=0}^{m-1} i_k f_k(a) = 0\}$.
Since $N(\vec{j})$ counts precisely those $a \in \F_{2^n}$ such that $f_k(a) = j_k$, we must count over all those $\vec{j}$ for which
$\sum_{k=0}^{m-1} i_k f_k(a) = 0$, \ie those such that $\sum_{k=0}^{m-1} i_k j_k = 0$. \qed
\end{proof}
Writing Eq.~(\ref{eq:basicrelation}) in matrix form, for $i,j \in \{0,\ldots, 2^m-1\}$ we have
\[ 
\begin{bmatrix}
V(\vec{i}\cdot\vec{f})
\end{bmatrix}^T
=
S_m
\cdot
\begin{bmatrix}
N(\vec{j})
\end{bmatrix}^T,
\] 
where $(S_m)_{i,j} = 1 - \vec{i}\cdot\vec{j}$ is nothing but Sylvester's construction~\cite{sylvester} of Hadamard matrices~\cite{hadamard}, 
with the minus ones replaced with zeros. 
Now let $H_m$ be the $2^m \times 2^m$ matrix with entries $(H_m)_{i,j} = (-1)^{\vec{i} \cdot \vec{j}}$, \ie the Walsh-Hadamard
transform~\cite{walsh-orig}, scaled by $2^{m/2}$. Since the Walsh-Hadamard matrix is involutory we have $H_{m}^{-1} = \frac{1}{2^m} H_m$. 
Also let $A_m$ be the $2^m \times 2^m$ 
matrix with $(A_m)_{0,0} = 1$ and all other entries $0$, and let $B_m$ be the $2^m \times 2^m$ matrix with all entries $1$.
Lastly let $\text{Id}_m$ be the $2^m \times 2^m$ identity matrix.

\begin{lemma}
We have 
\begin{equation}\label{eq:Sinverse}
S_{m}^{-1} = \frac{1}{2^{m-1}}H_m - A_m.
\end{equation}
\end{lemma}
\begin{proof}
Noting that $2S_m = H_m + B_m$, we have
\begin{eqnarray*}
\Big(\frac{1}{2^{m-1}}H_m - A_m \Big)S_m &=& \frac{1}{2^m} H_m \cdot 2S_m - A_mS_m\\
&=& \frac{1}{2^m} H_m( H_m + B_m) - A_m S_m\\
&=& \text{Id}_m + \frac{1}{2^m}  H_mB_m - A_m S_m.
\end{eqnarray*}
Since all but the first row of $H_m$ contains the same number of ones and minus ones, $\frac{1}{2^m} H_m B_m$ consists of the all-one vector in the 
first row and the all-zero vector for the others, as does $A_m S_m$. \qed
\end{proof}
Thus in order to compute any of the $2^m$ possible outputs $N(\vec{j})$ of any set of $m$ functions $\vec{f}$, it is sufficient to count the number 
of zeros of all the $2^m$ $\F_2$-linear combinations of the functions, and then apply $S_{m}^{-1}$. In particular, one may choose the 
$f_{0},\ldots,f_{m-1}$ to be any subset of the trace functions $T_1,\ldots,T_{n}$. 
Note that for $q = 2^r$ with $r > 1$ one must revert to using transform~(\ref{eq:Sminus1}).

\subsection{The indirect method}\label{subsec:indirectbinary}

Let the input traces whose values are prescribed be $\vec{f} = (T_{l_{m-1}},\ldots,T_{l_{0}})$ with $l_{m-1} > \cdots > l_0$.
Then by the transform of the previous subsection, for all $i \in \{1,\ldots,2^m-1\}$ one needs to compute
\begin{equation}\label{eq:Zeq}
V(\vec{i} \cdot \vec{f}) = \#\{ a \in \F_{2^n} \mid \sum_{k=0}^{m-1} i_k T_{l_{k}}(a) = 0 \}. 
\end{equation}
In general this problem appears to be non-trivial, since the degree of each $T_{l_k}(a)$ in $a$ and its Frobenius powers is $l_k$ and the 
approach of~\S\ref{sec:generalalg} can not be applied.
However, it can be obviated --  at least for $n$ odd -- by using the following degree-lowering idea.
%We begin with the following extremely simple lemma.
%\begin{lemma}\label{lem:param2}
%\begin{enumerate}[label={(\arabic*)}]
%\item For $a \in \F_{2^n}$ the condition $T_1(a) = 0$ is equivalent to $a = a_{0}^2 + a_{0}$, for two $a_0 \in \F_{2^n}$.
%\item For $a \in \F_{2^n}$ with $n$ odd, the condition $T_1(a) = 1$ is equivalent to $a = a_{0}^2 + a_{0} + 1$, for two $a_0 \in \F_{2^n}$.
%\end{enumerate}
%\end{lemma}
%\begin{proof}
%See~\cite[Theorem 2.25]{lidl}.
%\end{proof}
Firstly, note that since the input to the trace functions has linear trace either $0$ or $1$, Eq.~(\ref{eq:Zeq}) can be rewritten as
\begin{equation*}
V( \vec{i} \cdot \vec{f} ) = \#\{ a \in \F_{2^n} \mid T_1(a) = 0, \sum_{k=0}^{m-1} i_k T_{l_{k}}(a) = 0 \} +
\#\{ a \in \F_{2^n} \mid T_1(a) = 1, \sum_{k=0}^{m-1} i_k T_{l_{k}}(a) = 0 \}. 
\end{equation*}
Secondly, note that for $n$ odd Lemma~\ref{lem:paramq} implies
\begin{equation}\label{eq:Zeq1}
V( \vec{i} \cdot \vec{f} ) = \frac{1}{2} \sum_{r_0 \in \F_2} \#\{ a_0 \in \F_{2^n} \mid \sum_{k=0}^{m-1} i_k T_{l_{k}}(a_{0}^2 + a_0 + r_0) = 0 \}.
% + \frac{1}{2} \#\{ a_{0} \in \F_{2^n} \mid \sum_{k=0}^{m-1} i_k T_{l_{k}}(a_{0}^2 + a_{0} + 1) = 0 \}. 
\end{equation}
Thirdly, it happens that the functions $T_l(a_{0}^2 + a_{0})$ and $T_l(a_{0}^2 + a_{0} + 1)$ for $2 \le l \le 7$ are all expressible in 
characteristic two as polynomials of traces of lower degree whose arguments are polynomials in $a_0$, see~\S\ref{sec:Tlab}. 
Hence rather than having a single equation whose zeros one must count (Eq.~(\ref{eq:Zeq})), one now has two equations whose number of zeros 
one must add and divide by $2$ (Eq.~(\ref{eq:Zeq1})), both now of lower degree than before.

If after the above three steps there are terms that are not linear, \ie not of the form $T_1(\cdot)$ for some argument, then 
the idea is to pick an argument of a trace function featuring in a non-linear term and apply the above three steps again.
In particular, if the chosen argument is $g(a_0)$ then one introduces a new variable $a_1$ and as before writes 
$g(a_0) = a_{1}^2 + a_{1} + r_1$ with $r_1 \in \F_2$ to account for whether the linear trace of $g(a_0)$ is $0$ or $1$, 
and expands all those terms in Eq.~(\ref{eq:Zeq1}) which have this argument. This results in four equations whose number of zeros one must 
sum and divide by $4$, with the degrees of the terms which feature this argument having been lowered, as before.

By recursively applying this idea and introducing variables $a_0,\ldots,a_{s_{i}-1}$ as necessary, 
with corresponding linear trace variables $r_0,\ldots,r_{s_{i}-1}$, since the degrees of the non-linear terms
always decreases one eventually obtains a set of $2^{s_{i}}$ trace equations of the form $T_1(g_{\vec{r}}(a_0,\ldots,a_{s_{i}-1})) = 0$ 
indexed by $\vec{r} =(r_0,\ldots,r_{s_{i}-1}) \in (\F_2)^{s_{i}}$, the vector of trace values of the $s_{i}$ rewritten arguments.
Each of these can be eliminated by introducing a final variable $a_{s_{i}}$ and writing 
$a_{s_{i}}^2 + a_{s_{i}} = g_{\vec{r}}(a_0,\ldots,a_{s_{i}-1})$,
giving a system of $s_{i}$ equations in $s_{i}+1$ variables, with the initial variable $a$ having been completely eliminated 
in going from Eq.~(\ref{eq:Zeq}) to Eq.~(\ref{eq:Zeq1}). Observe that since $a_0$ is the only free variable in this system, the affine
algebraic set that it defines is one-dimensional and is thus a curve.

Depending on $i$, the final linear trace equation features a subset of the coefficients $\binom{n}{l_k}$ for $k = 0, \ldots, m-1$, 
which by Corollary~\ref{cor:period} have period a divisor of the maximum period $2^\theta  := 2^{1 + \lfloor \log_2{l_{m-1}}\rfloor}$.
Therefore let $\overline{n} \in \{1,3,5,\ldots,2^\theta-1\}$ represent the applicable residue classes of $n$ mod $2^\theta$
and substitute each $\binom{n}{l_k}$ featuring in the above intersection by $\binom{\overline{n}}{l_k}$.
We denote by $C_{i,\vec{r},\overline{n}}$ the affine curve defined by the above $s_{i}$ 
equations whose coefficients are functions of $\vec{r} = (r_0,\ldots,r_{s_{i}-1})$ and $\overline{n}$.
For each $\overline{n}$ we then take the average over all specialisations of $\vec{r}$ of the number of $\F_{2^n}$-rational points
on $C_{i,\vec{r},\overline{n}}$ -- remembering to also divide by $2$ to account for the presence of $a_{s_{i}}$ -- 
in order to determine $V(\vec{i}\cdot \vec{f})$, \ie for all $n \equiv \overline{n} \pmod{2^\theta}$ 
with $n \ge l_{m-1}$ we have
\begin{equation}\label{eq:indirectmethodbinary}
V(\vec{i}\cdot \vec{f}) = \frac{1}{2^{s_{i}+1}} \sum_{\vec{r} \in (\F_2)^{s_{i}}} \#C_{i,\vec{r},\overline{n}}(\F_{2^n}).
\end{equation}

\subsection{The direct method}\label{subsec:directbinary}
By definition, we have
\[
F_2(n,t_{l_0},\ldots,t_{l_{m-1}}) = \#\{ a \in \F_{2^n} \mid T_{l_{0}}(a) = t_{l_0},\ldots, T_{l_{m-1}}(a) = t_{l_{m-1}}\}.
\]
By using the same recursive procedure described in the previous subsection, for $n \ge l_{m-1}$ and odd suppose that for $0 \le k \le m-1$ one 
linearises $T_{l_k}(a)$, \ie reduces it to a $T_1$ expression, by introducing variables $a_{0},a_{k,1},\ldots,a_{k,s_k-1}$ with corresponding linear 
traces $r_{0},r_{k,1},\ldots,r_{k,s_k-1}$, resulting in $2^{s_k}$ trace equations of the form 
\begin{equation}\label{eq:traceequationbinary}
T_1(g_{\vec{r}_k}(a_0,a_{k,1},\ldots,a_{k,s_k-1})) =  t_{l_k}.
\end{equation}
Observe that amongst all of the $a_{k,i}$'s there may be definitional repetitions. 
Therefore, let $U$ be the union of all of the equations arising from the linearisation of 
each of the featured trace functions, once repetitions have been eliminated. Let $s \le \sum_{k=0}^{m-1} s_k$ be the number of variables in this union, 
which we relabel as $a_0,a_1,\ldots,a_{s-1}$, and let $\vec{r} = (r_0,\ldots,r_{s-1})$ be the corresponding vector of possible linear traces of the 
rewritten arguments. Note that since $a$ was completely eliminated by the introduction of $a_0$ there are only 
$s-1$ corresponding equations in addition to the $m$ trace conditions in~(\ref{eq:traceequationbinary}).
Each of these $m$ trace equations can be parameterised by introducing for $0 \le k \le m-1$ a variable $a_{s_k}$ 
and writing $a_{s_k}^2 + a_{s_k} + t_{l_k} = g_{\vec{r}_k}(a_{0},a_1\ldots,a_{s-1})$, giving a system of $m+s-1$ equations in $m+s$ variables.
As in the indirect method observe that $a_0$ is the only free variable in this system and so the affine
algebraic set that it defines is one-dimensional and is thus a curve.

Also as before, this system features the coefficients $\binom{n}{l_k}$, which by Corollary~\ref{cor:period} have period 
a divisor of $2^{\theta}$, so again let $\overline{n} \in \{1,3,5,\ldots,2^{\theta}-1\}$ represent the applicable residue classes of 
$n$ mod $2^{\theta}$, and substitute each $\binom{n}{l_k}$ by $\binom{\overline{n}}{l_k}$.
We denote by $C_{\vec{r},\overline{n}}$ the affine curve defined by the $m+s-1$ 
above equations whose coefficients are functions of the prescribed values, $\vec{r} = (r_0,\ldots,r_{s-1})$ and $\overline{n}$.
For each $\overline{n}$ we take the average over all specialisations of $\vec{r}$ of the number of $\F_{2^n}$-rational points
on $C_{\vec{r},\overline{n}}$ and divide by $2^{m}$ in order to determine $F_2(n,t_{l_0},\ldots,t_{l_{m-1}})$, \ie 
for all $n \equiv \overline{n} \pmod{2^{\theta}}$ with $n \ge l_{m-1}$ we have
\begin{equation}\label{eq:directmethodbinary}
F_2(n,t_{l_0},\ldots,t_{l_{m-1}}) = \frac{1}{2^{m+s}} \sum_{\vec{r} \in (\F_2)^s} \#C_{\vec{r},\overline{n}}(\F_{2^n}).
\end{equation}
Since in general there will be more variables and equations than for the indirect method, computing the zeta functions of the arising curves 
will likely be more costly for this method.

\subsection{Computing $T_{l}(\alpha - \beta)$}\label{sec:Tlab}
We now explain the how one can obtain expressions for $T_l(a_{0}^2 + a_0)$ and $T_l(a_{0}^2 + a_0 + 1)$ for $2 \le l \le 7$.
Expressions for $T_2(\alpha - \beta)$ and $T_3(\alpha - \beta)$ in characteristic two were given by Fitgerald and Yucas~\cite[Lemma 1.1]{fitzyucas} and
proven by expanding the bilinear forms $T_l(\alpha + \beta) + T_l(\alpha) + T_l(\beta)$ in terms of the Frobenius powers of $\alpha$ and $\beta$, 
and deducing the correct function of the lower degree traces. It is possible -- though laborious -- to continue in this manner (we did so for $l = 4$), 
so instead we present an easier method.

We begin by recalling Eq~(\ref{eq:NI}):
\begin{equation*}
k \,T_k(\alpha) = \sum_{j=1}^k (-1)^{j-1} T_{l-j}(\alpha)T_1(\alpha^j).
\end{equation*}
In order to use the argument $\alpha - \beta$ we need to work instead in the ring $\Z[\alpha_1,\ldots,\alpha_n,\beta_1,\ldots,\beta_n]$ and with
the ring of multisymmetric functions in two variables, with the symmetric group $S_n$ acting on $\alpha_1,\ldots,\alpha_n$ and $\beta_1,\ldots,\beta_n$ independently (see~\cite{dalbec} for a formal definition). Abusing notation slightly, in this ring Eq.~(\ref{eq:NI}) becomes:
\begin{equation}\label{eq:NIab}
l \,T_l(\alpha -\beta) = \sum_{k=1}^l (-1)^{k-1} T_{l-k}(\alpha - \beta)T_1( (\alpha - \beta)^k).
\end{equation}
If one works over $\Q$ rather than $\Z$ then Eq.~(\ref{eq:NIab}) leads to expressions for $T_{l}(\alpha - \beta)$ for any $l \ge 1$ as a sum of
products of $T_1$ terms with arguments being various powers of $\alpha - \beta$.
However, in positive characteristic this is not so useful. For instance, computing $T_2(\alpha- \beta)$ in
characteristic two in this way is not possible as the left hand side vanishes.
Nevertheless, working inductively for $2 \le l \le 7$ and applying Newton's identities evaluated at various products of powers of $\alpha$ 
and $\beta$ so that no trace occurs to any power larger than one, all of the coefficients become divisible by $l$. Upon dividing by $l$ one obtains an equation for $T_l(\alpha - \beta)$ over $\Z$,
which can then be substituted into Eq.~(\ref{eq:NIab}) in order to attempt to compute $T_{l+1}(\alpha - \beta)$.
For example, from Eq.~(\ref{eq:NIab}) we have
\begin{eqnarray*}
2T_2(\alpha - \beta) &=& T_1(\alpha - \beta)^2 - T_1( (\alpha - \beta)^2 ) \\
&=&  (T_1(\alpha) - T_1(\beta))^2 - T_1(\alpha^2) + 2T_1(\alpha\beta) - T_1(\beta^2) \\
&=& T_1(\alpha)^2 - 2T_1(\alpha)T_1(\beta) + T_1(\beta)^2 - T_1(\alpha^2) + 2T_1(\alpha\beta) - T_1(\beta^2) \\
&=& 2T_2(\alpha) + 2T_2(\beta) -  2T_1(\alpha)T_1(\beta) + 2T_1(\alpha\beta),
\end{eqnarray*}
where in the final line we have used Eq.~(\ref{eq:NI}) for $l=2$ for $\alpha$ and $\beta$ separately. We have therefore proven that
\[
T_2(\alpha - \beta) = T_2(\alpha) + T_2(\beta) - T_1(\alpha)T_1(\beta) + T_1(\alpha\beta).
\]
The parts of the following lemma can either be proven by induction on $n$ using the identity 
\[
T_{l}(\alpha_1,\ldots,\alpha_n) = T_{l-1}(\alpha_1,\ldots,\alpha_{n-1}) \, \alpha_n + T_{l}(\alpha_1,\ldots,\alpha_{n-1}),
\]
or via sequences of manipulations as described above\footnote[2]{See~\url{NewtonApproach_l_le_7.mw} for a derivation of these expressions.}.

\begin{lemma}\label{lem:T}
For all $n \ge l$ we have
\begin{enumerate}[label={(\arabic*)}]
\item $\displaystyle\begin{aligned}[t]
T_1(\alpha - \beta) = T_1(\alpha) - T_1(\beta),
\end{aligned}$
\item $\displaystyle\begin{aligned}[t]
T_2(\alpha - \beta) = T_2(\alpha) + T_2(\beta) - T_1(\alpha)T_1(\beta) + T_1(\alpha\beta),
\end{aligned}$
\item $\displaystyle\begin{aligned}[t]
T_3(\alpha - \beta) &=  T_3(\alpha)-T_3(\beta) + T_1(\alpha)T_2(\beta) - T_1(\beta)T_2(\alpha) +T_1(\alpha)T_1(\alpha\beta) -T_1(\beta)T_1(\alpha\beta)\\
& +  T_1(\alpha\beta^2) - T_1(\alpha^2\beta),
\end{aligned}$
\item $\displaystyle\begin{aligned}[t] 
T_4(\alpha - \beta) &= T_4(\alpha)+T_4(\beta) -T_1(\alpha)T_3(\beta)-T_1(\beta)T_3(\alpha) +T_2(\alpha)T_2(\beta) -T_1(\alpha)T_1(\beta)T_1(\alpha\beta)\\
&+T_1(\alpha\beta)T_2(\alpha)+T_1(\alpha\beta)T_2(\beta)-T_1(\alpha)T_1(\alpha^2\beta)+T_1(\alpha)T_1(\alpha\beta^2)+T_1(\beta)T_1(\alpha^2\beta)\\
& -T_1(\beta)T_1(\alpha\beta^2) +T_1(\alpha^3\beta)-T_1(\alpha^2\beta^2)+T_1(\alpha\beta^3)+T_2(\alpha\beta),
\end{aligned}$
\item $\displaystyle\begin{aligned}[t] 
T_5(\alpha - \beta) &= T_5(\alpha)-T_5(\beta) +T_1(\alpha)T_4(\beta) -T_1(\beta)T_4(\alpha) +T_2(\beta)T_3(\alpha)-T_2(\alpha)T_3(\beta) + T_1(\alpha)T_1(\beta)T_1(\alpha^2\beta)\\
& -T_1(\beta)T_1(\alpha\beta)T_2(\alpha)+T_1(\alpha)T_1(\alpha\beta)T_2(\beta)-T_1(\alpha)T_1(\beta)T_1(\alpha\beta^2) -T_1(\alpha^2 \beta)T_2(\alpha)\\
& -T_1(\alpha^2 \beta)T_2(\beta)+T_1(\alpha\beta^2)T_2(\alpha)+T_1(\alpha \beta^2)T_2(\beta)+T_1(\alpha\beta)T_3(\alpha)-T_1(\alpha\beta)T_3(\beta)\\
& -T_1(\alpha\beta)T_1(\alpha^2 \beta)+T_1(\alpha\beta)T_1(\alpha \beta^2)+T_1(\alpha)T_1(\alpha^3 \beta)-T_1(\alpha)T_1(\alpha^2 \beta^2)+T_1(\alpha)T_1(\alpha \beta^3)\\
& -T_1(\beta)T_1(\alpha^3 \beta)+T_1(\beta)T_1(\alpha^2 \beta^2) -T_1(\beta)T_1(\alpha \beta^3)+T_1(\alpha)T_2(\alpha\beta)-T_1(\beta)T_2(\alpha\beta)\\
& -T_1(\alpha^4 \beta)+2T_1(\alpha^3 \beta^2)-2T_1(\alpha^2 \beta^3)+T_1(\alpha \beta^4),\\
\end{aligned}$
\item $\displaystyle\begin{aligned}[t]
T_6(\alpha - \beta) &= T_6(\alpha)+T_6(\beta) -T_1(\alpha)T_5(\beta)-T_1(\beta)T_5(\alpha) +T_2(\alpha)T_4(\beta)+T_2(\beta)T_4(\alpha)-T_3(\beta)T_3(\alpha)\\
& +T_2(\alpha)T_2(\alpha\beta)+T_1(\alpha^5\beta)-2T_1(\alpha^4\beta^2)+2T_1(\alpha^3\beta^3)-2T_1(\alpha^2\beta^4)+T_1(\alpha\beta^5)\\
& +T_1(\alpha\beta)T_2(\alpha)T_2(\beta)+T_2(\alpha\beta)T_2(\beta)-T_1(\alpha)T_1(\alpha\beta)T_1(\alpha^2\beta)-T_1(\beta)T_1(\alpha\beta)T_1(\beta^2\alpha)\\
& -T_1(\alpha)T_1(\beta)T_2(\alpha\beta)+T_1(\beta)T_1(\alpha^2\beta)T_2(\alpha)+T_1(\alpha)T_1(\alpha\beta^2)T_2(\beta)-T_1(\alpha)T_1(\beta)T_1(\alpha\beta^3)\\
& +T_1(\alpha)T_1(\beta)T_1(\alpha^2\beta^2)+T_1(\alpha)T_1(\alpha\beta)T_1(\alpha\beta^2)-T_1(\beta)T_1(\alpha\beta^2)T_2(\alpha)-T_1(\alpha)T_1(\beta)T_1(\alpha^3\beta)\\
& -T_1(\beta)T_1(\alpha\beta)T_3(\alpha)-T_1(\alpha)T_1(\alpha^2\beta)T_2(\beta)-T_1(\alpha)T_1(\alpha\beta)T_3(\beta)+T_1(\beta)T_1(\alpha\beta)T_1(\alpha^2\beta)\\
& +4T_3(\alpha\beta)+T_2(\alpha\beta^2)-T_1(\alpha)T_1(\alpha^4\beta)+2T_1(\alpha)T_1(\alpha^3\beta^2)-2T_1(\alpha)T_1(\alpha^2\beta^3)+T_1(\alpha)T_1(\alpha\beta^4)\\
& +T_1(\beta)T_1(\alpha^4\beta)-2T_1(\beta)T_1(\alpha^3\beta^2)+2T_1(\beta)T_1(\alpha^2\beta^3)-T_1(\beta)T_1(\alpha\beta^4)+T_2(\alpha^2\beta)\\
& +T_1(\alpha\beta)T_4(\alpha)+T_1(\alpha\beta)T_4(\beta)+T_1(\alpha\beta)T_1(\alpha^3\beta)+T_1(\alpha\beta)T_1(\alpha\beta^3)-T_1(\alpha\beta)T_2(\alpha\beta)\\
& -T_1(\alpha^2\beta)T_3(\alpha)+T_1(\alpha^2\beta)T_3(\beta)-T_1(\alpha^2\beta)T_1(\alpha\beta^2)+T_1(\alpha\beta^2)T_3(\alpha)-T_1(\alpha\beta^2)T_3(\beta)\\
& +T_1(\alpha^3\beta)T_2(\alpha)+T_1(\alpha^3\beta)T_2(\beta)-T_1(\alpha^2\beta^2)T_2(\alpha)-T_1(\alpha^2\beta^2)T_2(\beta)+T_1(\alpha\beta^3)T_2(\alpha)\\
& +T_1(\alpha\beta^3)T_2(\beta),
\end{aligned}$
\item $\displaystyle\begin{aligned}[t]
T_7(\alpha - \beta) &= 
T_7(\alpha) -T_7(\beta) +T_1(\alpha)T_6(\beta) -T_6(\alpha)T_1(\beta) -T_2(\alpha)T_5(\beta) +T_2(\beta)T_5(\alpha)+T_3(\alpha)T_4(\beta)\\
&-T_3(\beta)T_4(\alpha) + T_1(\alpha\beta)T_5(\alpha) 
 -T_1(\alpha\beta)T_5(\beta)
+T_1(\alpha\beta^6) - T_1(\alpha^6\beta)
-T_1(\alpha\beta)T_1(\alpha^4\beta)\\
&+2T_1(\alpha\beta)T_1(\alpha^3\beta^2)
-2T_1(\alpha\beta)T_1(\alpha^2\beta^3)
+T_1(\alpha\beta)T_1(\alpha\beta^4)
-T_1(\alpha^2\beta)T_1(\alpha^3\beta)\\
&+T_1(\alpha^2\beta)T_1(\alpha^2\beta^2)
-T_1(\alpha^2\beta)T_1(\alpha\beta^3)
-T_2(\alpha\beta)T_3(\beta)
-T_1(\beta)T_1(\alpha\beta^2)T_3(\alpha)\\
&-T_1(\beta)T_1(\alpha^3\beta)T_2(\alpha)
+T_1(\beta)T_1(\alpha^2\beta^2)T_2(\alpha)
-T_1(\beta)T_2(\alpha\beta)T_2(\alpha)
-T_1(\beta)T_1(\alpha\beta^3)T_2(\alpha)\\
&+3T_1(\alpha^5\beta^2)
-5T_1(\alpha^4\beta^3)
+5T_1(\alpha^3\beta^4)
-3T_1(\alpha^2\beta^5)
+T_2(\alpha\beta)T_3(\alpha)
+T_1(\alpha)T_1(\alpha\beta)T_1(\alpha^3\beta)\\
&+2T_1(\alpha)T_1(\alpha\beta)T_1(\alpha^2\beta^2)
+T_1(\alpha)T_1(\alpha\beta)T_1(\alpha\beta^3)
+T_1(\alpha)T_1(\alpha\beta)T_4(\beta)\\
&-3T_1(\alpha)T_1(\alpha\beta)T_2(\alpha\beta)
+T_1(\alpha)T_1(\alpha^2\beta)T_3(\beta)
-T_1(\alpha)T_1(\alpha^2\beta)T_1(\alpha\beta^2)\\
&-T_1(\alpha)T_1(\alpha\beta^2)T_3(\beta)
+T_1(\alpha)T_1(\alpha^3\beta)T_2(\beta)
-T_1(\alpha)T_1(\alpha^2\beta^2)T_2(\beta)
+T_1(\alpha)T_1(\beta)T_1(\alpha^4\beta)\\
&-2T_1(\alpha)T_1(\beta)T_1(\alpha^3\beta^2)
+2T_1(\alpha)T_1(\beta)T_1(\alpha^2\beta^3)
-T_1(\alpha)T_1(\beta)T_1(\alpha\beta^4)\\
&+T_1(\alpha)T_2(\alpha\beta)T_2(\beta)
+T_1(\alpha)T_1(\alpha\beta^3)T_2(\beta)
-T_1(\beta)T_1(\alpha\beta)T_1(\alpha^3\beta)\\
&-2T_1(\beta)T_1(\alpha\beta)T_1(\alpha^2\beta^2)
-T_1(\beta)T_1(\alpha\beta)T_1(\alpha\beta^3)
-T_1(\beta)T_1(\alpha\beta)T_4(\alpha)\\
&+3T_1(\beta)T_1(\alpha\beta)T_2(\alpha\beta)
+T_1(\beta)T_1(\alpha^2\beta)T_3(\alpha)
+T_1(\beta)T_1(\alpha^2\beta)T_1(\alpha\beta^2)\\
&-T_1(\alpha^2\beta)T_2(\beta)T_2(\alpha)
+T_1(\alpha\beta^2)T_2(\beta)T_2(\alpha)
+T_1(\alpha\beta)T_2(\beta)T_3(\alpha)
+T_1(\alpha)T_2(\alpha\beta^2)\\
\end{aligned}$
\end{enumerate}
$\begin{aligned}
\hspace{22.5mm} &+T_1(\alpha)T_2(\alpha^2\beta)
+10T_1(\alpha)T_3(\alpha\beta)
-T_1(\beta)T_2(\alpha\beta^2)
-T_1(\beta)T_2(\alpha^2\beta)
-10T_1(\beta)T_3(\alpha\beta)\\
&+T_1(\alpha)T_1(\alpha^5\beta)
-2T_1(\alpha)T_1(\alpha^4\beta^2)
-2T_1(\alpha)T_1(\alpha^2\beta^4)
+T_1(\alpha)T_1(\alpha\beta^5)
-T_1(\beta)T_1(\alpha^5\beta)\\
&+2T_1(\beta)T_1(\alpha^4\beta^2)
+2T_1(\beta)T_1(\alpha^2\beta^4)
-T_1(\beta)T_1(\alpha\beta^5)
-T_1(\alpha^4\beta)T_2(\alpha)
-T_1(\alpha^4\beta)T_2(\beta)\\
&+2T_1(\alpha^3\beta^2)T_2(\alpha)
+2T_1(\alpha^3\beta^2)T_2(\beta)
-2T_1(\alpha^2\beta^3)T_2(\alpha)
-2T_1(\alpha^2\beta^3)T_2(\beta)\\
&+T_1(\alpha\beta^4)T_2(\alpha)
+T_1(\alpha\beta^4)T_2(\beta)
+T_1(\alpha^3\beta)T_3(\alpha)
-T_1(\alpha^3\beta)T_3(\beta)
-T_1(\alpha^2\beta^2)T_3(\alpha)\\
&+T_1(\alpha^2\beta^2)T_3(\beta)
+T_1(\alpha\beta^3)T_3(\alpha)
-T_1(\alpha\beta^3)T_3(\beta)
-T_1(\alpha^2\beta)T_4(\alpha)
-T_1(\alpha^2\beta)T_4(\beta)\\
&-T_1(\alpha^2\beta)T_2(\alpha\beta)
+T_1(\alpha\beta^2)T_1(\alpha^3\beta)
-T_1(\alpha\beta^2)T_1(\alpha^2\beta^2)
+T_1(\alpha\beta^2)T_1(\alpha\beta^3)\\
&+T_1(\alpha\beta^2)T_4(\alpha)
+T_1(\alpha\beta^2)T_4(\beta)
+T_1(\alpha\beta^2)T_2(\alpha\beta)
-T_1(\alpha\beta)T_1(\alpha^2\beta)T_2(\alpha)\\
&-T_1(\alpha\beta)T_1(\alpha^2\beta)T_2(\beta)
+T_1(\alpha\beta)T_1(\alpha\beta^2)T_2(\alpha)
+T_1(\alpha\beta)T_1(\alpha\beta^2)T_2(\beta)\\
&-T_1(\alpha\beta)T_2(\alpha)T_3(\beta)
+T_1(\alpha)T_1(\beta)T_1(\alpha\beta)T_1(\alpha^2\beta)
-T_1(\alpha)T_1(\beta)T_1(\alpha\beta)T_1(\alpha\beta^2).
\end{aligned}$
\end{lemma}

Observe that for each $1 \le l \le 7$ all of the terms appearing in Lemma~\ref{lem:T} part (l) have total degree $l$, when 
counted in the natural way. Hence when one reduces mod $2$ and sets 
$\beta = \alpha^2$, the two terms $T_l(\alpha)$ and $T_l(\beta)$ cancel, leaving only $T_l$'s of degree $< l$, as claimed earlier.

Unfortunately, our approach of using Newton's identities evaluated for various $l$ at products of powers of $\alpha$ and $\beta$ so that no 
trace occurs to any power larger than one, fails for $l = 8$ due to the presence of the term $T_2(xy)^2$, which can not be eliminated while 
keeping the remaining terms' coefficients divisible by $8$. Whether or not there exist such expressions for $T_l(\alpha - \beta)$ over $\Z$ 
for $l \ge 8$, we leave as an open problem. Note that it is known that the ring of multisymmetric functions in two sets of variables 
$\alpha_1,\ldots,\alpha_n$ and $\beta_1,\ldots,\beta_n$ is not generated over $\Z$ by the elementary multisymmetric functions that we are using, unless 
$n=2$~\cite{dalbec}. However, since we are only interested in a particular family of multisymmetric functions -- namely $T_l(\alpha - \beta)$ -- and not all of them, it is possible that such expressions exist.

%---------------------------------------------------------------------------------------------------

\section{Curves and Explicit Formulae for $q = 2$ and $l \le 7$}\label{sec:binaryzetas}

In this section we apply the indirect method and for some cases the direct method to determine relevant curves for $l \le 7$ and $n$ odd.
Using Magma we provide explicit formulae for $l \le 5$ and $n$ odd using the indirect method, and provide explicit formulae for a subset of these cases using the direct method, for all relevant $n$.

In practice, rather than obtain a curve as sketched in the previous section for each $i \in \{0,\ldots,2^l-1\}$, it is more efficient 
to compute a linearisation for each featured $T_{l_k}(a)$ and then combine them as appropriate according to whether $i_k$ is $0$ or $1$ for a 
given $i$, as we do in the examples that follow. 
Observe that for the indirect method, once $V(\vec{i}\cdot\vec{f})$ has been obtained for 
$\vec{f} = (T_l,\ldots,T_1)$ and $i \in \{0,\ldots,2^l-1\}$, these functions need not be recomputed for the subsequent $\vec{f} = (T_{l+1},\ldots,T_1)$.

\subsection{Computing $F_2(n,t_1,t_2,t_3)$}

The formulae for $l = 3$ were presented in~\cite[\S\S4\&5]{AGGMY} -- although obtained there in a slightly different 
manner --  but we include them here for demonstration purposes and completeness.

\subsubsection{Indirect method.} Setting $\vec{f} = (T_3,T_2,T_1)$, by the transform~(\ref{eq:Sinverse}) we have
\begin{equation}\label{eq:transform1_3}
{\small
\begin{bmatrix}
F_2(n,0,0,0)\\
F_2(n,1,0,0)\\
F_2(n,0,1,0)\\
F_2(n,1,1,0)\\
F_2(n,0,0,1)\\
F_2(n,1,0,1)\\
F_2(n,0,1,1)\\
F_2(n,1,1,1)\\
\end{bmatrix}
=
\begin{bmatrix}
N(\vec{0})\\
N(\vec{1})\\
N(\vec{2})\\
N(\vec{3})\\
N(\vec{4})\\
N(\vec{5})\\
N(\vec{6})\\
N(\vec{7})\\
\end{bmatrix}
=
\frac{1}{4}
\begin{bmatrix*}[r]
-3 & 1 & 1 & 1 & 1 & 1 & 1 & 1\\
1 & -1 & 1 & -1 & 1 & -1 & 1 & -1\\
1 & 1 & -1 & -1 & 1 & 1 & -1 & -1\\
1 & -1 & -1 & 1 & 1 & -1 & -1 & 1\\
1 & 1 & 1 & 1 & -1 & -1 & -1 & -1\\
1 & -1 & 1 & -1 & -1 & 1 & -1 & 1\\
1 & 1 & -1 & -1 & -1 & -1 & 1 & 1\\
1 & -1 & -1 & 1 & -1 & 1 & 1 & -1\\
\end{bmatrix*}
\begin{bmatrix}
V(\vec{0} \cdot \vec{f})\\
V(\vec{1} \cdot \vec{f})\\
V(\vec{2} \cdot \vec{f})\\
V(\vec{3} \cdot \vec{f})\\
V(\vec{4} \cdot \vec{f})\\
V(\vec{5} \cdot \vec{f})\\
V(\vec{6} \cdot \vec{f})\\
V(\vec{7} \cdot \vec{f})\\
\end{bmatrix}.}
\end{equation}
By definition we have $V(\vec{0} \cdot \vec{f}) = 2^n$, while $V(\vec{1} \cdot \vec{f}) = V(T_1) = \#\{a \in \F_{2^n} \mid T_1(a) = 0\} = 2^{n-1}$.
To determine $V( \vec{i} \cdot \vec{f})$ for $2 \le i \le 7$, we use Lemma~\ref{lem:T} parts (1) to (3). In particular, setting 
$\alpha = a_{0}^2$ and $\beta = a_{0}$ for $r_0 = 0$, and $\alpha = a_{0}^2 + a_0$ and $\beta = 1$ for $r_0 = 1$, and evaluating mod $2$ 
gives the following:
\begin{eqnarray}
\label{eq:T1linear} T_1(a_{0}^2 + a_0 + r_0) &=& T_1(r_0),\\
%T_1(a_{0}^2 + a_0 + 1) &=& \binom{n}{1},\\
%T_2(a_{0}^2 + a_0) &=& T_1(a_{0}^3 + a_{0}),\\
\label{eq:T2linear} T_2(a_{0}^2 + a_0 + r_0) &=& T_1\Big(a_{0}^3 + a_{0} + r_0\binom{n}{2}\Big),\\
%T_3(a_{0}^2 + a_0) &=& T_1(a_{0}^5 + a_{0}),\\
\label{eq:T3linear} T_3(a_{0}^2 + a_0 + r_0) &=& T_1\Big(a_{0}^5 + a_{0} + r_0\Big(a_{0}^3 + a_0 + \binom{n}{3} \Big)\Big).
\end{eqnarray}
For $2 \le i \le 7$ let $\vec{i} = (i_2,i_1,i_0)$. The curves we are interested in for $a$ of trace $r_0$
are therefore:
\begin{eqnarray}
%\nonumber a_{1}^2 + a_{1} &=& i_2(a_{0}^5 + a_0) + i_1(a_{0}^3 + a_0),\\
\label{3trace1} a_{1}^2 + a_{1} &=& i_2\Big(a_{0}^5 + a_{0} + r_0\Big(a_{0}^3 + a_0 + \binom{n}{3} \Big)\Big) 
+ i_1\Big(a_{0}^3 + a_{0} + r_0\binom{n}{2}\Big) + i_0 r_0 \binom{n}{1}.
\end{eqnarray}
These curves have genus $1$ if $i_2 = 0$ and genus $2$ if $i_2 = 1$, and are all supersingular. As pointed out in~\cite{AGGMY}, this is why the formulae are periodic in $n$.

By Corollary~\ref{cor:period} the vector $(\binom{n}{3},\binom{n}{2},\binom{n}{1})$ mod $2$ has period $4$ and is equal to $(0,0,1)$ if 
$n \equiv 1 \pmod{4}$, and $(1,1,1)$ if $n \equiv 3 \pmod{4}$. Hence there are two cases to consider when computing the zeta functions of the 
curves specified in Eq.~(\ref{3trace1}).
In order to express $F_2(n,t_1,t_2,t_3)$ compactly, we define the following polynomials:
\begin{eqnarray*}
\delta_{2,1} &=& X^2 + 2X + 2,\\
\delta_{4,1} &=& X^4 + 2X^3 + 2X^2 + 4X + 4,
\end{eqnarray*}
which are the characteristic polynomials of Frobenius of the following two supersingular curves, respectively:
\begin{eqnarray*}
C_{2,1}/\F_2: && y^2 + y = x^3 + x,\\
C_{4,1}/\F_2: && y^2 + y = x^5 + x^3. 
\end{eqnarray*}
Using Magma to compute the zeta functions of the curves~(\ref{3trace1}) and applying~(\ref{eq:transform1_3}) gives the following.

\begin{theorem}\label{thm:3traces}
For $n \ge 3$ we have
{\small
\begin{eqnarray*}
F_2(n,0,0,0) &=& 2^{n-3} - \frac{1}{8} \big( 2\rho_n(\delta_{2,1}) + \rho_n(\delta_{4,1}) \big) \ \ \text{if} \ n \equiv 1,3 \pmod{4}\\
F_2(n,1,0,0) &=& 2^{n-3} - \frac{1}{8} \cdot \begin{cases}
\begin{array}{lr}
2\rho_n(\delta_{2,1}) + \rho_n(\delta_{4,1}) & \ \text{if} \ n \equiv 1 \pmod{4}\\
-\rho_n(\delta_{4,1}) & \ \text{if} \ n \equiv 3 \pmod{4}\\
\end{array}
\end{cases}\\
F_2(n,0,1,0) &=&2^{n-3} - \frac{1}{8} \big( - \rho_n(\delta_{4,1}) \big) \ \ \text{if} \ n \equiv 1,3 \pmod{4}\\
F_2(n,1,1,0) &=&2^{n-3} - \frac{1}{8} \cdot \begin{cases}
\begin{array}{lr}
-2\rho_n(\delta_{2,1}) + \rho_n(\delta_{4,1}) & \ \text{if} \ n \equiv 1 \pmod{4}\\
-\rho_n(\delta_{4,1}) & \ \text{if} \ n \equiv 3 \pmod{4}\\
\end{array}
\end{cases}\\
F_2(n,0,0,1) &=&2^{n-3} - \frac{1}{8} \big( - \rho_n(\delta_{4,1}) \big) \ \ \text{if} \ n \equiv 1,3 \pmod{4}\\
F_2(n,1,0,1) &=&2^{n-3} - \frac{1}{8} \cdot \begin{cases}
\begin{array}{lr}
 - \rho_n(\delta_{4,1}) & \ \text{if} \ n \equiv 1 \pmod{4}\\
-2\rho_n(\delta_{2,1}) + \rho_n(\delta_{4,1}) & \ \text{if} \ n \equiv 3 \pmod{4}\\
\end{array}
\end{cases}\\
F_2(n,0,1,1) &=&2^{n-3} - \frac{1}{8} \big( -2\rho_n(\delta_{2,1}) + \rho_n(\delta_{4,1}) \big) \ \ \text{if} \ n \equiv 1,3 \pmod{4}\\
F_2(n,1,1,1) &=&2^{n-3} - \frac{1}{8} \cdot \begin{cases}
\begin{array}{lr}
 - \rho_n(\delta_{4,1}) & \ \text{if} \ n \equiv 1 \pmod{4}\\
2\rho_n(\delta_{2,1}) + \rho_n(\delta_{4,1}) & \ \text{if} \ n \equiv 3 \pmod{4}\\
\end{array}
\end{cases}\\
\end{eqnarray*}
}
\end{theorem}
The roots of $\delta_{2,1}$ are $\sqrt{2}\omega_{8}^3,\sqrt{2}\omega_{8}^5$, with $\omega_8 =  e^{i \pi/4} = (1 + i)/\sqrt{2}$,
while the roots of $\delta_{4,1}$ are $\sqrt{2}\omega_{24}^{5},\sqrt{2}\omega_{24}^{11}$,
$\sqrt{2}\omega_{24}^{13},\sqrt{2}\omega_{24}^{19}$,
with $\omega_{24} = e^{i \pi /12}  = ((1+\sqrt{3}) + (-1+\sqrt{3})i)/2\sqrt{2}$.
Observe that $F_2(n,t_1,t_2)$ can be obtained similarly, or by adding $F_2(n,t_1,t_2,0)$ and $F_2(n,t_1,t_2,1)$ as given in 
Theorem~\ref{thm:3traces}. Likewise $F_2(n,t_1)$ can be obtained as $F_2(n,t_1,0,0) + F_2(n,t_1,0,1)+ F_2(n,t_1,1,0) + F_2(n,t_1,1,1)$,
and summing all the expressions gives $2^n$, as they must.

\subsubsection{Direct method.} For odd $n$, applying Equations~(\ref{eq:T1linear}) to~(\ref{eq:T3linear}) we have
\begin{eqnarray}
\nonumber F_2(n,t_1,t_2,t_3) &=& \#\{a \in \F_{2^n} \mid T_1(a) = t_1, \ T_2(a) = t_2, \ T_3(a) = t_3\}\\
\nonumber &=& \frac{1}{2} \, \#\{a_0 \in \F_{2^n} \mid T_2(a_{0}^2 + a_{0} + t_1) = t_2, \ T_3(a_{0}^2 + a_{0} + t_1) = t_3 \}\\
\nonumber &=& \frac{1}{2} \, \#\{a_0 \in \F_{2^n} \mid T_1\Big(a_{0}^3 + a_{0} + t_1\binom{n}{2}\Big) = t_2, \ T_1\Big(a_{0}^5 + a_{0} + t_1\Big(a_{0}^3 + 
a_0+\binom{n}{3}\Big)\Big) = t_3 \}\\
\nonumber &=& \frac{1}{8} \, \#\{(a_0,a_1,a_2) \in (\F_{2^n})^3 \mid a_{1}^2 + a_1 = a_{0}^3 + a_{0} + t_1\binom{n}{2} + t_2,\\ 
\nonumber && \ \ \ \ a_{2}^2 + a_2 = a_{0}^5 + a_{0} + t_1\Big(a_{0}^3 + a_0+\binom{n}{3}\Big) + t_3\}.
\end{eqnarray}
These supersingular curves are all absolutely irreducible and are of genus $5$, and for $n$ odd their zeta functions reproduce Theorem~\ref{thm:3traces}.
Note that this method is slightly different to that used in~\cite[\S5]{AGGMY}, since it counts the number of points on a curve defined by 
an intersection of two curves (with one variable in common), rather than the sum of the number of points on three curves.
Letting $\delta_{2,2} = X^2 + 2$ (corresponding to the curve $C_{2,2}/\F_2: y^2 + y = x^3 + 1$), 
further note that for all $n \ge 3$ we have
\begin{eqnarray}
\nonumber F_2(n,0,0,0) &=& \frac{1}{8} \, \#\{ (a_0,a_1,a_2) \in (\F_{2^n})^3 \mid a_{1}^2 + a_{1}=a_{0}^3 + a_{0}, \ a_{2}^2 + a_{2} = a_{0}^5 + a_{0}\}\\
\nonumber  &=& \frac{1}{8}\big(2^{n} - 2\rho_n(\delta_{2,1}) - \rho_n(\delta_{2,2}) - \rho_n(\delta_{4,1})\big),
\end{eqnarray}
since one does not need to parameterise any `linear trace $= 1$' conditions. Using the same basic observations from~\cite[\S\S4\&5]{AGGMY} one can
determine the formulae, valid for all $n \ge 3$, for all four $F_2(n,0,t_2,t_3)$.

\subsection{Computing $F_2(n,t_1,t_2,t_3,t_4)$}

We begin by applying the indirect method.

\subsubsection{Indirect method.} Setting $\vec{f} = (T_4,T_3,T_2,T_1)$,  by the transform~(\ref{eq:Sinverse}) we have
\begin{equation}\label{eq:transform1_4}
{\small
\begin{bmatrix}
F_2(n,0,0,0,0)\\
F_2(n,1,0,0,0)\\
F_2(n,0,1,0,0)\\
F_2(n,1,1,0,0)\\
F_2(n,0,0,1,0)\\
F_2(n,1,0,1,0)\\
F_2(n,0,1,1,0)\\
F_2(n,1,1,1,0)\\
F_2(n,0,0,0,1)\\
F_2(n,1,0,0,1)\\
F_2(n,0,1,0,1)\\
F_2(n,1,1,0,1)\\
F_2(n,0,0,1,1)\\
F_2(n,1,0,1,1)\\
F_2(n,0,1,1,1)\\
F_2(n,1,1,1,1)\\
\end{bmatrix}
=
%\begin{bmatrix}
%N(0,0,0,0)\\
%N(0,0,0,1)\\
%N(0,0,1,0)\\
%N(0,0,1,1)\\
%N(0,1,0,0)\\
%N(0,1,0,1)\\
%N(0,1,1,0)\\
%N(0,1,1,1)\\
%N(1,0,0,0)\\
%N(1,0,0,1)\\
%N(1,0,1,0)\\
%N(1,0,1,1)\\
%N(1,1,0,0)\\
%N(1,1,0,1)\\
%N(1,1,1,0)\\
%N(1,1,1,1)\\
%\end{bmatrix}
\begin{bmatrix}
N(\vec{0})\\
N(\vec{1})\\
N(\vec{2})\\
N(\vec{3})\\
N(\vec{4})\\
N(\vec{5})\\
N(\vec{6})\\
N(\vec{7})\\
N(\vec{8})\\
N(\vec{9})\\
N(\vec{10})\\
N(\vec{11})\\
N(\vec{12})\\
N(\vec{13})\\
N(\vec{14})\\
N(\vec{15})\\
\end{bmatrix}
=
\frac{1}{8}
\begin{bmatrix*}[r]
-7 & 1 & 1 & 1 & 1 & 1 & 1 & 1 & 1 & 1 & 1 & 1 & 1 & 1 & 1 & 1\\
1 & -1 & 1 & -1 & 1 & -1 & 1 & -1 & 1 & -1 & 1 & -1 & 1 & -1 & 1 & -1\\
1 & 1 & -1 & -1 & 1 & 1 & -1 & -1 & 1 & 1 & -1 & -1 & 1 & 1 & -1 & -1\\
1 & -1 & -1 & 1 & 1 & -1 & -1 & 1 & 1 & -1 & -1 & 1 & 1 & -1 & -1 & 1\\
1 & 1 & 1 & 1 & -1 & -1 & -1 & -1 & 1 & 1 & 1 & 1 & -1 & -1 & -1 & -1\\
1 & -1 & 1 & -1 & -1 & 1 & -1 & 1 & 1 & -1 & 1 & -1 & -1 & 1 & -1 & 1\\
1 & 1 & -1 & -1 & -1 & -1 & 1 & 1 & 1 & 1 & -1 & -1 & -1 & -1 & 1 & 1\\
1 & -1 & -1 & 1 & -1 & 1 & 1 & -1 & 1 & -1 & -1 & 1 & -1 & 1 & 1 & -1\\
1 & 1 & 1 & 1 & 1 & 1 & 1 & 1 & -1 & -1 & -1 & -1 & -1 & -1 & -1 & -1\\
1 & -1 & 1 & -1 & 1 & -1 & 1 & -1 & -1 & 1 & -1 & 1 & -1 & 1 & -1 & 1\\
1 & 1 & -1 & -1 & 1 & 1 & -1 & -1 & -1 & -1 & 1 & 1 & -1 & -1 & 1 & 1\\
1 & -1 & -1 & 1 & 1 & -1 & -1 & 1 & -1 & 1 & 1 & -1 & -1 & 1 & 1 & -1\\
1 & 1 & 1 & 1 & -1 & -1 & -1 & -1 & -1 & -1 & -1 & -1 & 1 & 1 & 1 & 1\\
1 & -1 & 1 & -1 & -1 & 1 & -1 & 1 & -1 & 1 & -1 & 1 & 1 & -1 & 1 & -1\\
1 & 1 & -1 & -1 & -1 & -1 & 1 & 1 & -1 & -1 & 1 & 1 & 1 & 1 & -1 & -1\\
1 & -1 & -1 & 1 & -1 & 1 & 1 & -1 & -1 & 1 & 1 & -1 & 1 & -1 & -1 & 1\\
\end{bmatrix*}
\begin{bmatrix}
V(\vec{0} \cdot \vec{f})\\
V(\vec{1} \cdot \vec{f})\\
V(\vec{2} \cdot \vec{f})\\
V(\vec{3} \cdot \vec{f})\\
V(\vec{4} \cdot \vec{f})\\
V(\vec{5} \cdot \vec{f})\\
V(\vec{6} \cdot \vec{f})\\
V(\vec{7} \cdot \vec{f})\\
V(\vec{8} \cdot \vec{f})\\
V(\vec{9} \cdot \vec{f})\\
V(\vec{10} \cdot \vec{f})\\
V(\vec{11} \cdot \vec{f})\\
V(\vec{12} \cdot \vec{f})\\
V(\vec{13} \cdot \vec{f})\\
V(\vec{14} \cdot \vec{f})\\
V(\vec{15} \cdot \vec{f})\\
\end{bmatrix}.
}
\end{equation}
To determine $V( \vec{i} \cdot \vec{f})$ for $8 \le i \le 15$, we use Lemma~\ref{lem:T} part (4). 
In particular, setting $\alpha = a_{0}^2$ and $\beta = a_{0}$ for $r_0 = 0$, and $\alpha = a_{0}^2 + a_0$ and $\beta = 1$ for $r_0 = 1$, 
and evaluating mod $2$ gives the following:
\begin{eqnarray}
%T_4(a_{0}^2 + a_0) &=& T_2(a_{0}^3) + T_2(a_{0}) + T_1(a_{0}^3)T_1(a_0) + T_1(a_{0}^7 + a_{0}^5 + a_{0}^3),\\
\nonumber T_4(a_{0}^2 + a_0 + r_0) &=& T_2(a_{0}^3) + T_2(a_{0}) + T_1(a_{0}^3)T_1(a_0)\\
%&+& T_1\Big(a_{0}^7 + a_{0}^5 + r_0\Big(1 + \binom{n}{2}\Big)(a_{0}^3 + a_0) + a_{0}^3 + r_0\binom{n}{4}\Big).\\
 &+& T_1\Big(a_{0}^7 + a_{0}^5 + a_{0}^3 + r_0\Big(a_{0}^3 + a_0 + (a_{0}^3 + a_0)\binom{n}{2}  + \binom{n}{4}\Big)\Big).\label{eq:4coeffsparam1}
\end{eqnarray}
This can be linearised using the substitutions $a_0 = a_{1}^2 + a_1 + r_1$ and $a_{0}^3 = a_{2}^2 + a_2 + r_2$,
where $r_1,r_2 \in \F_2$ are the traces of $a_0$ and $a_{0}^3$ respectively. This results in\footnote[2]{See~\url{NewtonApproach_l_le_7.mw} to verify the linearised expressions for $4,5,6$ and $7$ coefficients.}
\begin{eqnarray*}
%T_4(a_{0}^2 + a_0) &=& T_1\Big(a_{2}^3 + a_2 + a_{1}^3 + a_1 + a_{0}^7 + a_{0}^5 + a_{0}^3 + \binom{n}{2}(r_0 + r_1) + r_0r_1 \Big), \\
\nonumber T_4(a_{0}^2 + a_0 + r_0) &=& T_1\Big(a_{2}^3 + a_2 + a_{1}^3 + a_1 + a_{0}^7 + a_{0}^5 + r_0r_1 + r_0r_2 + r_1r_2 + r_2\\
&+&(r_0 + 1)(r_1 + r_2)\binom{n}{2} + r_0\binom{n}{4} \Big).
\end{eqnarray*}
For $8 \le i \le 15$ let $\vec{i} = (i_3,i_2,i_1,i_0)$. The curves we are interested in are given by the following intersections:
\begin{eqnarray}
\nonumber a_{3}^2 + a_{3} &=& i_3\Big(a_{2}^3 + a_2 + a_{1}^3 + a_1 + a_{0}^7 + a_{0}^5 + r_0r_1 + r_0r_2 + r_1r_2 + r_2 + (r_0 + 1)(r_1 + r_2)\binom{n}{2} + r_0\binom{n}{4}\Big)\\
\nonumber &+& i_2\Big(a_{0}^5 + a_{0} + r_0\Big(a_{0}^3 + a_0 + \binom{n}{3} \Big)\Big) + i_1\Big(a_{0}^3 + a_{0} + r_0\binom{n}{2}\Big) + i_0r_0,\\
\label{eq:4trace1} a_0 &=& a_{1}^2 + a_1 + r_1,\\
\nonumber a_{0}^3 &=& a_{2}^2 + a_2 + r_2.
\end{eqnarray}
%while the four curves for $a$ of trace $1$ are given by:
%\begin{eqnarray*}
%a_{3}^2 + a_{3} &=& i_3\Big( a_{2}^3 + a_2 + a_{1}^3 + a_1 + a_{0}^7 + a_{0}^5 + \binom{n}{2}(a_{0}^3 + a_{0} + r_0 + r_1) + a_0 
%+ \binom{n}{4} + r_0r_1 \Big) \\
%&+& i_2\Big(a_{0}^5 + a_{0}^3 + \binom{n}{3}\Big) + i_1\Big(a_{0}^3 + a_0 + \binom{n}{2}\Big) + i_0 \binom{n}{1},\\
%a_0 &=& a_{1}^2 + a_1 + r_0,\\
%a_{0}^3 &=& a_{2}^2 + a_2 + r_1.
%\end{eqnarray*}
For $i_3 = 1$ the genus of all of these absolutely irreducible curves is $14$. Corollary~\ref{cor:period} implies that mod $2$ one has
\begin{eqnarray*}
\Big(\binom{n}{4},\binom{n}{3},\binom{n}{2},\binom{n}{1}\Big) \equiv \begin{cases}
\begin{array}{lr}
(0,0,0,1) & \ \text{if} \ n \equiv 1 \pmod{8}\\
(0,1,1,1) & \ \text{if} \ n \equiv 3 \pmod{8}\\
(1,0,0,1) & \ \text{if} \ n \equiv 5 \pmod{8}\\
(1,1,1,1) & \ \text{if} \ n \equiv 7 \pmod{8}\\
\end{array},
\end{cases}
\end{eqnarray*}
and hence there are four cases to consider when computing the zeta functions of each of the curves~(\ref{eq:4trace1}).
In order to express $F_2(n,t_1,t_2,t_3,t_4)$ compactly, we further define the following polynomials:
\begin{eqnarray*}
\delta_{8,1} &=& X^8 + 4X^7 + 6X^6 + 4X^5 + 2X^4 + 8X^3 + 24X^2 + 32X + 16,\\
\delta_{8,2} &=& X^8 + 2X^6 + 4X^5 + 2X^4 + 8X^3 + 8X^2 + 16.
\end{eqnarray*}
Note that by Theorem~\ref{thm:SS}, neither $\delta_{8,1}$ or $\delta_{8,2}$ are the 
characteristic polynomials of the Frobenius endomorphism of supersingular abelian varieties. 
There are two other polynomials which occur as factors of the characteristic polynomial of Frobenius of the above curves, 
but they are even polynomials and hence can be ignored for $n$ odd.
Using Magma to compute the zeta functions of the curves~(\ref{eq:4trace1}) and applying~(\ref{eq:transform1_4}) gives the following result. 

\begin{theorem}\label{thm:4traces}
For $n \ge 4$ we have
{\small
\begin{eqnarray*}
F_2(n,0,0,0,0) &=& 2^{n-4} - \frac{1}{16} \big( 4\rho_n(\delta_{2,1}) + \rho_n(\delta_{4,1}) + \rho_n(\delta_{8,1}) + \rho_n(\delta_{8,2}) \big) \ \ \text{if} \ n \equiv 1,3,5,7 \pmod{8}\\
F_2(n,1,0,0,0) &=& 2^{n-4} - \frac{1}{16} \cdot \begin{cases}
\begin{array}{lr}
4\rho_n(\delta_{2,1}) +  \rho_n(\delta_{4,1}) + \rho_n(\delta_{8,1}) + \rho_n(\delta_{8,2}) & \ \text{if} \ n \equiv 1 \pmod{8}\\
2\rho_n(\delta_{2,1}) -  \rho_n(\delta_{4,1}) - \rho_n(\delta_{8,1}) + \rho_n(\delta_{8,2}) & \ \text{if} \ n \equiv 3 \pmod{8}\\
\rho_n(\delta_{4,1}) - \rho_n(\delta_{8,1}) - \rho_n(\delta_{8,2}) & \ \text{if} \ n \equiv 5 \pmod{8}\\
- 2\rho_n(\delta_{2,1}) - \rho_n(\delta_{4,1}) + \rho_n(\delta_{8,1}) - \rho_n(\delta_{8,2}) & \ \text{if} \ n \equiv 7 \pmod{8}\\
\end{array}
\end{cases}\\
F_2(n,0,1,0,0) &=& 2^{n-4} - \frac{1}{16} \cdot \begin{cases}
\begin{array}{lr}
- 2\rho_n(\delta_{2,1}) - \rho_n(\delta_{4,1}) + \rho_n(\delta_{8,1}) - \rho_n(\delta_{8,2})  & \ \text{if} \ n \equiv 1,5 \pmod{8}\\
2\rho_n(\delta_{2,1}) - \rho_n(\delta_{4,1}) - \rho_n(\delta_{8,1}) + \rho_n(\delta_{8,2})  & \ \text{if} \ n \equiv 3,7 \pmod{8}\\
\end{array}
\end{cases}\\
F_2(n,1,1,0,0) &=& 2^{n-4} - \frac{1}{16} \cdot \begin{cases}
\begin{array}{lr}
-4\rho_n(\delta_{2,1}) + \rho_n(\delta_{4,1})  + \rho_n(\delta_{8,1}) + \rho_n(\delta_{8,2})  & \ \text{if} \ n \equiv 1 \pmod{8}\\
-2\rho_n(\delta_{2,1})  - \rho_n(\delta_{4,1}) + \rho_n(\delta_{8,1}) - \rho_n(\delta_{8,2})  & \ \text{if} \ n \equiv 3 \pmod{8}\\
\rho_n(\delta_{4,1}) - \rho_n(\delta_{8,1}) - \rho_n(\delta_{8,2})  & \ \text{if} \ n \equiv 5 \pmod{8}\\
2\rho_n(\delta_{2,1})  - \rho_n(\delta_{4,1}) - \rho_n(\delta_{8,1}) + \rho_n(\delta_{8,2})  & \ \text{if} \ n \equiv 7 \pmod{8}\\
\end{array}
\end{cases}\\
F_2(n,0,0,1,0) &=& 2^{n-4} - \frac{1}{16} \big(-2\rho_n(\delta_{2,1}) - \rho_n(\delta_{4,1}) + \rho_n(\delta_{8,1}) - \rho_n(\delta_{8,2} )\big)  \ \ \text{if} \ n \equiv 1,3,5,7 \pmod{8}\\
F_2(n,1,0,1,0) &=& 2^{n-4} - \frac{1}{16} \cdot \begin{cases}
\begin{array}{lr}
-2\rho_n(\delta_{2,1}) - \rho_n(\delta_{4,1}) + \rho_n(\delta_{8,1}) - \rho_n(\delta_{8,2}) & \ \text{if} \ n \equiv 1 \pmod{8}\\
-4\rho_n(\delta_{2,1}) + \rho_n(\delta_{4,1}) + \rho_n(\delta_{8,1}) + \rho_n(\delta_{8,2}) & \ \text{if} \ n \equiv 3 \pmod{8}\\
2\rho_n(\delta_{2,1}) - \rho_n(\delta_{4,1}) - \rho_n(\delta_{8,1}) + \rho_n(\delta_{8,2}) & \ \text{if}  \ n \equiv 5 \pmod{8}\\
\rho_n(\delta_{4,1}) - \rho_n(\delta_{8,1}) - \rho_n(\delta_{8,2}) & \ \text{if} \  n \equiv 7 \pmod{8}\\
\end{array}
\end{cases}\\
F_2(n,0,1,1,0) &=& 2^{n-4} - \frac{1}{16} \cdot \begin{cases}
\begin{array}{lr}
\rho_n(\delta_{4,1}) - \rho_n(\delta_{8,1}) - \rho_n(\delta_{8,2}) &   \ \text{if} \ n \equiv 1,5 \pmod{8}\\
-4\rho_n(\delta_{2,1}) + \rho_n(\delta_{4,1}) + \rho_n(\delta_{8,1}) + \rho_n(\delta_{8,2}) &   \ \text{if} \ n \equiv 3,7 \pmod{8}\\
\end{array}
\end{cases}\\
F_2(n,1,1,1,0) &=& 2^{n-4} - \frac{1}{16} \cdot \begin{cases}
\begin{array}{lr}
2\rho_n(\delta_{2,1}) - \rho_n(\delta_{4,1}) - \rho_n(\delta_{8,1}) + \rho_n(\delta_{8,2}) & \ \text{if} \ n \equiv 1 \pmod{8}\\
4\rho_n(\delta_{2,1}) + \rho_n(\delta_{4,1}) + \rho_n(\delta_{8,1}) + \rho_n(\delta_{8,2}) & \ \text{if} \ n \equiv 3 \pmod{8}\\
- 2\rho_n(\delta_{2,1}) - \rho_n(\delta_{4,1}) + \rho_n(\delta_{8,1}) - \rho_n(\delta_{8,2}) & \ \text{if} \ n \equiv 5 \pmod{8}\\
\rho_n(\delta_{4,1}) - \rho_n(\delta_{8,1}) - \rho_n(\delta_{8,2}) & \ \text{if} \ n \equiv 7 \pmod{8}\\
\end{array}
\end{cases}\\
F_2(n,0,0,0,1) &=& 2^{n-4} - \frac{1}{16} \big( \rho_n(\delta_{4,1}) - \rho_n(\delta_{8,1}) - \rho_n(\delta_{8,2}) \big)  \ \ \text{if} \ n \equiv 1,3,5,7 \pmod{8}\\
F_2(n,1,0,0,1) &=& 2^{n-4} - \frac{1}{16} \cdot \begin{cases}
\begin{array}{lr}
\rho_n(\delta_{4,1}) - \rho_n(\delta_{8,1}) - \rho_n(\delta_{8,2}) & \ \text{if}  \ n \equiv 1 \pmod{8}\\
-2\rho_n(\delta_{2,1}) - \rho_n(\delta_{4,1}) + \rho_n(\delta_{8,1}) - \rho_n(\delta_{8,2}) & \ \text{if} \ n \equiv 3 \pmod{8}\\
4\rho_n(\delta_{2,1}) + \rho_n(\delta_{4,1}) + \rho_n(\delta_{8,1}) + \rho_n(\delta_{8,2})  & \ \text{if} \ n \equiv 5 \pmod{8}\\
2\rho_n(\delta_{2,1}) - \rho_n(\delta_{4,1}) - \rho_n(\delta_{8,1}) + \rho_n(\delta_{8,2}) & \ \text{if} \ n \equiv 7 \pmod{8}\\
\end{array}
\end{cases}\\
F_2(n,0,1,0,1) &=& 2^{n-4} - \frac{1}{16} \cdot \begin{cases}
\begin{array}{lr}
2\rho_n(\delta_{2,1}) - \rho_n(\delta_{4,1}) - \rho_n(\delta_{8,1}) + \rho_n(\delta_{8,2}) & \ \text{if} \  n \equiv 1,5 \pmod{8}\\
-2\rho_n(\delta_{2,1}) - \rho_n(\delta_{4,1}) + \rho_n(\delta_{8,1}) - \rho_n(\delta_{8,2}) & \ \text{if}  \ n \equiv 3,7 \pmod{8}\\
\end{array}
\end{cases}\\
%\end{eqnarray*}
%}
%\vspace{-0mm}
%{\small
%\begin{eqnarray*}
F_2(n,1,1,0,1) &=& 2^{n-4} - \frac{1}{16} \cdot \begin{cases}
\begin{array}{lr}
\rho_n(\delta_{4,1}) - \rho_n(\delta_{8,1}) - \rho_n(\delta_{8,2}) & \ \text{if}  \ n \equiv 1 \pmod{8}\\
2\rho_n(\delta_{2,1}) - \rho_n(\delta_{4,1}) - \rho_n(\delta_{8,1}) + \rho_n(\delta_{8,2}) & \ \text{if} \ n \equiv 3 \pmod{8}\\
-4\rho_n(\delta_{2,1}) + \rho_n(\delta_{4,1}) + \rho_n(\delta_{8,1}) + \rho_n(\delta_{8,2})  & \ \text{if} \ n \equiv 5 \pmod{8}\\
-2\rho_n(\delta_{2,1}) - \rho_n(\delta_{4,1}) + \rho_n(\delta_{8,1}) - \rho_n(\delta_{8,2})  & \ \text{if} \ n \equiv 7 \pmod{8}\\
\end{array}
\end{cases}\\
F_2(n,0,0,1,1) &=& 2^{n-4} - \frac{1}{16} \big( 2\rho_n(\delta_{2,1}) - \rho_n(\delta_{4,1}) - \rho_n(\delta_{8,1}) + \rho_n(\delta_{8,2}) \big) \ \ \text{if} \ n \equiv 1,3,5,7 \pmod{8}\\
F_2(n,1,0,1,1) &=& 2^{n-4} - \frac{1}{16} \cdot \begin{cases}
\begin{array}{lr}
2\rho_n(\delta_{2,1}) - \rho_n(\delta_{4,1}) - \rho_n(\delta_{8,1}) + \rho_n(\delta_{8,2}) & \ \text{if}  \ n \equiv 1 \pmod{8}\\
\rho_n(\delta_{4,1}) - \rho_n(\delta_{8,1}) - \rho_n(\delta_{8,2}) & \ \text{if} \ n \equiv 3 \pmod{8}\\
- 2\rho_n(\delta_{2,1}) - \rho_n(\delta_{4,1}) + \rho_n(\delta_{8,1}) - \rho_n(\delta_{8,2})  & \ \text{if} \ n \equiv 5 \pmod{8}\\
- 4\rho_n(\delta_{2,1}) + \rho_n(\delta_{4,1}) + \rho_n(\delta_{8,1}) + \rho_n(\delta_{8,2})  & \ \text{if} \ n \equiv 7 \pmod{8}\\  
\end{array}
\end{cases}\\
%\end{eqnarray*}
%}
%{\small
%\begin{eqnarray*}
F_2(n,0,1,1,1) &=& 2^{n-4} - \frac{1}{16} \cdot \begin{cases}
\begin{array}{lr}
-4\rho_n(\delta_{2,1}) + \rho_n(\delta_{4,1}) + \rho_n(\delta_{8,1}) + \rho_n(\delta_{8,2}) & \ \text{if} \  n \equiv 1,5 \pmod{8}\\
\rho_n(\delta_{4,1}) - \rho_n(\delta_{8,1}) - \rho_n(\delta_{8,2}) & \ \text{if}  \ n \equiv 3,7 \pmod{8}\\
\end{array}
\end{cases}\\
F_2(n,1,1,1,1) &=& 2^{n-4} - \frac{1}{16} \cdot \begin{cases}
\begin{array}{lr}
- 2\rho_n(\delta_{2,1}) - \rho_n(\delta_{4,1}) + \rho_n(\delta_{8,1}) - \rho_n(\delta_{8,2})  & \ \text{if}  \ n \equiv 1 \pmod{8}\\
\rho_n(\delta_{4,1}) - \rho_n(\delta_{8,1}) - \rho_n(\delta_{8,2}) & \ \text{if} \ n \equiv 3 \pmod{8}\\
2\rho_n(\delta_{2,1}) - \rho_n(\delta_{4,1}) - \rho_n(\delta_{8,1}) + \rho_n(\delta_{8,2})  & \ \text{if} \ n \equiv 5 \pmod{8}\\
4\rho_n(\delta_{2,1}) + \rho_n(\delta_{4,1}) + \rho_n(\delta_{8,1}) + \rho_n(\delta_{8,2})  & \ \text{if} \ n \equiv 7 \pmod{8}\\  
\end{array}
\end{cases}\\
\end{eqnarray*}
}
\end{theorem}
One can check that the roots of $\delta_{8,1}$ are $\alpha_1,\alpha_2,\alpha_3,\alpha_4$ and their complex conjugates 
$\overline{\alpha_1},\overline{\alpha_2},\overline{\alpha_3},\overline{\alpha_4}$, where:
\begin{eqnarray*}
\alpha_1 &=& -\frac{1}{2}+\frac{\sqrt{2}}{4}\Big(1+\sqrt{7+4\sqrt{2}}\Big)-\frac{\sqrt{2}}{4}\Big(1+\sqrt{5-2\sqrt{2}}\Big)i,\\
\alpha_2 &=& -\frac{1}{2}+\frac{\sqrt{2}}{4}\Big(1-\sqrt{7+4\sqrt{2}}\Big)-\frac{\sqrt{2}}{4}\Big(1-\sqrt{5-2\sqrt{2}}\Big)i,\\
\alpha_3 &=& -\frac{1}{2}-\frac{\sqrt{2}}{4}\Big(1 + \frac{1}{\sqrt{17}}(3\sqrt{2}-1)\sqrt{5-2\sqrt{2}}\Big) + \frac{\sqrt{2}}{4}\Big(1  - \frac{1}{\sqrt{17}}(3\sqrt{2}-1)\sqrt{7 + 4\sqrt{2}}\Big)i,\\ 
%\alpha_3 &=& -\frac{1}{2}-\frac{\sqrt{2}}{4} - \frac{\sqrt{34}}{68}(3\sqrt{2}-1)\sqrt{5-2\sqrt{2}} + \Big(\frac{\sqrt{2}}{4}  - \frac{\sqrt{34}}{68}(3\sqrt{2}-1)\sqrt{7 + 4\sqrt{2}}\Big)i,\\ 
\alpha_4 &=& -\frac{1}{2}-\frac{\sqrt{2}}{4}\Big(1 - \frac{1}{\sqrt{17}}(3\sqrt{2}-1)\sqrt{5-2\sqrt{2}}\Big) + \frac{\sqrt{2}}{4}\Big(1 + \frac{1}{\sqrt{17}}(3\sqrt{2}-1)\sqrt{7 + 4\sqrt{2}}\Big)i.\\
%\alpha_4 &=& -\frac{1}{2}-\frac{\sqrt{2}}{4} + \frac{\sqrt{34}}{68}(3\sqrt{2}-1)\sqrt{5-2\sqrt{2}} + \Big(\frac{\sqrt{2}}{4}  + \frac{\sqrt{34}}{68}(3\sqrt{2}-1)\sqrt{7 + 4\sqrt{2}}\Big)i. 
\end{eqnarray*}
One can also check that the roots of $\delta_{8,2}$ are $i\alpha_1,i\alpha_2,i\alpha_3,i\alpha_4$ and $\overline{i\alpha_1},\overline{i\alpha_2},\overline{i\alpha_3},\overline{i\alpha_4}$.
In Theorem~\ref{thm:4traces} the formulae for each $F_2(n,t_1,t_2,t_3,t_4)$ and each odd $n$ mod $8$ have non-supersingular terms of the form $\pm \rho_n(\delta_{8,1}) \pm \rho_n(\delta_{8,2})$ or $\pm \rho_n(\delta_{8,1}) \mp \rho_n(\delta_{8,2})$.
A simple application of Kronecker's theorem to the phases of these non-supersingular Weil numbers allows one to deduce that the formulae are not periodic in $n$. We leave stating what abelian varieties $\delta_{8,1}$ and $\delta_{8,2}$ are the characteristic polynomials of Frobenius for as an open problem, but note that they must be isogenous over $\F_{2^4}$.

\subsubsection*{An alternative parameterisation.}\label{sec:4coeffsparam2}

We now present an alternative parameterisation of Eq.~(\ref{eq:4coeffsparam1}) which requires one less variable to linearise. 
In particular, we have
\begin{eqnarray}
\nonumber T_4(a_{0}^2 + a_0 + r_0) &=& T_2(a_{0}^3) + T_2(a_{0}) + T_1(a_{0}^3)T_1(a_0)\\
\nonumber &+& T_1\Big(a_{0}^7 + a_{0}^5 + a_{0}^3 + r_0\Big(a_{0}^3 + a_0 + (a_{0}^3 + a_0)\binom{n}{2}  + \binom{n}{4}\Big)\Big)\\
\label{eq:T4linearalt} &=& T_2(a_{0}^3 + a_{0}) + T_1\Big(a_{0}^7 + a_{0}^5 + a_{0}^3 + a_{0} + r_0\Big(a_{0}^3 + a_0 + (a_{0}^3 + a_0)\binom{n}{2}+\binom{n}{4}\Big)\Big),
\end{eqnarray}
where the second equality follows from Lemma~\ref{lem:T} part (2), by setting $\alpha = a_{0}^3$ and $\beta = a_{0}$. 
This can be linearised using the substitution $a_{0}^3 + a_{0} = a_{1}^2 + a_{1} + r_1$, where $r_1$ is the trace of $a_{0}^3 + a_{0}$.
This results in 
\begin{eqnarray*}
\nonumber T_4(a_{0}^2 + a_0 + r_0) &=& T_1\Big( a_{1}^3 + a_{1} + a_{0}^7 + a_{0}^5 + a_{0}^3 + a_{0}\\
&+& r_0\Big(a_{0}^3 + a_0 + (a_{0}^3 + a_0)\binom{n}{2}  + \binom{n}{4}\Big) + r_1\binom{n}{2}\Big).
\end{eqnarray*}
For $8 \le i \le 15$ let $\vec{i} = (i_3,i_2,i_1,i_0)$. The curves we are interested in are given by the following intersections:
\begin{eqnarray}
\nonumber a_{2}^2 + a_{2} &=& i_3\Big( a_{1}^3 + a_{1} + a_{0}^7 + a_{0}^5 + a_{0}^3 + a_{0} + r_0\Big(a_{0}^3 + a_0 + (a_{0}^3 + a_0)\binom{n}{2}  + \binom{n}{4}\Big) + r_1\binom{n}{2}\Big)\\
\label{4coeffsaltparamintersection} &+& i_2\Big(a_{0}^5 + a_{0} + r_0\Big(a_{0}^3 + a_0 + \binom{n}{3} \Big)\Big) + i_1\Big(a_{0}^3 + a_{0} + r_0\binom{n}{2}\Big) + i_0r_0,\\
\nonumber a_{0}^3 + a_{0} &=& a_{1}^2 + a_1 + r_1.
\end{eqnarray}
For $i_3 = 1$ the genus of all of these absolutely irreducible curves is $7$ -- rather than $14$ as in the first parameterisation -- and 
therefore the characteristic polynomials of Frobenius have lower degrees than before. Indeed, there are fewer even polynomials appearing as 
factors and those that are not even occur to lower powers. Nevertheless, once~(\ref{eq:transform1_4}) is applied one again obtains the formulae 
given in Theorem~\ref{thm:4traces}.

\subsubsection{Direct method.}\label{sec:4coeffsdirect}
For odd $n$, applying Equations~(\ref{eq:T1linear}) to~(\ref{eq:T3linear}) and~(\ref{eq:T4linearalt}) we have
\begin{eqnarray}
\nonumber F_2(n,t_1,t_2,t_3,t_4) &=& \#\{a \in \F_{2^n} \mid T_1(a) = t_1, \ T_2(a) = t_2, \ T_3(a) = t_3, \ T_4(a) = t_4\}\\
\nonumber &=& \frac{1}{2} \, \#\{a_0 \in \F_{2^n} \mid T_2(a_{0}^2 + a_{0} + t_1) = t_2, \ T_3(a_{0}^2 + a_{0} + t_1) = t_3,  
\ T_4(a_{0}^2 + a_{0} + t_1) = t_4\}\\
\nonumber &=& \frac{1}{2} \, \#\{a_0 \in \F_{2^n} \mid T_1\Big(a_{0}^3 + a_{0} + t_1\binom{n}{2}\Big) = t_2, \ T_1\Big(a_{0}^5 + a_{0} + t_1\Big(a_{0}^3 + 
a_0+\binom{n}{3}\Big)\Big) = t_3,\\
\nonumber && \ \ \ \ T_2(a_{0}^3 + a_{0}) + T_1\Big(a_{0}^7 + a_{0}^5 + a_{0}^3 + a_{0} + t_1\Big(a_{0}^3 + a_0 + (a_{0}^3 + a_0)\binom{n}{2}+\binom{n}{4}\Big)\Big) = t_4\}\\
\nonumber &=& \frac{1}{16} \, \#\{(a_0,a_1,a_2,a_3) \in (\F_{2^n})^4 \mid a_{1}^2 + a_1 = a_{0}^3 + a_{0} + t_1\binom{n}{2} + t_2,\\ 
\nonumber && \ \ \ \ a_{2}^2 + a_2 = a_{0}^5 + a_{0} + t_1\Big(a_{0}^3 + a_0+\binom{n}{3}\Big) + t_3, \ a_{3}^2 + a_3 = a_{1}^3 + a_1 + (t_1 + t_2)\binom{n}{2} + \\
\nonumber && \ \ \ \  a_{0}^7 + a_{0}^5 + a_{0}^3 + a_{0} + t_1\Big(a_{0}^3 + a_0 
+ (a_{0}^3 + a_{0})\binom{n}{2} +\binom{n}{4}\Big) + t_4\},
\end{eqnarray}
where we have used Lemma~\ref{lem:T} part (2) and the parameterisation of $T_2(a) = t_2$ to compute $T_2(a_{0}^3 + a_{0}) = T_2(a_{1}^2 + a_{1} + t_1\binom{n}{2} + t_2) = T_2(a_{1}^2 + a_{1}) + T_2\Big(t_1\binom{n}{2} + t_2\Big) + T_1(a_{1}^2 + a_{1})T_1\Big(t_1\binom{n}{2} + t_2\Big) + T_1\Big(\Big(t_1\binom{n}{2} + t_2\Big)(a_{1}^2 + a_1)\Big) = T_1(a_{1}^3 + a_{1}) + (t_1 + t_2)\binom{n}{2}$, since the last two terms are zero.
These curves are all absolutely irreducible and of genus $17$, and for $n$ odd their zeta functions reproduce Theorem~\ref{thm:4traces}.

\subsubsection{General formulae.}\label{subsubsec:generalformula}
We have the following result.

\begin{theorem}\label{thm:4coeffsgeneral}
For $n \ge 4$ we have
\begin{eqnarray*}
F_2(n,0,0,0,0) &=& 2^{n-4} - \frac{1}{16}\big(4\rho_n(\delta_{2,1}) + 3\rho_n(\delta_{2,2}) + \rho_n(\delta_{4,1}) + \rho_n(\delta_{8,1}) + \rho_n(\delta_{8,2})\big),\\
F_2(n,0,0,0,1) &=& 2^{n-4} - \frac{1}{16}\big(-\rho_n(\delta_{2,2}) + \rho_n(\delta_{4,1}) - \rho_n(\delta_{8,1}) - \rho_n(\delta_{8,2})\big),\\
F_2(n,0,0,1,0) &=& 2^{n-4} - \frac{1}{16}\big(-2\rho_n(\delta_{2,1}) + \rho_n(\delta_{2,2}) - \rho_n(\delta_{4,1}) + \rho_n(\delta_{8,1})	- \rho_n(\delta_{8,2})\big),\\
F_2(n,0,0,1,1) &=& 2^{n-4} - \frac{1}{16}\big(2\rho_n(\delta_{2,1}) - 3\rho_n(\delta_{2,2}) - \rho_n(\delta_{4,1}) - \rho_n(\delta_{8,1}) + \rho_n(\delta_{8,2})\big).
\end{eqnarray*}
\end{theorem}

\begin{proof}
Using the direct method, for all $n \ge 4$ we have
\begin{eqnarray}
\nonumber F_2(n,0,0,0,0) &=& \frac{1}{16} \, \#\{ (a_{0},a_{1},a_{2},a_{3}) \in (\F_{2^n})^4 \mid a_{1}^2 + a_{1} = a_{0}^3 + a_{0},\  a_{2}^2 + a_{2} =a_{0}^5 + a_{0},\\
\nonumber && \ \ \ \ a_{3}^2 + a_{3} = a_{1}^3 + a_{1} + a_{0}^7 + a_{0}^5 + a_{0}^3 + a_0\}\\
\nonumber &=& 2^{n-4} - \frac{1}{16}\big(4\rho_n(\delta_{2,1}) + 3\rho_n(\delta_{2,2}) + \rho_n(\delta_{4,1}) + \rho_n(\delta_{8,1}) + \rho_n(\delta_{8,2})\big),
\end{eqnarray}
since one does not need to parameterise any `linear trace $= 1$' conditions. Since $F_2(n,0,0,0) = F_2(n,0,0,0,0) + F_2(n,0,0,0,1)$, for all 
$n \ge 4$ we also have: 
\begin{eqnarray}
\nonumber F_2(n,0,0,0,1) &=& 2^{n-4} - \frac{1}{16}\big(-\rho_n(\delta_{2,2}) + \rho_n(\delta_{4,1}) - \rho_n(\delta_{8,1}) - \rho_n(\delta_{8,2})\big).
\end{eqnarray}
Omitting the $T_3(a)$ condition, we obtain:
\begin{eqnarray}
\nonumber F_2(n,0,0,*,0) &=& 
%\#\{a \in \F_{2^n} \mid T_1(a) = 0, \ T_2(a) = 0, \ T_4(a) = 0\}\\
%\nonumber &=& \frac{1}{2} \, \#\{ a_{0} \in \F_{2^n} \mid T_2(a_{0}^2 + a_{0}) = 0, \ T_4(a_{0}^2 + a_{0}) = 0\}\\
%\nonumber &=& \frac{1}{2} \, \#\{ a_{0} \in \F_{2^n} \mid T_1(a_{0}^3 + a_{0}) = 0, \ T_2(a_{0}^3 + a_{0}) + T_1( a_{0}^7 + a_{0}^5 + a_{0}^3 + a_0)=0\}\\
\frac{1}{8} \, \#\{ (a_{0},a_{1},a_{2}) \in (\F_{2^n})^3 \mid a_{1}^2 + a_{1} = a_{0}^3 + a_{0}, \ a_{2}^2 + a_{2} = a_{1}^3 + a_{1} + a_{0}^7 + a_{0}^5 + a_{0}^3 + a_0\}\\
&=& 2^{n-3} - \frac{1}{8}\big(\rho_n(\delta_{2,1}) + 2\rho_n(\delta_{2,2}) + \rho_n(\delta_{8,1})\big),\label{eq:kloosterman}
\end{eqnarray}
this intersection describing an absolutely irreducible curve of genus $7$. Since $F_2(n,0,0,*,0) = F_2(n,0,0,0,0) + F_2(n,0,0,1,0)$ we also have: 
\begin{eqnarray}
\nonumber F_2(n,0,0,1,0) &=& 2^{n-4} - \frac{1}{16}\big(-2\rho_n(\delta_{2,1}) + \rho_n(\delta_{2,2}) - \rho_n(\delta_{4,1}) + \rho_n(\delta_{8,1})	- \rho_n(\delta_{8,2})\big).
\end{eqnarray}
Moreover, since $F_2(n,0,0) = F_2(n,0,0,0,0) + F_2(n,0,0,0,1) + F_2(n,0,0,1,0) + F_2(n,0,0,1,1)$ we have
\begin{eqnarray}
\nonumber F_2(n,0,0,1,1) &=& 2^{n-4} - \frac{1}{16}\big(2\rho_n(\delta_{2,1}) - 3\rho_n(\delta_{2,2}) - \rho_n(\delta_{4,1}) - \rho_n(\delta_{8,1}) + \rho_n(\delta_{8,2})\big). 
\end{eqnarray}\qed

\end{proof}

Note that only the $F_2(n,0,0,0,0)$ formula comes from the characteristic polynomial of Frobenius of a curve, since the other three have terms with
the wrong sign. Note also that the total degree of the corresponding polynomials (numerator degree plus denominator degree) is $2 \cdot 17 = 34$ only
for $F_2(n,0,0,0,0)$. This does not contradict the fact that the direct method always produces curves of genus $17$, since the direct method in general
only represents $F_2(n,t_1,t_2,t_3,t_4)$ for $n$ odd, so there are cancellations, and furthermore all the featured `+' signs in the formulae for 
$F_2(n,0,0,t_3,t_4)$ may be replaced by `-' signs once the featured polynomials $\delta(X)$ are replaced with $\delta(-X)$.

If one similarly tries to omit the $T_2(a)$ condition then one can not automatically simplify the $T_2(a_{0}^3 + a_0)$ term which arises from the condition 
$T_4(a_{0}^2 + a_{0}) = 0$; one is forced to condition on whether the linear trace of $a_{0}^3 + a_0$ is $0$ or $1$, in which case one needs $n$ to be
odd in order to parameterise the latter condition. Therefore, it is apparently not possible to find formulae for all $n \ge 4$ for 
$F_2(n,0,1,t_3,t_4)$ with this approach.
Nevertheless, we expect that similar formulae hold for all $n \ge 4$ for each $F_2(n,t_1,t_2,t_3,t_4)$, with additional terms arising from the $n$-th 
powers of roots of a set of even polynomials. Furthermore, if for a given $F_2(n,t_1,t_2,t_3,t_4)$ the coefficients of the various $\rho_n(\delta_i)$ 
also depend on the residue of $n$ mod $8$, as they do for $n$ odd, then one can Fourier analyse the coefficients in order to express them in terms 
of the complex $8$-th roots of unity to obtain a single formula, as in~\cite[Prop.~3\&5]{AGGMY}.
%(note that one can also Fourier analyse the coefficients of each 
%$\rho_n(p_i)$ appearing in the formulae for each $F_2(n,t_1,t_2,t_3,t_4)$ for $n$ odd, expressing each as a linear combination of
% $1,i^{(n-1)/2},(-1)^{(n-1)/2},(-i)^{(n-1)/2}$). 
So while the formulae themselves are not periodic in $n$, it may be that the coefficients of each $\rho_n(\delta_i)$ featured in each 
$F_2(n,t_1,t_2,t_3,t_4)$ are periodic in $n$.

\subsubsection{An alternative proof of general formulae.}\label{subsubsec:altgeneralformulae}
We now provide another proof of Theorem~\ref{thm:4coeffsgeneral} which relies on the alternative parameterisation 
of $T_4(a_{0}^2 + a_{0})$ and a generalisation of the transform~(\ref{eq:Sinverse}).
Observe that the functions $f_0,\ldots,f_{m-1}: \F_{2^n} \rightarrow \F_2$ in~\S\ref{sec:transform1char2} may be replaced by functions 
$f_0,\ldots,f_{m-1}: A \rightarrow \F_2$ where $A$ is any relevant domain -- and thus for instance any algebraic set -- and exactly the same 
argument holds, provided that $V(\vec{0} \cdot \vec{f})$ is redefined to be $|A|$.

The conditions $T_1(a) = T_2(a) = 0$ imply that we should set $A = \{(a_0,a_1) \in (\F_{2^n})^2 \mid a_{0}^3 + a_{0} = a_{1}^2 + a_{1}\}$.
Furthermore let $\vec{f} = (T_4(a_{0}^2 + a_{0}), T_3(a_{0}^2 + a_{0})) = (T_1(a_{1}^3 + a_1 + a_{0}^7 + a_{0}^5 + a_{0}^3 + a_{0}), T_1(a_{0}^5 + a_0))$.
%Since
%\begin{eqnarray}
%\nonumber F_2(n,0,0,t_3,t_4) &=& \frac{1}{4} \, \#\{(a_0,a_1) \in (\F_{2^n})^2 \mid a_{0}^3 + a_{0} = a_{1}^2 + a_{1}, \ 
%T_3(a_{0}^2 + a_{0}) = t_3, \ T_4(a_{0}^2 + a_{0}) = t_4\}\\
%\end{eqnarray}
Then by the stated generalisation we have
\begin{equation}\label{eq:alternativeproof}
{\small
\begin{bmatrix}
F_2(n,0,0,0,0)\\
F_2(n,0,0,1,0)\\
F_2(n,0,0,0,1)\\
F_2(n,0,0,1,1)\\
\end{bmatrix}
=
\begin{bmatrix}
N(\vec{0})\\
N(\vec{1})\\
N(\vec{2})\\
N(\vec{3})\\
\end{bmatrix}
=
\frac{1}{8}
\begin{bmatrix*}[r]
-1 & 1 & 1 & 1\\
1 & -1 & 1 & -1\\
1 & 1 & -1 & -1\\
1 & -1 & -1 & 1\\
\end{bmatrix*}
\begin{bmatrix}
V(\vec{0} \cdot \vec{f})\\
V(\vec{1} \cdot \vec{f})\\
V(\vec{2} \cdot \vec{f})\\
V(\vec{3} \cdot \vec{f})\\
\end{bmatrix},
}
\end{equation}
where we have a factor of $1/8$ rather than $1/2$ because of the two factors of $1/2$ arising from the introduction of the variables $a_0$ and $a_1$ defining $A$. Note that $V(\vec{0} \cdot \vec{f}) = |A| = 2^n - \rho_n(\delta_{2,1})$. For $1 \le i \le 3$ let $\vec{i} = (i_1,i_0)$. We thus have:
\begin{equation*}
V(\vec{i} \cdot \vec{f}) = \frac{1}{2}\#\{(a_0,a_1,a_2) \mid a_{0}^3 + a_{0} = a_{1}^2 + a_{1}, \ a_{2}^2 + a_{2} = i_1(a_{1}^3 + a_1 + 
a_{0}^7 + a_{0}^5 + a_{0}^3 + a_{0}) + i_0(a_{0}^5 + a_{0}) \}.
\end{equation*}
For $i = 1,2$ and $3$ these are absolutely irreducible curves of genus $5,7$ and $7$ respectively, and thus their zeta functions 
are easier to compute than for those curves arising from the direct method, which have genus $17$. Combining them as per Eq.~(\ref{eq:alternativeproof}) 
gives the formulae in~\S\ref{subsubsec:generalformula}; it is therefore a slightly more streamlined argument than the one given there.

Another way of viewing this approach is to take the intersections of~(\ref{4coeffsaltparamintersection}) and set $r_0 = r_1 = i_0 = i_1 = 0$ and observe
that the $i_2 = i_3 = 0$ condition gives $V(\vec{0} \cdot \vec{f}) = |A|$, without the factor $1/2$ present for the other three $V(\vec{i} \cdot \vec{f})$ 
because there is no $a_2$ variable in this case. The approach is therefore more useful in these cases than the direct method, which can not 
produce formulae for all $F_2(n,0,0,t_3,t_4)$ for all $n \ge 4$, as it only works for $n$ odd in general.

%------------------------------------------------------------------------------------------

\subsubsection{Connection with binary Kloosterman sums.}
The binary Kloosterman sum $\mathcal{K}_{2^n}: \F_{2^n} \rightarrow \Z$ can be defined by 
\[
\mathcal{K}_{2^n}(a) = 1 + \sum_{x \in \F_{2^n}^{\times}} (-1)^{T_1(x^{-1} + ax)}.
\]
Kloosterman sums have applications in cryptography and coding theory, see for example~\cite{gong,moisiocode}. 
In particular, zeros of $\mathcal{K}_{2^{n}}$ lead to bent functions from $\F_{2^{2n}} \rightarrow \F_{2}$~\cite{dillon}.
The following elementary lemma connects Kloosterman sums to a family of elliptic curves.

\begin{lemma}[\cite{wolfmann}]\label{lis1}
Let $a \in \F_{2^n}^{\times}$ and define the elliptic curve $E_{2^n}(a)$ over
$\F_{2^n}$ by 
\[
E_{2^n}(a): y^2 + xy = x^3 + a.
\]
Then $\#E_{2^n}(a) = 2^n + \mathcal{K}_{2^n}(a)$.
\end{lemma}
Computing Kloosterman sum zeros is generally regarded as being difficult, currently taking exponential time (in $n$) to find a single non-trivial 
($a \ne 0$) zero.
Besides the deterministic test due to Ahmadi and Granger~\cite{KloostermanAG}, which computes the cardinality of the Sylow $2$-subgroup of any 
$E_{2^n}(a)$ via point-halving, and thus by Lemma~\ref{lis1} the maximum power of $2$ dividing $\mathcal{K}_{2^n}(a)$,
research has focused on characterising Kloosterman sums modulo small integers~\cite{moisio2,lisonek,lisonek2,helleseth2,lisonek3,faruk1,faruk2,faruk3}. 
In order to analyse the expected running time of the algorithm of Ahmadi and Granger, it is necessary to know the distribution of Kloosterman sums 
which are divisible by successive powers of $2$. Table~\ref{dist2} presents this distribution for $n \le 13$, which was also presented 
in~\cite{KloostermanAG}. 

Let $T(n,k)$ denote the $(n,k)$-th entry of Table~\ref{dist2}, \ie the number of $a \in \F_{2^n}^{\times}$ for which $\#E_{2^n}(a)$ is divisible 
by $2^{k}$.
By using a result of Katz and Livn\'e~\cite{katz} it is possible to express $T(n,k)$ in terms of the class numbers of certain imaginary 
quadratic fields. However, it remains an open problem to give exact formulae for $k > 4$, with the formulae for the first four columns being as follows. 
Since the orders of all of the elliptic curves in Lemma~\ref{lis1} are divisible by $4$, one has $T(n,1) = T(n,2) = 2^n - 1$. 
One can show that $E_{2^n}(a)$ has a point of order $8$ if and only if $T_1(a) = 0$ (see e.g.~\cite{geer}), hence $T_{2^n}(3) = 2^{n-1}-1$. 
Finally, Lison\v{e}k and Moisio proved that $T(n,4) = (2^n - (-1 + i)^n - (-1 - i)^n)/4 - 1$, connecting it with the number of points on a supersingular
elliptic curve~\cite[Theorem 3.6]{lisonek3}.

\begin{table}\caption{$T(n,k) = \#\{a \in \F_{2^n}^{\times} \mid \#E_{2^n}(a) \equiv 0 \pmod{2^k} \}$}\label{dist2}
\begin{center}
%\begin{tabularx}{0.75\textwidth}{c|r|r|r|r|r|r|r|r|r|r|r|r|r|}
\begin{tabularx}{0.75\textwidth}{c *{13}{|Y}}
%\begin{tabular}{c|r|r|r|r|r|r|r|r|r|r|r|r|r|}
%\hline
$n \backslash k$ & $1$ & $2$    & $3$   & $4$   & $5$  & $6$  & $7$ & $8$ &
$9$ & $10$ & $11$ & $12$ & $13$\\
\hline
1  &  1       &   1     &        &       &       &      &      &     &  & & & &    \\
2  &  3       &   3    &        &       &       &      &      &     &   & & & &   \\
3  &  7       &   7    &     3  &       &       &      &      &     &  & & & &    \\
4  &  15      &   15   &     7  &   5   &       &      &      &     &  & & & &    \\
5  &  31      &   31   &    15  &   5   &   5   &      &      &     &  & & & &    \\
6  &  63      &   63   &    31  &  15   &  12   &  12  &      &     &   & & & &   \\
7  & 127      &  127   &    63  &  35   &  14   &  14  & 14   &     &   & & & &   \\
8  & 255      &  255   &   127  &  55   &  21   &  16  & 16   & 16  &  & & & &    \\
9  & 511      &  511   &   255  & 135   &  63   &  18  & 18   & 18  & 18 & & & &\\
10 & 1023     & 1023   &   511  & 255   & 125   &  65  & 60   & 60  & 60 & 60 & & &\\
11 & 2047     & 2047   &  1023  & 495   & 253   & 132  & 55   & 55  & 55 & 55 & 55 & & \\
12 & 4095     & 4095   &  2047  & 1055  & 495   & 252  & 84   & 72  & 72 & 72 & 72 & 72 &\\
13 & 8191     & 8191   &  4095  & 2015  & 1027  & 481  & 247  & 52  & 52 & 52
& 52 & 52 &  52\\
%\hline
\end{tabularx}
\end{center}
\end{table}

The following theorem connects the distribution of binary Kloosterman sums mod $32$ to the distribution of the first four coefficients of 
the characteristic polynomial.

\begin{theorem}[\cite{faruk3}]\label{thm:Kloosterman256}
Let $a \in \F_{2^n}$ with $n \ge 4$ and let $e_1,\ldots,e_4$ be the coefficients of the characteristic polynomial of $a$, regarded as integers.
Then 
\[
\mathcal{K}_{2^n}(a) \equiv 28e_1 +8e_2 + 16(e_1 e_2 + e_1 e_3 + e_4) \pmod{32}.
\]
\end{theorem}
Combining Theorems~\ref{thm:4traces} and~\ref{thm:Kloosterman256} therefore provides explicit formulae for the distribution of binary Kloosterman sums 
mod $32$, for $n$ odd. Furthermore, combining Theorem~\ref{thm:Kloosterman256} with Eq.~(\ref{eq:kloosterman}) provides an explicit formula for 
$\#\{a \in \F_{2^n} \mid \mathcal{K}_{2^n}(a) \equiv 0 \pmod{32}\} = T(n,5) + 1$; indeed, this connection was our original motivation for considering 
the first-four prescribed traces problem.
\begin{corollary}
For $n \ge 5$ we have
\begin{eqnarray}
\nonumber \#\{a \in \F_{2^n} \mid \mathcal{K}_{2^n}(a) \equiv 0 \pmod{32}\} &=& F_2(n,0,0,0,0) + F_2(n,0,0,1,0)\\ 
\nonumber &=&  2^{n-3} - \frac{1}{8}\big(\rho_n(\delta_{2,1}) + 2\rho_n(\delta_{2,2}) + \rho_n(\delta_{8,1})\big).
\end{eqnarray} 
\end{corollary}
We also have:

\begin{theorem}[\cite{faruk3}]\label{thm:Kloosterman64}
Let $a \in \F_{2^n}$ with $n \ge 6$ and let $e_1,\ldots,e_8$ be the coefficients of the characteristic polynomial of $a$, regarded as integers.
Then 
\begin{eqnarray*}
\mathcal{K}_{2^n}(a) &\equiv& 28e_1 + 40e_2 + 16(e_1 e_2 + e_1 e_3 + e_4)\\ 
&+& 32(e_1e_4 + e_1e_5 + e_1e_6 + e_1e_7 + e_2e_3 + e_2e_4 + e_2e_6 + 
e_3e_5 + e_1e_2e_3 + e_1e_2e_4 + e_8) \pmod{64}.
\end{eqnarray*}
\end{theorem}
Therefore, if one could solve the first-eight prescribed traces problem, then one could also determine a formula for the entries of the sixth column of 
Table~\ref{dist2}. 

%----------------------------------------------------------------

\subsection{Computing $F_2(n,t_1,t_2,t_3,t_4,t_5)$}

We begin by applying the indirect method.

\subsubsection{Indirect method.} Setting $\vec{f} = (T_5,T_4,T_3,T_2,T_1)$, by the transform~(\ref{eq:Sinverse}) we have
\[
{\tiny
\begin{bmatrix}
F_2(n,0,0,0,0,0)\\
F_2(n,1,0,0,0,0)\\
F_2(n,0,1,0,0,0)\\
F_2(n,1,1,0,0,0)\\
F_2(n,0,0,1,0,0)\\
F_2(n,1,0,1,0,0)\\
F_2(n,0,1,1,0,0)\\
F_2(n,1,1,1,0,0)\\
F_2(n,0,0,0,1,0)\\
F_2(n,1,0,0,1,0)\\
F_2(n,0,1,0,1,0)\\
F_2(n,1,1,0,1,0)\\
F_2(n,0,0,1,1,0)\\
F_2(n,1,0,1,1,0)\\
F_2(n,0,1,1,1,0)\\
F_2(n,1,1,1,1,0)\\
F_2(n,0,0,0,0,1)\\
F_2(n,1,0,0,0,1)\\
F_2(n,0,1,0,0,1)\\
F_2(n,1,1,0,0,1)\\
F_2(n,0,0,1,0,1)\\
F_2(n,1,0,1,0,1)\\
F_2(n,0,1,1,0,1)\\
F_2(n,1,1,1,0,1)\\
F_2(n,0,0,0,1,1)\\
F_2(n,1,0,0,1,1)\\
F_2(n,0,1,0,1,1)\\
F_2(n,1,1,0,1,1)\\
F_2(n,0,0,1,1,1)\\
F_2(n,1,0,1,1,1)\\
F_2(n,0,1,1,1,1)\\
F_2(n,1,1,1,1,1)\\
\end{bmatrix}
=
\begin{bmatrix}
N(\vec{0})\\
N(\vec{1})\\
N(\vec{2})\\
N(\vec{3})\\
N(\vec{4})\\
N(\vec{5})\\
N(\vec{6})\\
N(\vec{7})\\
N(\vec{8})\\
N(\vec{9})\\
N(\vec{10})\\
N(\vec{11})\\
N(\vec{12})\\
N(\vec{13})\\
N(\vec{14})\\
N(\vec{15})\\
N(\vec{16})\\
N(\vec{17})\\
N(\vec{18})\\
N(\vec{19})\\
N(\vec{20})\\
N(\vec{21})\\
N(\vec{22})\\
N(\vec{23})\\
N(\vec{24})\\
N(\vec{25})\\
N(\vec{26})\\
N(\vec{27})\\
N(\vec{28})\\
N(\vec{29})\\
N(\vec{30})\\
N(\vec{31})\\
\end{bmatrix}
=
\frac{1}{16}
\begin{bmatrix*}[c]
-15& 1 & 1 & 1 & 1 & 1 & 1 & 1 & 1 & 1 & 1 & 1 & 1 & 1 & 1 & 1 & 1 & 1 & 1 & 1 & 1 & 1 & 1 & 1 & 1 & 1 & 1 & 1 & 1 & 1 & 1 & 1\\
1 & - & 1 & - & 1 & - & 1 & - & 1 & - & 1 & - & 1 & - & 1 & - & 1 & - & 1 & - & 1 & - & 1 & - & 1 & - & 1 & - & 1 & - & 1 & -\\
1 & 1 & - & - & 1 & 1 & - & - & 1 & 1 & - & - & 1 & 1 & - & - & 1 & 1 & - & - & 1 & 1 & - & - & 1 & 1 & - & - & 1 & 1 & - & -\\
1 & - & - & 1 & 1 & - & - & 1 & 1 & - & - & 1 & 1 & - & - & 1 & 1 & - & - & 1 & 1 & - & - & 1 & 1 & - & - & 1 & 1 & - & - & 1\\
1 & 1 & 1 & 1 & - & - & - & - & 1 & 1 & 1 & 1 & - & - & - & - & 1 & 1 & 1 & 1 & - & - & - & - & 1 & 1 & 1 & 1 & - & - & - & -\\
1 & - & 1 & - & - & 1 & - & 1 & 1 & - & 1 & - & - & 1 & - & 1 & 1 & - & 1 & - & - & 1 & - & 1 & 1 & - & 1 & - & - & 1 & - & 1\\
1 & 1 & - & - & - & - & 1 & 1 & 1 & 1 & - & - & - & - & 1 & 1 & 1 & 1 & - & - & - & - & 1 & 1 & 1 & 1 & - & - & - & - & 1 & 1\\
1 & - & - & 1 & - & 1 & 1 & - & 1 & - & - & 1 & - & 1 & 1 & - & 1 & - & - & 1 & - & 1 & 1 & - & 1 & - & - & 1 & - & 1 & 1 & -\\
1 & 1 & 1 & 1 & 1 & 1 & 1 & 1 & - & - & - & - & - & - & - & - & 1 & 1 & 1 & 1 & 1 & 1 & 1 & 1 & - & - & - & - & - & - & - & -\\
1 & - & 1 & - & 1 & - & 1 & - & - & 1 & - & 1 & - & 1 & - & 1 & 1 & - & 1 & - & 1 & - & 1 & - & - & 1 & - & 1 & - & 1 & - & 1\\
1 & 1 & - & - & 1 & 1 & - & - & - & - & 1 & 1 & - & - & 1 & 1 & 1 & 1 & - & - & 1 & 1 & - & - & - & - & 1 & 1 & - & - & 1 & 1\\
1 & - & - & 1 & 1 & - & - & 1 & - & 1 & 1 & - & - & 1 & 1 & - & 1 & - & - & 1 & 1 & - & - & 1 & - & 1 & 1 & - & - & 1 & 1 & -\\
1 & 1 & 1 & 1 & - & - & - & - & - & - & - & - & 1 & 1 & 1 & 1 & 1 & 1 & 1 & 1 & - & - & - & - & - & - & - & - & 1 & 1 & 1 & 1\\
1 & - & 1 & - & - & 1 & - & 1 & - & 1 & - & 1 & 1 & - & 1 & - & 1 & - & 1 & - & - & 1 & - & 1 & - & 1 & - & 1 & 1 & - & 1 & -\\
1 & 1 & - & - & - & - & 1 & 1 & - & - & 1 & 1 & 1 & 1 & - & - & 1 & 1 & - & - & - & - & 1 & 1 & - & - & 1 & 1 & 1 & 1 & - & -\\
1 & - & - & 1 & - & 1 & 1 & - & - & 1 & 1 & - & 1 & - & - & 1 & 1 & - & - & 1 & - & 1 & 1 & - & - & 1 & 1 & - & 1 & - & - & 1\\
%%%%%%%%%%%%%%%%%%%%%%%%%%%%%%%%%%%%%%%%%%%%%%%%%%%%%%%%%%%%%%%
1 & 1 & 1 & 1 & 1 & 1 & 1 & 1 & 1 & 1 & 1 & 1 & 1 & 1 & 1 & 1 & - & - & - & - & - & - & - & - & - & - & - & - & - & - & - & -\\
1 & - & 1 & - & 1 & - & 1 & - & 1 & - & 1 & - & 1 & - & 1 & - & - & 1 & - & 1 & - & 1 & - & 1 & - & 1 & - & 1 & - & 1 & - & 1\\
1 & 1 & - & - & 1 & 1 & - & - & 1 & 1 & - & - & 1 & 1 & - & - & - & - & 1 & 1 & - & - & 1 & 1 & - & - & 1 & 1 & - & - & 1 & 1\\
1 & - & - & 1 & 1 & - & - & 1 & 1 & - & - & 1 & 1 & - & - & 1 & - & 1 & 1 & - & - & 1 & 1 & - & - & 1 & 1 & - & - & 1 & 1 & -\\
1 & 1 & 1 & 1 & - & - & - & - & 1 & 1 & 1 & 1 & - & - & - & - & - & - & - & - & 1 & 1 & 1 & 1 & - & - & - & - & 1 & 1 & 1 & 1\\
1 & - & 1 & - & - & 1 & - & 1 & 1 & - & 1 & - & - & 1 & - & 1 & - & 1 & - & 1 & 1 & - & 1 & - & - & 1 & - & 1 & 1 & - & 1 & -\\
1 & 1 & - & - & - & - & 1 & 1 & 1 & 1 & - & - & - & - & 1 & 1 & - & - & 1 & 1 & 1 & 1 & - & - & - & - & 1 & 1 & 1 & 1 & - & -\\
1 & - & - & 1 & - & 1 & 1 & - & 1 & - & - & 1 & - & 1 & 1 & - & - & 1 & 1 & - & 1 & - & - & 1 & - & 1 & 1 & - & 1 & - & - & 1\\
1 & 1 & 1 & 1 & 1 & 1 & 1 & 1 & - & - & - & - & - & - & - & - & - & - & - & - & - & - & - & - & 1 & 1 & 1 & 1 & 1 & 1 & 1 & 1\\
1 & - & 1 & - & 1 & - & 1 & - & - & 1 & - & 1 & - & 1 & - & 1 & - & 1 & - & 1 & - & 1 & - & 1 & 1 & - & 1 & - & 1 & - & 1 & -\\
1 & 1 & - & - & 1 & 1 & - & - & - & - & 1 & 1 & - & - & 1 & 1 & - & - & 1 & 1 & - & - & 1 & 1 & 1 & 1 & - & - & 1 & 1 & - & -\\
1 & - & - & 1 & 1 & - & - & 1 & - & 1 & 1 & - & - & 1 & 1 & - & - & 1 & 1 & - & - & 1 & 1 & - & 1 & - & - & 1 & 1 & - & - & 1\\
1 & 1 & 1 & 1 & - & - & - & - & - & - & - & - & 1 & 1 & 1 & 1 & - & - & - & - & 1 & 1 & 1 & 1 & 1 & 1 & 1 & 1 & - & - & - & -\\
1 & - & 1 & - & - & 1 & - & 1 & - & 1 & - & 1 & 1 & - & 1 & - & - & 1 & - & 1 & 1 & - & 1 & - & 1 & - & 1 & - & - & 1 & - & 1\\
1 & 1 & - & - & - & - & 1 & 1 & - & - & 1 & 1 & 1 & 1 & - & - & - & - & 1 & 1 & 1 & 1 & - & - & 1 & 1 & - & - & - & - & 1 & 1\\
1 & - & - & 1 & - & 1 & 1 & - & - & 1 & 1 & - & 1 & - & - & 1 & - & 1 & 1 & - & 1 & - & - & 1 & 1 & - & - & 1 & - & 1 & 1 & -\\
\end{bmatrix*}
\begin{bmatrix}
V(\vec{0} \cdot \vec{f})\\
V(\vec{1} \cdot \vec{f})\\
V(\vec{2} \cdot \vec{f})\\
V(\vec{3} \cdot \vec{f})\\
V(\vec{4} \cdot \vec{f})\\
V(\vec{5} \cdot \vec{f})\\
V(\vec{6} \cdot \vec{f})\\
V(\vec{7} \cdot \vec{f})\\
V(\vec{8} \cdot \vec{f})\\
V(\vec{9} \cdot \vec{f})\\
V(\vec{10} \cdot \vec{f})\\
V(\vec{11} \cdot \vec{f})\\
V(\vec{12} \cdot \vec{f})\\
V(\vec{13} \cdot \vec{f})\\
V(\vec{14} \cdot \vec{f})\\
V(\vec{15} \cdot \vec{f})\\
V(\vec{16} \cdot \vec{f})\\
V(\vec{17} \cdot \vec{f})\\
V(\vec{18} \cdot \vec{f})\\
V(\vec{19} \cdot \vec{f})\\
V(\vec{20} \cdot \vec{f})\\
V(\vec{21} \cdot \vec{f})\\
V(\vec{22} \cdot \vec{f})\\
V(\vec{23} \cdot \vec{f})\\
V(\vec{24} \cdot \vec{f})\\
V(\vec{25} \cdot \vec{f})\\
V(\vec{26} \cdot \vec{f})\\
V(\vec{27} \cdot \vec{f})\\
V(\vec{28} \cdot \vec{f})\\
V(\vec{29} \cdot \vec{f})\\
V(\vec{30} \cdot \vec{f})\\
V(\vec{31} \cdot \vec{f})\\
\end{bmatrix},}
\]
where for reasons of space and in common with Hadamard matrix notation, we represent each $-1$ simply as a `$-$'.
To determine $V( \vec{i} \cdot \vec{f})$ for $16 \le i \le 31$, we use Lemma~\ref{lem:T} part (5). 
In particular, setting $\alpha = a_{0}^2$ and $\beta = a_{0}$ for $r_0 = 0$, and $\alpha = a_{0}^2 + a_0$ and $\beta = 1$ for $r_0 = 1$, 
and evaluating mod $2$ gives the following:
\begin{eqnarray}
%T_5(a_{0}^2 + a_0) &=& T_1(a_{0})T_1(a_{0}^3) + T_1(a_0)T_1(a_{0}^5) + T_1(a_{0}^3)T_1(a_{0}^5) + T_1(a_{0}^9 + a_{0}^3 + a_0),\\
%T_5(a_{0}^2 + a_0 + 1) &=& T_2(a_0) + T_2(a_{0}^3) +  T_1(a_{0})T_1(a_{0}^5) + T_1(a_{0}^3)T_1(a_{0}^5) \\
%&+& T_1\Big(a_{0}^9 + a_{0}^7 + a_{0}^3 + a_{0} + \binom{n}{2}(a_{0}^5 + a_0) + \binom{n}{3}(a_{0}^3 + a_0) + \binom{n}{5}\Big).
\nonumber T_5(a_{0}^2 + a_0 + r_0) &=& T_1(a_{0}^5)T_1(a_{0}^3) + T_1(a_{0}^5)T_1(a_{0}) + T_1(a_{0}^3)T_1(a_{0}) + r_0(T_2(a_{0}^3) + T_2(a_{0}) + T_1(a_{0}^3)T_1(a_{0}))\\
&+& T_1\Big(a_{0}^9 + a_{0}^3 + a_{0} + r_0 \Big(a_{0}^7 + (a_{0}^5 + a_0)\binom{n}{2} + (a_{0}^3 + a_0)\binom{n}{3} + \binom{n}{5}\Big)\Big).\label{eq:5coeffsparam1}
\end{eqnarray}
This can be linearised using the same the substitutions that were used for the four coefficient case, namely, 
$a_0 = a_{1}^2 + a_1 + r_1$ and $a_{0}^3 = a_{2}^2 + a_2 + r_2$, where $r_1,r_2 \in \F_2$ are the traces of $a_0$ and $a_{0}^3$ respectively. 
This results in 
\begin{eqnarray*}
T_5(a_{0}^2 + a_0 + r_0) &=& T_1\Big(r_0(a_{2}^3 + a_{2} + a_{1}^3 + a_{1}) + a_{0}^9 + r_0a_{0}^7+ (r_1 + r_2)a_{0}^5 
+ r_1 + r_2 + r_1r_2 + r_0r_1r_2\\
&+& (r_0a_{0}^5 + r_0r_2)\binom{n}{2} + (r_0r_1 + r_0r_2)\binom{n}{3} + r_0\binom{n}{5}\Big).
%T_5(a_{0}^2 + a_0) &=& T_1( a_{0}^9 + a_{0}^3 + a_0 + a_{0}^5(r_0 + r_1) + r_0 r_1),\\
%T_5(a_{0}^2 + a_0 + 1) &=& T_1\Big( a_{2}^3 + a_2 + a_{1}^3 + a_1 + a_{0}^9 + a_{0}^7 + a_{0}^3 + a_{0} + a_{0}^5(r_0 + r_1) 
%+ \binom{n}{2}(a_{0}^5 + a_0 + r_0 + r_1)\\
%&+& \binom{n}{3}(a_{0}^3 + a_0) + \binom{n}{5} \Big).
\end{eqnarray*}
For $16 \le i \le 31$ let $\vec{i} = (i_4,i_3,i_2,i_1,i_0)$. The curves we are interested in are given by the following intersections:
\begin{eqnarray}
\nonumber a_{3}^2 + a_{3} &=& 
\nonumber i_4\Big(r_0(a_{2}^3 + a_{2} + a_{1}^3 + a_{1}) + a_{0}^9 + r_0a_{0}^7 + (r_1 + r_2)a_{0}^5 + r_1 + r_2 + r_1r_2 + r_0r_1r_2\\
\nonumber &+& (r_0a_{0}^5 + r_0r_2)\binom{n}{2} + (r_0r_1 + r_0r_2)\binom{n}{3} + r_0\binom{n}{5}\Big)\\
\nonumber &+& i_3\Big(a_{2}^3 + a_2 + a_{1}^3 + a_1 + a_{0}^7 + a_{0}^5 + (r_0 + 1)(r_1 + r_2)\binom{n}{2} + r_0r_1 + r_0r_2 + r_1r_2 + r_2 + r_0\binom{n}{4}\Big)\\
\nonumber &+& i_2\Big(a_{0}^5 + a_{0} + r_0\Big(a_{0}^3 + a_0 + \binom{n}{3} \Big)\Big) + i_1\Big(a_{0}^3 + a_{0} + r_0\binom{n}{2}\Big) + i_0r_0,\\
\label{eq:5trace1} a_0 &=& a_{1}^2 + a_1 + r_1,\\
\nonumber a_{0}^3 &=& a_{2}^2 + a_2 + r_2.
%&+& i_3\Big(a_{2}^3 + a_2 + a_{1}^3 + a_1 + a_{0}^7 + a_{0}^5 + a_{0}^3 + \binom{n}{2}(r_0 + r_1) + r_0r_1\Big) + i_2(a_{0}^5 + a_0) 
%+ i_1(a_{0}^3 + a_0),\\
%a_0 &=& a_{1}^2 + a_1 + r_0,\\
%a_{0}^3 &=& a_{2}^2 + a_2 + r_1,
%\end{eqnarray*}
%while the four curves for $a$ of trace $1$ are given by:
%\begin{eqnarray*}
%a_{3}^2 + a_{3} &=& 
%i_4\Big(  a_{2}^3 + a_2 + a_{1}^3 + a_1 + a_{0}^9 + a_{0}^7 + a_{0}^3 + a_{0} + a_{0}^5(r_0 + r_1) + \binom{n}{2}(a_{0}^5 + a_0 + r_0 + r_1)\\
%&+& \binom{n}{3}(a_{0}^3 + a_0) + \binom{n}{5} \Big) \\
%&+& i_3\Big( a_{2}^3 + a_2 + a_{1}^3 + a_1 + a_{0}^7 + a_{0}^5 + \binom{n}{2}(a_{0}^3 + a_{0} + r_0 + r_1) + a_0 
%+ \binom{n}{4} + r_0r_1 \Big) \\
%&+& i_2\Big(a_{0}^5 + a_{0}^3 + \binom{n}{3}\Big) + i_1\Big(a_{0}^3 + a_0 + \binom{n}{2}\Big) + i_0 \binom{n}{1},\\
%a_0 &=& a_{1}^2 + a_1 + r_0,\\
%a_{0}^3 &=& a_{2}^2 + a_2 + r_0.
\end{eqnarray}
For each $16 \le i \le 31$ the genus of all of these absolutely irreducible curves is $18$.
Corollary~\ref{cor:period} implies that mod $2$ one has
\begin{eqnarray*}
\Big(\binom{n}{5},\binom{n}{4},\binom{n}{3},\binom{n}{2},\binom{n}{1}\Big) \equiv \begin{cases}
\begin{array}{lr}
(0,0,0,0,1) & \ \text{if} \ n \equiv 1 \pmod{8}\\
(0,0,1,1,1) & \ \text{if} \ n \equiv 3 \pmod{8}\\
(1,1,0,0,1) & \ \text{if} \ n \equiv 5 \pmod{8}\\
(1,1,1,1,1) & \ \text{if} \ n \equiv 7 \pmod{8}\\
\end{array},
\end{cases}
\end{eqnarray*}
and hence there are four cases to consider when computing the zeta functions of each of the curves~(\ref{eq:5trace1}).
In order to express $F_2(n,t_1,t_2,t_3,t_4,t_5)$ compactly, we define the following polynomial:
\begin{eqnarray*}
\delta_{8,3} &=& X^8 + 2X^7 + 2X^6 - 4X^4 + 8X^2 + 16X + 16,
\end{eqnarray*}
which is the characteristic polynomial of Frobenius of the supersingular curve 
\[
C_{8,3}/\F_2: y^2 + y = x^9 + x^5.
\]
As with the four coefficient case there are several other even polynomials which occur as factors of the characteristic polynomial of Frobenius
of the above curves, which can hence be ignored for $n$ odd.
 
We used Magma V22.2-3 to compute the zeta functions of the curves~(\ref{eq:5trace1}), which took just under $15$ minutes on a 
2.0GHz AMD Opteron computer and leads to the following theorem\footnote[2]{See~\url{F2(n,t1,t2,t3,t4,t5).m}.}.

\begin{theorem}\label{thm:5traces}
For $n \ge 5$ we have
{\small
\begin{eqnarray*}
F_2(n,0,0,0,0,0) &=& 2^{n-5} - \frac{1}{32} \big( 5\rho_n(\delta_{2,1}) + 5\rho_n(\delta_{4,1}) + 2\rho_n(\delta_{8,1}) + \rho_n(\delta_{8,3}) \big) \ \ \text{if} \ n \equiv 1,3,5,7 \pmod{8}\\
F_2(n,1,0,0,0,0) &=& 2^{n-5} - \frac{1}{32} \cdot \begin{cases}
\begin{array}{lr}
5\rho_n(\delta_{2,1}) +  5\rho_n(\delta_{4,1}) + 2\rho_n(\delta_{8,1}) + \rho_n(\delta_{8,3}) & \ \text{if} \ n \equiv 1 \pmod{8}\\
3\rho_n(\delta_{2,1}) -  3\rho_n(\delta_{4,1}) + 2\rho_n(\delta_{8,2}) - \rho_n(\delta_{8,3}) & \ \text{if} \ n \equiv 3 \pmod{8}\\
\rho_n(\delta_{2,1}) +  \rho_n(\delta_{4,1}) - 2\rho_n(\delta_{8,1}) + \rho_n(\delta_{8,3}) & \ \text{if} \ n \equiv 5 \pmod{8}\\
-\rho_n(\delta_{2,1}) +  \rho_n(\delta_{4,1}) - 2\rho_n(\delta_{8,2}) - \rho_n(\delta_{8,3}) & \ \text{if} \ n \equiv 7 \pmod{8}\\
\end{array}
\end{cases}\\
F_2(n,0,1,0,0,0) &=& 2^{n-5} - \frac{1}{32} \cdot \begin{cases}
\begin{array}{lr}
-\rho_n(\delta_{2,1}) + \rho_n(\delta_{4,1}) -2\rho_n(\delta_{8,2}) - \rho_n(\delta_{8,3})  & \ \text{if} \ n \equiv 1,5 \pmod{8}\\
3\rho_n(\delta_{2,1}) - 3\rho_n(\delta_{4,1}) + 2\rho_n(\delta_{8,2}) - \rho_n(\delta_{8,3})  & \ \text{if} \ n \equiv 3,7 \pmod{8}\\
\end{array}
\end{cases}\\
F_2(n,1,1,0,0,0) &=& 2^{n-5} - \frac{1}{32} \cdot \begin{cases}
\begin{array}{lr}
-3\rho_n(\delta_{2,1}) - 3\rho_n(\delta_{4,1})  + 2\rho_n(\delta_{8,1}) + \rho_n(\delta_{8,3})  & \ \text{if} \ n \equiv 1 \pmod{8}\\
-\rho_n(\delta_{2,1})  + \rho_n(\delta_{4,1}) - 2\rho_n(\delta_{8,2}) - \rho_n(\delta_{8,3})  & \ \text{if} \ n \equiv 3 \pmod{8}\\
\rho_n(\delta_{2,1})  + \rho_n(\delta_{4,1}) - 2\rho_n(\delta_{8,1}) + \rho_n(\delta_{8,3})  & \ \text{if} \ n \equiv 5 \pmod{8}\\
3\rho_n(\delta_{2,1})  - 3\rho_n(\delta_{4,1}) + 2\rho_n(\delta_{8,2}) - \rho_n(\delta_{8,3})  & \ \text{if} \ n \equiv 7 \pmod{8}\\
\end{array}
\end{cases}\\
F_2(n,0,0,1,0,0) &=& 2^{n-5} - \frac{1}{32} \big(-\rho_n(\delta_{2,1})+ \rho_n(\delta_{4,1}) -2\rho_n(\delta_{8,2}) -\rho_n(\delta_{8,3} )\big)  \ \ \text{if} \ n \equiv 1,3,5,7 \pmod{8}\\
F_2(n,1,0,1,0,0) &=& 2^{n-5} - \frac{1}{32} \cdot \begin{cases}
\begin{array}{lr}
-3\rho_n(\delta_{2,1}) - 3\rho_n(\delta_{4,1}) + 2 \rho_n(\delta_{8,1})+ \rho_n(\delta_{8,3}) & \ \text{if} \ n \equiv 1 \pmod{8}\\
-5\rho_n(\delta_{2,1}) +5\rho_n(\delta_{4,1}) + 2\rho_n(\delta_{8,2}) - \rho_n(\delta_{8,3}) & \ \text{if} \ n \equiv 3 \pmod{8}\\
\rho_n(\delta_{2,1}) + \rho_n(\delta_{4,1}) - 2\rho_n(\delta_{8,1}) + \rho_n(\delta_{8,3}) & \ \text{if}  \ n \equiv 5 \pmod{8}\\
-\rho_n(\delta_{2,1}) + \rho_n(\delta_{4,1}) - 2\rho_n(\delta_{8,2}) - \rho_n(\delta_{8,3}) & \ \text{if} \  n \equiv 7 \pmod{8}\\
\end{array}
\end{cases}\\
F_2(n,0,1,1,0,0) &=& 2^{n-5} - \frac{1}{32} \cdot \begin{cases}
\begin{array}{lr}
- \rho_n(\delta_{2,1}) + \rho_n(\delta_{4,1}) - 2\rho_n(\delta_{8,2}) - \rho_n(\delta_{8,3}) &   \ \text{if} \ n \equiv 1,5 \pmod{8}\\
-5\rho_n(\delta_{2,1}) + 5\rho_n(\delta_{4,1}) + 2\rho_n(\delta_{8,2}) - \rho_n(\delta_{8,3}) &   \ \text{if} \ n \equiv 3,7 \pmod{8}\\
\end{array}
\end{cases}\\
F_2(n,1,1,1,0,0) &=& 2^{n-5} - \frac{1}{32} \cdot \begin{cases}
\begin{array}{lr}
3\rho_n(\delta_{2,1}) - 3\rho_n(\delta_{4,1}) + 2\rho_n(\delta_{8,2}) - \rho_n(\delta_{8,3}) & \ \text{if} \ n \equiv 1 \pmod{8}\\
5\rho_n(\delta_{2,1}) + 5\rho_n(\delta_{4,1}) + 2\rho_n(\delta_{8,1}) + \rho_n(\delta_{8,3}) & \ \text{if} \ n \equiv 3 \pmod{8}\\
- \rho_n(\delta_{2,1}) + \rho_n(\delta_{4,1}) - 2\rho_n(\delta_{8,2}) - \rho_n(\delta_{8,3}) & \ \text{if} \ n \equiv 5 \pmod{8}\\
\rho_n(\delta_{2,1}) + \rho_n(\delta_{4,1}) - 2\rho_n(\delta_{8,1}) + \rho_n(\delta_{8,3}) & \ \text{if} \ n \equiv 7 \pmod{8}\\
\end{array}
\end{cases}\\
F_2(n,0,0,0,1,0) &=& 2^{n-5} - \frac{1}{32} \big( \rho_n(\delta_{2,1}) + \rho_n(\delta_{4,1}) - 2\rho_n(\delta_{8,1}) + \rho_n(\delta_{8,3}) \big)  \ \ \text{if} \ n \equiv 1,3,5,7 \pmod{8}\\
F_2(n,1,0,0,1,0) &=& 2^{n-5} - \frac{1}{32} \cdot \begin{cases}
\begin{array}{lr}
-\rho_n(\delta_{2,1}) + \rho_n(\delta_{4,1}) -2\rho_n(\delta_{8,2}) - \rho_n(\delta_{8,3}) & \ \text{if}  \ n \equiv 1 \pmod{8}\\
-3\rho_n(\delta_{2,1}) - 3\rho_n(\delta_{4,1}) + 2\rho_n(\delta_{8,1}) + \rho_n(\delta_{8,3}) & \ \text{if} \ n \equiv 3 \pmod{8}\\
3\rho_n(\delta_{2,1}) - 3\rho_n(\delta_{4,1}) + 2\rho_n(\delta_{8,2}) -\rho_n(\delta_{8,3})  & \ \text{if} \ n \equiv 5 \pmod{8}\\
\rho_n(\delta_{2,1}) + \rho_n(\delta_{4,1}) - 2\rho_n(\delta_{8,1}) + \rho_n(\delta_{8,3}) & \ \text{if} \ n \equiv 7 \pmod{8}\\
\end{array}
\end{cases}\\
F_2(n,0,1,0,1,0) &=& 2^{n-5} - \frac{1}{32} \cdot \begin{cases}
\begin{array}{lr}
3\rho_n(\delta_{2,1}) -3\rho_n(\delta_{4,1}) + 2\rho_n(\delta_{8,2}) -\rho_n(\delta_{8,3}) & \ \text{if} \  n \equiv 1,5 \pmod{8}\\
-\rho_n(\delta_{2,1}) + \rho_n(\delta_{4,1}) - 2\rho_n(\delta_{8,2}) -\rho_n(\delta_{8,3}) & \ \text{if}  \ n \equiv 3,7 \pmod{8}\\
\end{array}
\end{cases}\\
F_2(n,1,1,0,1,0) &=& 2^{n-5} - \frac{1}{32} \cdot \begin{cases}
\begin{array}{lr}
- \rho_n(\delta_{2,1}) + \rho_n(\delta_{4,1}) - 2\rho_n(\delta_{8,2}) - \rho_n(\delta_{8,3}) & \ \text{if}  \ n \equiv 1 \pmod{8}\\
\rho_n(\delta_{2,1}) + \rho_n(\delta_{4,1}) - 2\rho_n(\delta_{8,1}) + \rho_n(\delta_{8,3}) & \ \text{if} \ n \equiv 3 \pmod{8}\\
-5\rho_n(\delta_{2,1}) +5\rho_n(\delta_{4,1}) + 2\rho_n(\delta_{8,2}) - \rho_n(\delta_{8,3})  & \ \text{if} \ n \equiv 5 \pmod{8}\\
-3\rho_n(\delta_{2,1}) -3\rho_n(\delta_{4,1}) + 2\rho_n(\delta_{8,1}) + \rho_n(\delta_{8,3})  & \ \text{if} \ n \equiv 7 \pmod{8}\\
\end{array}
\end{cases}\\
F_2(n,0,0,1,1,0) &=& 2^{n-5} - \frac{1}{32} \big( 3\rho_n(\delta_{2,1}) - 3\rho_n(\delta_{4,1}) + 2\rho_n(\delta_{8,2}) - \rho_n(\delta_{8,3}) \big) \ \ \text{if} \ n \equiv 1,3,5,7 \pmod{8}\\
F_2(n,1,0,1,1,0) &=& 2^{n-5} - \frac{1}{32} \cdot \begin{cases}
\begin{array}{lr}
3\rho_n(\delta_{2,1}) - 3\rho_n(\delta_{4,1}) + 2\rho_n(\delta_{8,2}) - \rho_n(\delta_{8,3}) & \ \text{if}  \ n \equiv 1 \pmod{8}\\
\rho_n(\delta_{2,1}) + \rho_n(\delta_{4,1}) - 2\rho_n(\delta_{8,1}) + \rho_n(\delta_{8,3}) & \ \text{if} \ n \equiv 3 \pmod{8}\\
- \rho_n(\delta_{2,1}) + \rho_n(\delta_{4,1}) - 2\rho_n(\delta_{8,2}) - \rho_n(\delta_{8,3})  & \ \text{if} \ n \equiv 5 \pmod{8}\\
- 3\rho_n(\delta_{2,1}) - 3\rho_n(\delta_{4,1}) + 2\rho_n(\delta_{8,1}) + \rho_n(\delta_{8,3})  & \ \text{if} \ n \equiv 7 \pmod{8}\\  
\end{array}
\end{cases}\\
F_2(n,0,1,1,1,0) &=& 2^{n-5} - \frac{1}{32} \cdot \begin{cases}
\begin{array}{lr}
-5\rho_n(\delta_{2,1}) + 5\rho_n(\delta_{4,1}) + 2\rho_n(\delta_{8,2}) - \rho_n(\delta_{8,3}) & \ \text{if} \  n \equiv 1,5 \pmod{8}\\
- \rho_n(\delta_{2,1}) + \rho_n(\delta_{4,1}) - 2\rho_n(\delta_{8,2}) - \rho_n(\delta_{8,3}) & \ \text{if}  \ n \equiv 3,7 \pmod{8}\\
\end{array}
\end{cases}\\
F_2(n,1,1,1,1,0) &=& 2^{n-5} - \frac{1}{32} \cdot \begin{cases}
\begin{array}{lr}
- 3\rho_n(\delta_{2,1}) - 3\rho_n(\delta_{4,1}) + 2\rho_n(\delta_{8,1}) + \rho_n(\delta_{8,3})  & \ \text{if}  \ n \equiv 1 \pmod{8}\\
- \rho_n(\delta_{2,1}) + \rho_n(\delta_{4,1}) - 2\rho_n(\delta_{8,2}) - \rho_n(\delta_{8,3}) & \ \text{if} \ n \equiv 3 \pmod{8}\\
\rho_n(\delta_{2,1}) + \rho_n(\delta_{4,1}) - 2\rho_n(\delta_{8,1}) + \rho_n(\delta_{8,3})  & \ \text{if} \ n \equiv 5 \pmod{8}\\
3\rho_n(\delta_{2,1}) - 3\rho_n(\delta_{4,1}) + 2\rho_n(\delta_{8,2}) - \rho_n(\delta_{8,3})  & \ \text{if} \ n \equiv 7 \pmod{8}\\  
\end{array}
\end{cases}\\
\end{eqnarray*}
}
{\small
\begin{eqnarray*}
F_2(n,0,0,0,0,1) &=& 2^{n-5} - \frac{1}{32} \big( 3\rho_n(\delta_{2,1})- 3\rho_n(\delta_{4,1}) + 2\rho_n(\delta_{8,2}) - \rho_n(\delta_{8,3}) \big) \ \ \text{if} \ n \equiv 1,3,5,7 \pmod{8}\\
F_2(n,1,0,0,0,1) &=& 2^{n-5} - \frac{1}{32} \cdot \begin{cases}
\begin{array}{lr}
3\rho_n(\delta_{2,1}) - 3\rho_n(\delta_{4,1}) + 2\rho_n(\delta_{8,2}) - \rho_n(\delta_{8,3})  & \ \text{if}  \ n \equiv 1 \pmod{8}\\
\rho_n(\delta_{2,1}) + \rho_n(\delta_{4,1}) - 2\rho_n(\delta_{8,1}) + \rho_n(\delta_{8,3}) & \ \text{if} \ n \equiv 3 \pmod{8}\\
- \rho_n(\delta_{2,1}) + \rho_n(\delta_{4,1})- 2\rho_n(\delta_{8,2}) - \rho_n(\delta_{8,3})  & \ \text{if} \ n \equiv 5 \pmod{8}\\
- 3\rho_n(\delta_{2,1}) - 3\rho_n(\delta_{4,1}) + 2\rho_n(\delta_{8,1}) + \rho_n(\delta_{8,3}) & \ \text{if} \ n \equiv 7 \pmod{8}\\ 
\end{array}
\end{cases}\\
F_2(n,0,1,0,0,1) &=& 2^{n-5} - \frac{1}{32} \cdot \begin{cases}
\begin{array}{lr}
- 3\rho_n(\delta_{2,1}) - 3\rho_n(\delta_{4,1}) + 2\rho_n(\delta_{8,1}) + \rho_n(\delta_{8,3})  & \ \text{if}  \ n \equiv 1,5 \pmod{8}\\
\rho_n(\delta_{2,1}) + \rho_n(\delta_{4,1}) - 2\rho_n(\delta_{8,1}) + \rho_n(\delta_{8,3}) & \ \text{if} \ n \equiv 3,7 \pmod{8}\\
\end{array}
\end{cases}\\
F_2(n,1,1,0,0,1) &=& 2^{n-5} - \frac{1}{32} \cdot \begin{cases}
\begin{array}{lr}
-5\rho_n(\delta_{2,1}) + 5\rho_n(\delta_{4,1}) + 2\rho_n(\delta_{8,2}) - \rho_n(\delta_{8,3})   & \ \text{if}  \ n \equiv 1 \pmod{8}\\
-3\rho_n(\delta_{2,1}) - 3\rho_n(\delta_{4,1}) + 2\rho_n(\delta_{8,1}) + \rho_n(\delta_{8,3}) & \ \text{if} \ n \equiv 3 \pmod{8}\\
-\rho_n(\delta_{2,1}) + \rho_n(\delta_{4,1}) - 2\rho_n(\delta_{8,2}) - \rho_n(\delta_{8,3})  & \ \text{if} \ n \equiv 5 \pmod{8}\\
\rho_n(\delta_{2,1}) + \rho_n(\delta_{4,1}) - 2\rho_n(\delta_{8,1}) + \rho_n(\delta_{8,3}) & \ \text{if} \ n \equiv 7 \pmod{8}\\ 
\end{array}
\end{cases}\\
F_2(n,0,0,1,0,1) &=& 2^{n-5} - \frac{1}{32} \big( - 3\rho_n(\delta_{2,1})  - 3\rho_n(\delta_{4,1}) + 2\rho_n(\delta_{8,1}) + \rho_n(\delta_{8,3}) \big) \ \ \text{if} \ n \equiv 1,3,5,7 \pmod{8}\\
F_2(n,1,0,1,0,1) &=& 2^{n-5} - \frac{1}{32} \cdot \begin{cases}
\begin{array}{lr}
- \rho_n(\delta_{2,1}) + \rho_n(\delta_{4,1}) - 2\rho_n(\delta_{8,2}) - \rho_n(\delta_{8,3})  & \ \text{if}  \ n \equiv 1 \pmod{8}\\
- 3\rho_n(\delta_{2,1}) - 3\rho_n(\delta_{4,1}) + 2\rho_n(\delta_{8,1}) + \rho_n(\delta_{8,3}) & \ \text{if} \ n \equiv 3 \pmod{8}\\
3\rho_n(\delta_{2,1}) -3\rho_n(\delta_{4,1}) + 2\rho_n(\delta_{8,2}) - \rho_n(\delta_{8,3}) & \ \text{if} \ n \equiv 5 \pmod{8}\\
\rho_n(\delta_{2,1}) + \rho_n(\delta_{4,1}) - 2\rho_n(\delta_{8,1}) + \rho_n(\delta_{8,3}) & \ \text{if} \ n \equiv 7 \pmod{8}\\ 
\end{array}
\end{cases}\\
F_2(n,0,1,1,0,1) &=& 2^{n-5} - \frac{1}{32} \cdot \begin{cases}
\begin{array}{lr}
\rho_n(\delta_{2,1})+ \rho_n(\delta_{4,1}) - 2\rho_n(\delta_{8,1}) + \rho_n(\delta_{8,3}) & \ \text{if}  \ n \equiv 1,5 \pmod{8}\\
- 3\rho_n(\delta_{2,1}) - 3\rho_n(\delta_{4,1}) + 2\rho_n(\delta_{8,1}) + \rho_n(\delta_{8,3})& \ \text{if} \ n \equiv 3,7 \pmod{8}\\
\end{array}
\end{cases}\\
F_2(n,1,1,1,0,1) &=& 2^{n-5} - \frac{1}{32} \cdot \begin{cases}
\begin{array}{lr}
\rho_n(\delta_{2,1}) + \rho_n(\delta_{4,1}) - 2\rho_n(\delta_{8,1}) + \rho_n(\delta_{8,3}) & \ \text{if}  \ n \equiv 1 \pmod{8}\\
3\rho_n(\delta_{2,1}) - 3\rho_n(\delta_{4,1}) + 2\rho_n(\delta_{8,2}) - \rho_n(\delta_{8,3}) & \ \text{if} \ n \equiv 3 \pmod{8}\\
- 3\rho_n(\delta_{2,1}) - 3\rho_n(\delta_{4,1}) + 2\rho_n(\delta_{8,1}) + \rho_n(\delta_{8,3}) & \ \text{if} \ n \equiv 5 \pmod{8}\\
- \rho_n(\delta_{2,1}) + \rho_n(\delta_{4,1}) - 2\rho_n(\delta_{8,2}) - \rho_n(\delta_{8,3})  & \ \text{if} \ n \equiv 7 \pmod{8}\\ 
\end{array}
\end{cases}\\
F_2(n,0,0,0,1,1) &=& 2^{n-5} - \frac{1}{32} \big(-\rho_n(\delta_{2,1}) + \rho_n(\delta_{4,1}) - 2\rho_n(\delta_{8,2}) -\rho_n(\delta_{8,3}) \big) \ \ \text{if} \ n \equiv 1,3,5,7 \pmod{8}\\
F_2(n,1,0,0,1,1) &=& 2^{n-5} - \frac{1}{32} \cdot \begin{cases}
\begin{array}{lr}
\rho_n(\delta_{2,1}) + \rho_n(\delta_{4,1}) - 2\rho_n(\delta_{8,1}) + \rho_n(\delta_{8,3}) & \ \text{if}  \ n \equiv 1 \pmod{8}\\
- \rho_n(\delta_{2,1}) + \rho_n(\delta_{4,1}) - 2\rho_n(\delta_{8,2}) - \rho_n(\delta_{8,3}) & \ \text{if} \ n \equiv 3 \pmod{8}\\
5 \rho_n(\delta_{2,1}) + 5 \rho_n(\delta_{4,1}) + 2 \rho_n(\delta_{8,1}) + \rho_n(\delta_{8,3}) & \ \text{if} \ n \equiv 5 \pmod{8}\\
3 \rho_n(\delta_{2,1}) - 3\rho_n(\delta_{4,1}) + 2\rho_n(\delta_{8,2}) - \rho_n(\delta_{8,3}) & \ \text{if} \ n \equiv 7 \pmod{8}\\
\end{array}
\end{cases}\\
F_2(n,0,1,0,1,1) &=& 2^{n-5} - \frac{1}{32} \cdot \begin{cases}
\begin{array}{lr}
\rho_n(\delta_{2,1}) + \rho_n(\delta_{4,1}) - 2\rho_n(\delta_{8,1}) + \rho_n(\delta_{8,3})  & \ \text{if}  \ n \equiv 1,5 \pmod{8}\\
- 3\rho_n(\delta_{2,1}) - 3\rho_n(\delta_{4,1}) + 2\rho_n(\delta_{8,1}) + \rho_n(\delta_{8,3})  & \ \text{if}  \ n \equiv 3,7 \pmod{8}\\
\end{array}
\end{cases}\\
F_2(n,1,1,0,1,1) &=& 2^{n-5} - \frac{1}{32} \cdot \begin{cases}
\begin{array}{lr}
\rho_n(\delta_{2,1}) + \rho_n(\delta_{4,1}) - 2\rho_n(\delta_{8,1}) + \rho_n(\delta_{8,3}) & \ \text{if}  \ n \equiv 1 \pmod{8}\\
3 \rho_n(\delta_{2,1}) - 3\rho_n(\delta_{4,1}) + 2\rho_n(\delta_{8,2}) - \rho_n(\delta_{8,3}) & \ \text{if}  \ n \equiv 3 \pmod{8}\\
-3 \rho_n(\delta_{2,1}) - 3 \rho_n(\delta_{4,1}) + 2\rho_n(\delta_{8,1}) + \rho_n(\delta_{8,3})& \ \text{if}  \ n \equiv 5 \pmod{8}\\
- \rho_n(\delta_{2,1}) + \rho_n(\delta_{4,1}) - 2\rho_n(\delta_{8,2}) - \rho_n(\delta_{8,3}) & \ \text{if}  \ n \equiv 7 \pmod{8}\\
\end{array}
\end{cases}\\
F_2(n,0,0,1,1,1) &=& 2^{n-5} - \frac{1}{32} \big( \rho_n(\delta_{2,1}) + \rho_n(\delta_{4,1}) - 2\rho_n(\delta_{8,1}) + \rho_n(\delta_{8,3}) \big) \ \ \text{if} \ n \equiv 1,3,5,7 \pmod{8}\\
F(n,1,0,1,1,1) &=& 2^{n-5} - \frac{1}{32} \cdot \begin{cases}
\begin{array}{lr}
\rho_n(\delta_{2,1}) + \rho_n(\delta_{4,1}) - 2\rho_n(\delta_{8,1}) + \rho_n(\delta_{8,3}) & \ \text{if}  \ n \equiv 1 \pmod{8}\\
- \rho_n(\delta_{2,1}) + \rho_n(\delta_{4,1}) - 2 \rho_n(\delta_{8,2}) - \rho_n(\delta_{8,3}) & \ \text{if}  \ n \equiv 3 \pmod{8}\\
-3 \rho_n(\delta_{2,1}) -3 \rho_n(\delta_{4,1}) + 2 \rho_n(\delta_{8,1}) + \rho_n(\delta_{8,3}) & \ \text{if}  \ n \equiv 5 \pmod{8}\\
-5 \rho_n(\delta_{2,1}) + 5 \rho_n(\delta_{4,1}) + 2 \rho_n(\delta_{8,2}) - \rho_n(\delta_{8,3})& \ \text{if}  \ n \equiv 7 \pmod{8}\\
\end{array}
\end{cases}\\
F_2(n,0,1,1,1,1) &=& 2^{n-5} - \frac{1}{32} \cdot \begin{cases}
\begin{array}{lr}
-3 \rho_n(\delta_{2,1}) -3 \rho_n(\delta_{4,1})+ 2 \rho_n(\delta_{8,1}) + \rho_n(\delta_{8,3})  & \ \text{if}  \ n \equiv 1,5 \pmod{8}\\
\rho_n(\delta_{2,1}) + \rho_n(\delta_{4,1}) - 2\rho_n(\delta_{8,1}) + \rho_n(\delta_{8,3}) & \ \text{if}  \ n \equiv 3,7 \pmod{8}\\
\end{array}
\end{cases}\\
F_2(n,1,1,1,1,1) &=& 2^{n-5} - \frac{1}{32} \cdot \begin{cases}
\begin{array}{lr}
- \rho_n(\delta_{2,1})+ \rho_n(\delta_{4,1}) -2 \rho_n(\delta_{8,2}) - \rho_n(\delta_{8,3}) & \ \text{if}  \ n \equiv 1 \pmod{8}\\
\rho_n(\delta_{2,1}) + \rho_n(\delta_{4,1}) - 2\rho_n(\delta_{8,1}) + \rho_n(\delta_{8,3})& \ \text{if}  \ n \equiv 3 \pmod{8}\\
3 \rho_n(\delta_{2,1}) -3 \rho_n(\delta_{4,1}) + 2 \rho_n(\delta_{8,2}) - \rho_n(\delta_{8,3}) & \ \text{if}  \ n \equiv 5 \pmod{8}\\
5 \rho_n(\delta_{2,1})+ 5 \rho_n(\delta_{4,1})+ 2 \rho_n(\delta_{8,1}) + \rho_n(\delta_{8,3}) & \ \text{if}  \ n \equiv 7 \pmod{8}\\
\end{array}
\end{cases}\\
\end{eqnarray*}
}
\end{theorem}
One can check that the roots $\delta_{8,3}$ are $\sqrt{2}\omega_{40}^3,\sqrt{2}\omega_{40}^{11},\sqrt{2}\omega_{40}^{13},\sqrt{2}\omega_{40}^{19},
\sqrt{2}\omega_{40}^{21},\sqrt{2}\omega_{40}^{27},\sqrt{2}\omega_{40}^{29},\sqrt{2}\omega_{40}^{37}$, where
\begin{eqnarray*}
\omega_{40} = \frac{1}{\sqrt{2}}\Big( \frac{1+\sqrt{5}}{4} - \frac{1-\sqrt{5}}{8}\sqrt{10 + 2\sqrt{5}}\Big)
+ \frac{1}{\sqrt{2}}\Big( \frac{1+\sqrt{5}}{4} + \frac{1-\sqrt{5}}{8}\sqrt{10 + 2\sqrt{5}}\Big)i
\end{eqnarray*}
is the primitive $40$-th root of unity with smallest (positive) argument.
Due to the presence of the roots of either $\delta_{8,1}$ or $\delta_{8,2}$ in each formula, they can not be periodic in $n$.

\subsubsection*{An alternative parameterisation.}\label{sec:5coeffsparam2}
We now present an alternative parameterisation of Eq.~(\ref{eq:5coeffsparam1}) which although requires the same number of
variables to linearise, is perhaps more natural given the form of $T_2(a_{0}^2 + a_0 + r_0)$ and $T_3(a_{0}^2 + a_0 + r_0)$, 
see~\S\ref{sec:5direct}. In particular, 
using Lemma~\ref{lem:T} part (2) three times with $(\alpha,\beta) = (a_{0}^5, a_{0}^3)$, $(a_{0}^5, a_{0})$ and $(a_{0}^3, a_{0})$,
we have
\begin{eqnarray}
\nonumber T_5(a_{0}^2 + a_0 + r_0) &=& T_1(a_{0}^5)T_1(a_{0}^3) + T_1(a_{0}^5)T_1(a_{0}) + T_1(a_{0}^3)T_1(a_{0}) + r_0(T_2(a_{0}^3) + T_2(a_{0}) + T_1(a_{0}^3)T_1(a_{0}))\\
\nonumber  &+& T_1\Big(a_{0}^9 + a_{0}^3 + a_{0} + r_0 \Big(a_{0}^7 + (a_{0}^5 + a_0)\binom{n}{2} + (a_{0}^3 + a_0)\binom{n}{3} + \binom{n}{5}\Big)\Big)\\
\nonumber &=& T_2(a_{0}^5 + a_{0}^3) + T_2(a_{0}^5 + a_{0}) + (r_0 + 1)T_2(a_{0}^3 + a_{0}) \\
\label{eq:T5linearalt} &+& T_1\Big(a_{0}^9 + a_{0} + r_0 \Big(a_{0}^7 + a_0 +(a_{0}^5 + a_0)\binom{n}{2} + (a_{0}^3 + a_0)\binom{n}{3} + \binom{n}{5}\Big)\Big).
\end{eqnarray}
This can be linearised using the substitutions $a_{0}^3 + a_{0} = a_{1}^2 + a_{1} + r_1$ and $a_{0}^5 + a_{0} = a_{2}^2 + a_{2} + r_2$, 
where $r_1$ and $r_2$ are the traces of $a_{0}^3 + a_{0}$ and $a_{0}^5 + a_{0}$ respectively, which when added together imply that
$a_{0}^5 + a_{0}^3 = (a_1 + a_2)^2 + (a_1 + a_2) + (r_1 + r_2)$.
This results in 
\begin{eqnarray*}
\nonumber T_5(a_{0}^2 + a_0 + r_0) &=& T_1\Big( a_{1}^2 a_2 + a_1 a_{2}^2 + a_{0}^9 + a_{0} \\
&+&  r_0\Big(a_{1}^3 + a_{1} + a_{0}^7 + a_0 +(a_{0}^5 + a_0 + r_1)\binom{n}{2} + (a_{0}^3 + a_0)\binom{n}{3} + \binom{n}{5}\Big)\Big).
\end{eqnarray*}
For $16 \le i \le 31$ let $\vec{i} = (i_4,i_3,i_2,i_1,i_0)$. The curves we are interested in are given by the following intersections:
\begin{eqnarray}
\nonumber a_{3}^2 + a_{3} &=& i_4\Big(a_{1}^2 a_2 + a_1 a_{2}^2 + a_{0}^9 + a_{0} + r_0\Big(a_{1}^3 + a_{1} + a_{0}^7 + a_0 +(a_{0}^5 + a_0 + r_1)\binom{n}{2} + (a_{0}^3 + a_0)\binom{n}{3} + \binom{n}{5}\Big)\Big)\\
\nonumber &+& i_3\Big( a_{1}^3 + a_{1} + a_{0}^7 + a_{0}^5 + a_{0}^3 + a_{0} + r_0\Big(a_{0}^3 + a_0 + (a_{0}^3 + a_0)\binom{n}{2}  + \binom{n}{4}\Big) + r_1\binom{n}{2}\Big)\\
\nonumber &+& i_2\Big(a_{0}^5 + a_{0} + r_0\Big(a_{0}^3 + a_0 + \binom{n}{3} \Big)\Big) + i_1\Big(a_{0}^3 + a_{0} + r_0\binom{n}{2}\Big) + i_0r_0,\\
\nonumber a_{0}^3 + a_{0} &=& a_{1}^2 + a_1 + r_1\\
\nonumber a_{0}^5 + a_{0} &=& a_{2}^2 + a_2 + r_2.
\end{eqnarray}
For $i_4 = 1$ the genus of all of these absolutely irreducible curves is $21$ -- rather than $18$ as in the first parameterisation.
Nevertheless, once the transform~(\ref{eq:Sinverse}) is applied one again obtains the formulae given in Theorem~\ref{thm:5traces}.

\subsubsection{Some examples from the direct method.}\label{sec:5direct}
For $n$ odd, using Equations~(\ref{eq:T1linear}) to~(\ref{eq:T3linear}) and the alternative parametrisations~(\ref{eq:T4linearalt}) 
and~(\ref{eq:T5linearalt}), one can write down curves arising from the direct method for $F_2(n,t_1,t_2,t_3,t_4,t_5)$, as in~\S\ref{sec:4coeffsdirect}. 
Here we give the two simplest examples, valid for all $n \ge 5$. 
\begin{eqnarray}
\nonumber F_2(n,0,0,0,0,0) &=& \#\{a \in \F_{2^n} \mid T_1(a) = 0, \ T_2(a) = 0, \ T_3(a) = 0, \ T_4(a) = 0, \ T_5(a) = 0\}\\
\nonumber &=& \frac{1}{2} \, \#\{ a_{0} \in \F_{2^n} \mid T_2(a_{0}^2 + a_{0}) = 0, \ T_3(a_{0}^2 + a_{0}) = 0, \ T_4(a_{0}^2 + a_{0}) = 0, \ 
T_5(a_{0}^2 + a_{0}) = 0\}\\
\nonumber &=& \frac{1}{2} \, \#\{ a_{0} \in \F_{2^n} \mid T_1(a_{0}^3 + a_{0}) = 0, \ T_1(a_{0}^5 + a_{0}) = 0,\\
\nonumber && \ \ \ \ T_2(a_{0}^3 + a_{0}) + T_1( a_{0}^7 + a_{0}^5 + a_{0}^3 + a_0) = 0,\\
\nonumber && \ \ \ \ T_1(a_{0}^5)T_1(a_{0}^3) + T_1(a_{0}^5)T_1(a_{0}) + T_1(a_{0}^3)T_1(a_{0}) + T_1(a_{0}^9 + a_{0}^3 + a_{0}) = 0\}\\
\nonumber &=& \frac{1}{2} \, \#\{ a_{0} \in \F_{2^n} \mid T_1(a_{0}^3 + a_{0}) = 0, \ T_1(a_{0}^5 + a_{0}) = 0,\\
\nonumber && \ \ \ \  T_2(a_{0}^3 + a_{0}) + T_1( a_{0}^7 + a_{0}^5 + a_{0}^3 + a_0) = 0,\\
\nonumber && \ \ \ \  T_2(a_{0}^5 + a_{0}^3) +  T_2(a_{0}^5 + a_{0}) + T_2(a_{0}^3 + a_{0}) + T_1(a_{0}^9 + a_0) = 0\}\\
\nonumber &=& \frac{1}{8} \, \#\{ (a_{0},a_{1},a_{2}) \in (\F_{2^n})^3 \mid a_{1}^2 + a_{1} = a_{0}^3 + a_{0}, \ a_{2}^2 + a_{2} = a_{0}^5+a_{0},\\
\nonumber && \ \ \ \  T_1( a_{1}^3 + a_{1} + a_{0}^7 + a_{0}^5 + a_{0}^3 + a_0) = 0,\\
\nonumber && \ \ \ \  T_2((a_{1} + a_{2})^2 + (a_{1} + a_{2})) +  T_1(a_{2}^3 + a_{2}) + T_1(a_{1}^3 + a_{1}) + T_1(a_{0}^9 + a_0) = 0\}\\
\nonumber &=& \frac{1}{32} \, \#\{ (a_{0},a_{1},a_{2},a_{3},a_{4}) \in (\F_{2^n})^5 \mid a_{1}^2 + a_{1} = a_{0}^3 + a_{0}, \ a_{2}^2 + a_{2} = 
a_{0}^5+a_{0},\\
\label{genus49} && \ \ \ \  a_{3}^2 + a_{3} =  a_{1}^3 + a_{1} + a_{0}^7 + a_{0}^5 + a_{0}^3 + a_0, \ a_{4}^2 + a_{4} = a_{1}^2a_{2} + a_{1}a_{2}^2 + a_{0}^9 
+ a_0\}.
\end{eqnarray}
The genus of the absolutely irreducible curve this intersection describes is $49$, and thus one can not easily brute-force compute 
the characteristic polynomial of Frobenius by computing the number of points over $\F_{2^n}$ for $n \le 49$, as for 
the four coefficient direct method cases. The state-of-the-art $p$-adic point counting
algorithms of Tuitman~\cite{tuitman1,tuitman2} are unfortunately not easily adaptable to such intersections, while the prime $2$ is problematic.
However, by ignoring the fourth trace, letting $\delta_{2,3}(X) = X^2 - 2X +2$ and $\delta_{4,2} = X^4 + 2X^2 + 4$, we have
\begin{eqnarray}
\nonumber F_2(n,0,0,0,*,0) &=& \#\{a \in \F_{2^n} \mid T_1(a) = 0, \ T_2(a) = 0, \ T_3(a) = 0, \ T_5(a) = 0\}\\
\nonumber &=& \frac{1}{2} \, \#\{ a_{0} \in \F_{2^n} \mid T_2(a_{0}^2 + a_{0}) = 0, \ T_3(a_{0}^2 + a_{0}) = 0, \ T_5(a_{0}^2 + a_{0}) = 0\}\\
\nonumber &=& \frac{1}{2} \, \#\{ a_{0} \in \F_{2^n} \mid T_1(a_{0}^3 + a_{0}) = 0, \ T_1(a_{0}^5 + a_{0}) = 0,\\
\nonumber && \ \ \ \  T_2(a_{0}^5 + a_{0}^3) +  T_2(a_{0}^5 + a_{0}) + T_2(a_{0}^3 + a_{0}) + T_1(a_{0}^9 + a_0) = 0\}\\
\nonumber &=& \frac{1}{16} \, \#\{ (a_{0},a_{1},a_{2},a_{3}) \in (\F_{2^n})^4 \mid a_{1}^2 + a_{1} = a_{0}^3 + a_{0}, \ a_{2}^2 + a_{2} = a_{0}^5+a_{0},\\
\nonumber && \ \ \ \ a_{3}^2 + a_{3} = a_{1}^2a_{2} + a_{1}a_{2}^2 + a_{0}^9 + a_0\}\\
\nonumber &=& 2^{n-4} - \frac{1}{16}\big(5\rho_n(\delta_{2,1}) + 2\rho_n(\delta_{2,2}) + 2\rho_n(\delta_{2,3}) + 3\rho_n(\delta_{4,1}) + \rho_n(\delta_{4,2}) + \rho_n(\delta_{8,3})\big),
\end{eqnarray}
where in the final equality we have used Magma to compute the characteristic polynomial of Frobenius of the genus $21$ curve 
described by the intersection. Since neither $\delta_{8,1}$ nor $\delta_{8,2}$ feature in this formula, we deduce that
the set of formulae for $F_2(n,0,0,0,*,0)$ are periodic in $n$ with period $\text{LCM}(8,8,4,24,6,40) = 120$, in similarity with the two and 
three coefficient cases for which the periods are $8$ and $24$ in $n$, respectively.

\subsubsection{Further general formulae.}\label{subsubsec:5coeffgeneral}
While one can not easily compute the zeta function for $F_2(n,0,0,0,0,0)$ using the direct method, one can use the approach 
of~\S\ref{subsubsec:altgeneralformulae} and the alternative parameterisations~(\ref{eq:T4linearalt}) and~(\ref{eq:T5linearalt}) to compute
$F_2(n,0,0,0,t_4,t_5)$ for all $n \ge 5$, as follows.
 
The conditions $T_1(a) = T_2(a) = T_3(a) = 0$ imply that we should set $A = \{(a_0,a_1,a_2) \in (\F_{2^n})^3 \mid a_{0}^3 + a_{0} = a_{1}^2 + a_{1}
\ a_{0}^5 + a_{0} = a_{2}^2 + a_{2}\}$.
Furthermore let $\vec{f} = (T_5(a_{0}^2 + a_{0}), T_4(a_{0}^2 + a_{0})) = (T_1(a_{1}^2 a_2 + a_{1}a_{2}^2 + a_{0}^9 + a_0), T_1(a_{1}^3 + a_1 + a_{0}^7 + a_{0}^5 + a_{0}^3 + a_{0}))$.
%\end{eqnarray}
Then by the generalisation of the transform~(\ref{eq:Sinverse}) we have
\begin{equation}\label{eq:general5}
{\small
\begin{bmatrix}
F_2(n,0,0,0,0,0)\\
F_2(n,0,0,0,1,0)\\
F_2(n,0,0,0,0,1)\\
F_2(n,0,0,0,1,1)\\
\end{bmatrix}
=
\begin{bmatrix}
N(\vec{0})\\
N(\vec{1})\\
N(\vec{2})\\
N(\vec{3})\\
\end{bmatrix}
=
\frac{1}{16}
\begin{bmatrix*}[r]
-1 & 1 & 1 & 1\\
1 & -1 & 1 & -1\\
1 & 1 & -1 & -1\\
1 & -1 & -1 & 1\\
\end{bmatrix*}
\begin{bmatrix}
V(\vec{0} \cdot \vec{f})\\
V(\vec{1} \cdot \vec{f})\\
V(\vec{2} \cdot \vec{f})\\
V(\vec{3} \cdot \vec{f})\\
\end{bmatrix},}
\end{equation}
where we have a factor of $1/16$ rather than $1/2$ because of the three factors of $1/2$ arising from the introduction of the variables $a_0, a_1$ 
and $a_2$ defining $A$. Note that $V(\vec{0} \cdot \vec{f}) = |A| = 2^n - 2\rho_n(\delta_{2,1}) - \rho_n(\delta_{2,2}) - \rho_n(\delta_{4,1})$.
For $1 \le i \le 3$ let $\vec{i} = (i_1,i_0)$. We thus have:
\begin{eqnarray*}
V(\vec{i} \cdot \vec{f}) &=& \frac{1}{2}\#\{(a_0,a_1,a_2,a_3) \mid a_{0}^3 + a_{0} = a_{1}^2 + a_{1}, \ a_{0}^5 + a_{0} = a_{2}^2 + a_{2}\\ 
&& \ \ \ \ a_{3}^2 + a_{3} = i_1(a_{1}^2 a_2 + a_{1}a_{2}^2 + a_{0}^9 + a_0) + i_0(a_{1}^3 + a_1 + a_{0}^7 + a_{0}^5 + a_{0}^3 + a_{0}) \}.
\end{eqnarray*}
For $i = 1,2$ and $3$ these are absolutely irreducible curves of genus $5,17,21$ and $21$ respectively, and thus their zeta functions 
are far easier to compute than for those curves arising from the direct method. Let $\delta_{8,4}(X) = X^8 + 2X^6 - 4X^5 + 2X^4 - 8X^3 + 8X^2 + 16$.
Combining the $V(\vec{i} \cdot \vec{f})$ as per Eq.~(\ref{eq:general5}) gives the following theorem.

\begin{theorem}
For $n \ge 5$ we have
\begin{eqnarray*}
F_2(n,0,0,0,0,0) &=& 2^{n-5} - \frac{1}{32}\big(7\rho_n(\delta_{2,1}) + 4\rho_n(\delta_{2,2}) + 2\rho_n(\delta_{2,3})+ 5\rho_n(\delta_{4,1}) + 3\rho_n(\delta_{4,2}) + 2\rho_n(\delta_{8,1})\\
&& \ \ \ \ \ \ \ + \rho_n(\delta_{8,2}) + \rho_n(\delta_{8,3}) + \rho_n(\delta_{8,4})\big),\\ 
F_2(n,0,0,0,1,0) &=& 2^{n-5} - \frac{1}{32}\big(3\rho_n(\delta_{2,1}) + 2\rho_n(\delta_{2,3}) + \rho_n(\delta_{4,1}) -\rho_n(\delta_{4,2})
- 2\rho_n(\delta_{8,1}) - \rho_n(\delta_{8,2})\\
&& \ \ \ \ \ \ \  + \rho_n(\delta_{8,3})  - \rho_n(\delta_{8,4}) \big),\\ 
F_2(n,0,0,0,0,1) &=& 2^{n-5} - \frac{1}{32}\big(\rho_n(\delta_{2,1}) + 2\rho_n(\delta_{2,2}) - 2\rho_n(\delta_{2,3}) -3\rho_n(\delta_{4,1}) - 3\rho_n(\delta_{4,2})
+ \rho_n(\delta_{8,2})\\
&& \ \ \ \ \ \ \  - \rho_n(\delta_{8,3})  - \rho_n(\delta_{8,4}) \big),\\ 
F_2(n,0,0,0,1,1) &=& 2^{n-5} - \frac{1}{32}\big(-3\rho_n(\delta_{2,1}) - 2\rho_n(\delta_{2,2}) - 2\rho_n(\delta_{2,3}) + \rho_n(\delta_{4,1}) + \rho_n(\delta_{4,2})
 - \rho_n(\delta_{8,2})\\
&& \ \ \ \ \ \ \  - \rho_n(\delta_{8,3}) + \rho_n(\delta_{8,4}) \big).
\end{eqnarray*}
\end{theorem}
Note that the formula for $F_2(n,0,0,0,0,0)$ arises from the curve~(\ref{genus49}) of genus $49$. This approach therefore provides a much more efficient way to compute the zeta function than the direct method does.
Further note that the terms involving $\delta_{8,1}$, $\delta_{8,2}$ and $\delta_{8,4}$ all cancel in 
$F_2(n,0,0,0,0,0) + F_2(n,0,0,0,1,0) = F_2(0,0,0,*,0)$, as well as in 
$F_2(n,0,0,0,0,1) + F_2(n,0,0,0,1,1) = F_2(0,0,0,*,1)$, which has period $120$ in $n$ as well.

\subsection{Computing $F_2(n,t_1,t_2,t_3,t_4,t_5,t_6)$}

In this subsection and the next we shall only detail the indirect method, since the curves produced have genera which are already very large,
making the brute force computation of the characteristic polynomials of Frobenius prohibitive, while the direct method produce curves of even 
larger genus. 

Let $\vec{f} = (T_6,T_5,T_4,T_3,T_2,T_1)$. To determine $V( \vec{i} \cdot \vec{f})$ for $32 \le i \le 63$, we use Lemma~\ref{lem:T} part (6).
In particular, setting $\alpha = a_{0}^2$ and $\beta = a_{0}$ for $r_0 = 0$, and $\alpha = a_{0}^2 + a_0$ and $\beta = 1$ for $r_0 = 1$, 
and evaluating mod $2$ gives the following:
%{\small
\begin{eqnarray*}
%T_6(a_{0}^2 + a_{0}) &=& T_3(a_{0}) + T_2(a_{0}^3)T_1(a_{0}^3) + T_2(a_{0}^3)T_1(a_{0}) + T_2(a_{0})T_1(a_{0}^3)  + T_2(a_{0}^5) + T_2(a_{0})\\
%&& + T_1(a_{0}^7)T_1(a_{0}^3) + T_1(a_{0}^7)T_1(a_{0}) + T_1(a_{0}^5)T_1(a_{0}^3) + T_1(a_{0}^3)T_1(a_{0}) + T_1(a_{0}^{11} + a_{0}^7),\\
%T_6(a_{0}^2 + a_{0} + 1)  &=& T_3(a_{0}) + T_2(a_{0}^3)T_1(a_{0}^3) + T_2(a_{0}^3)T_1(a_{0}) + T_2(a_{0})T_1(a_{0}^3)  + T_2(a_{0}^5) + T_2(a_{0})\\
%&& + T_1(a_{0}^7)T_1(a_{0}^3) + T_1(a_{0}^7)T_1(a_{0}) + T_1(a_{0}^5)T_1(a_{0}^3) + \Big(\binom{n}{2} + 1\Big)T_1(a_{0}^3)T_1(a_{0})\\
%&& + \binom{n}{2}(T_2(a_{0}^3) + T_2(a_{0})) + T_1\Big(a_{0}^{11} + \Big(\binom{n}{2} + 1\Big)a_{0}^7 + \binom{n}{3}a_{0}^5 + \Big(\binom{n}%{4}+1\Big)a_{0}^3\\
%&& + \Big(\binom{n}{4} + \binom{n}{3} + 1\Big)a_{0}  + \binom{n}{6}\Big). \\
T_6(a_{0}^2 + a_{0} + r_0) &=& T_3(a_0) + T_2(a_{0}^3)T_1(a_{0}^3) + T_2(a_{0}^3)T_1(a_0) + T_2(a_0)T_1(a_{0}^3) + T_2(a_{0}^5) + T_2(a_0)\\
&+& T_1(a_{0}^7)T_1(a_{0}^3) + T_1(a_{0}^7)T_1(a_{0}) + T_1(a_{0}^5)T_1(a_{0}^3) + T_1(a_{0}^3)T_1(a_{0})\\
&+& r_0\big(T_2(a_{0}^3) + T_2(a_0) + T_1(a_{0}^3)T_1(a_0) + T_1(a_{0}^7)\big)\binom{n}{2}\\
&+& T_1\Big(a_{0}^{11} + a_{0}^7 + r_0(a_{0}^3 + a_0) + r_0(a_{0}^5 + a_0)\binom{n}{3} + r_0(a_{0}^3 + a_0)\binom{n}{4} + r_0\binom{n}{6}\Big).
\end{eqnarray*}
%}
This can be reduced to a $T_1$ expression using the substitutions $a_0 = a_{1}^2 + a_1 + r_1$, $a_{0}^3 = a_{2}^2 + a_2 + r_2$ and
$a_{0}^5 = a_{3}^2 + a_3 + r_3$, where $r_1,r_2,r_3 \in \F_2$ are the traces of $a_0,a_{0}^3$ and $a_{0}^5$ respectively. 
This results in 
%{\small
\begin{eqnarray*}
T_6(a_{0}^2 + a_0 + r_0) &=& T_1\Big( a_{3}^3 + a_3 + (r_1 + r_2)(a_{2}^3 + a_2 + a_{1}^3 + a_1 + a_{0}^7) + a_{1}^5 + a_{1}^3 + a_{0}^{11} + a_{0}^7
+ r_0r_1 + r_0r_2 \\
&+& r_1r_2 + r_2r_3 + \big(r_0(a_{2}^3 + a_2 + a_{1}^3 + a_1 + a_{0}^7 + r_1 + r_2 + r_1r_2 +1) +r_1 + r_2 + r_3\big)\binom{n}{2}\\
&+& (r_0r_1 + r_0r_3 + r_1)\binom{n}{3} + r_0(r_1 + r_2)\binom{n}{4} + r_0\binom{n}{6}\Big).
%T_6(a_{0}^2 + a_0) &=& T_1\Big( a_{1}^5 + a_{1} + a_{3}^3 + a_{3} + a_{0}^{11} + a_{0}^7 + a_{0}^3 + a_0 
%+ (r_0 + r_1)\Big(a_{2}^3 + a_{2} + a_{0}^7 + r_1\binom{n}{2}\Big)\\
%&+& r_1(a_{0}^3 + a_{0}) + r_1(r_0 + r_2) + r_0\binom{n}{3} + r_2\binom{n}{2} \Big),\\
%T_6(a_{0}^2 + a_0 + 1) &=& T_1\Big( 
%a_{1}^5 + a_{1} + a_{3}^3 + a_{3} + (r_0 + r_1)\Big( a_{2}^3 + a_{2} + a_{0}^7 + (r_1+1)\binom{n}{2}\Big) + (r_1 + 1)(a_{0}^3 + a_{0}) \\
%&+& \binom{n}{2}(a_{2}^3 + a_2 + a_{1}^3 + a_{1}) + r_0\binom{n}{3} + r_2\binom{n}{2} + r_1r_2 + \Big(\binom{n}{2} + 1\Big)(a_{0}^7 + r_0r_1) + 
%a_{0}^{11} \\
%&+& \binom{n}{3}a_{0}^5 + \Big(\binom{n}{4} + 1\Big)a_{0}^3 + \Big(\binom{n}{4} + \binom{n}{3} + 1 \Big)a_0 + \binom{n}{6}\Big).
\end{eqnarray*}
%}
For $32 \le i \le 63$ let $\vec{i} = (i_5,i_4,i_3,i_2,i_1,i_0)$. The curves we are interested in are given by the following intersections:
\begin{eqnarray}
\nonumber a_{4}^2 + a_{4} &=& i_5\Big( a_{3}^3 + a_3 + (r_1 + r_2)(a_{2}^3 + a_2 + a_{1}^3 + a_1 + a_{0}^7) + a_{1}^5 + a_{1}^3 + a_{0}^{11} + a_{0}^7
+ r_0r_1 + r_0r_2 \\
\nonumber &+& r_1r_2 + r_2r_3 + \big(r_0(a_{2}^3 + a_2 + a_{1}^3 + a_1 + a_{0}^7 + r_1 + r_2 + r_1r_2 +1) +r_1 + r_2 + r_3\big)\binom{n}{2}\\
\nonumber &+& (r_0r_1 + r_0r_3 + r_1)\binom{n}{3} + r_0(r_1 + r_2)\binom{n}{4} + r_0\binom{n}{6}\Big)\\
\nonumber &+& i_4\Big(r_0(a_{2}^3 + a_{2} + a_{1}^3 + a_{1}) + a_{0}^9 + r_0a_{0}^7 + (r_1 + r_2)a_{0}^5 + r_1 + r_2 + r_1r_2 + r_0r_1r_2\\
\nonumber &+& (r_0a_{0}^5 + r_0r_2)\binom{n}{2} + (r_0r_1 + r_0r_2)\binom{n}{3} + r_0\binom{n}{5}\Big)\\
\nonumber &+& i_3\Big(a_{2}^3 + a_2 + a_{1}^3 + a_1 + a_{0}^7 + a_{0}^5 + (r_0 + 1)(r_1 + r_2)\binom{n}{2} + r_0r_1 + r_0r_2 + r_1r_2 + r_2 + r_0\binom{n}{4}\Big)\\
\nonumber &+& i_2\Big(a_{0}^5 + a_{0} + r_0\Big(a_{0}^3 + a_0 + \binom{n}{3} \Big)\Big) + i_1\Big(a_{0}^3 + a_{0} + r_0\binom{n}{2}\Big) + i_0r_0,\\
\nonumber a_0 &=& a_{1}^2 + a_1 + r_1,\\
\label{eq:6coefficients} a_{0}^3 &=& a_{2}^2 + a_2 + r_2.\\
\nonumber a_{0}^5 &=& a_{3}^2 + a_3 + r_3.
%{\small
%\begin{eqnarray*}
%a_{4}^2 + a_{4} &=& 
%i_5\Big(a_{1}^5 + a_{1} + a_{3}^3 + a_{3} + a_{0}^{11} + a_{0}^7 + a_{0}^3 + a_0 
%+ (r_0 + r_1)\Big(a_{2}^3 + a_{2} + a_{0}^7 + r_1\binom{n}{2}\Big)\\
%&+& r_1(a_{0}^3 + a_{0}) + r_1(r_0 + r_2) + r_0\binom{n}{3} + r_2\binom{n}{2} \Big) + i_4( a_{0}^9 + a_{0}^3 + a_0 + a_{0}^5(r_0 + r_1) + r_0 r_1 )\\
%&+& i_3\Big(a_{2}^3 + a_2 + a_{1}^3 + a_1 + a_{0}^7 + a_{0}^5 + a_{0}^3 + \binom{n}{2}(r_0 + r_1) + r_0r_1\Big) + i_2(a_{0}^5 + a_0) + i_1(a_{0}^3 +a_0),\%\
%a_0 &=& a_{1}^2 + a_1 + r_0,\\
%a_{0}^3 &=& a_{2}^2 + a_2 + r_1,\\
%a_{0}^5 &=& a_{3}^2 + a_3 + r_2.\\
\end{eqnarray}
%}
%The eight curves for $a$ of trace $1$ are given by:
%%{\small 
%\begin{eqnarray*}
%a_{4}^2 + a_{4} &=&
%i_5\Big(a_{1}^5 + a_{1} + a_{3}^3 + a_{3} + (r_0 + r_1)\Big( a_{2}^3 + a_{2} + a_{0}^7 + (r_1+1)\binom{n}{2}\Big) + (r_1 + 1)(a_{0}^3 + a_{0}) \\
%&+& \binom{n}{2}(a_{2}^3 + a_2 + a_{1}^3 + a_{1}) + r_0\binom{n}{3} + r_2\binom{n}{2} + r_1r_2 + \Big(\binom{n}{2} + 1\Big)(a_{0}^7 + r_0r_1) + 
%a_{0}^{11} \\
%&+& \binom{n}{3}a_{0}^5 + \Big(\binom{n}{4} + 1\Big)a_{0}^3 + \Big(\binom{n}{4} + \binom{n}{3} + 1 \Big)a_0 + \binom{n}{6}\Big)\\
%&+&i_4\Big(  a_{2}^3 + a_2 + a_{1}^3 + a_1 + a_{0}^9 + a_{0}^7 + a_{0}^3 + a_{0} + a_{0}^5(r_0 + r_1) + \binom{n}{2}(a_{0}^5 + a_0 + r_0 + r_1)\\
%&+& \binom{n}{3}(a_{0}^3 + a_0) + \binom{n}{5} \Big) \\
%&+& i_3\Big( a_{2}^3 + a_2 + a_{1}^3 + a_1 + a_{0}^7 + a_{0}^5 + \binom{n}{2}(a_{0}^3 + a_{0} + r_0 + r_1) + a_0 
%+ \binom{n}{4} + r_0r_1 \Big) \\
%&+& i_2\Big(a_{0}^5 + a_{0}^3 + \binom{n}{3}\Big) + i_1\Big(a_{0}^3 + a_0 + \binom{n}{2}\Big) + i_0 \binom{n}{1},\\
%a_0 &=& a_{1}^2 + a_1 + r_0,\\
%a_{0}^3 &=& a_{2}^2 + a_2 + r_1,\\
%a_{0}^5 &=& a_{3}^2 + a_3 + r_2.
%\end{eqnarray*}
%%}
Corollary~\ref{cor:period} implies that mod $2$ one has
\begin{eqnarray*}
\Big(\binom{n}{6},\binom{n}{5},\binom{n}{4},\binom{n}{3},\binom{n}{2},\binom{n}{1}\Big) \equiv \begin{cases}
\begin{array}{lr}
(0,0,0,0,0,1) & \ \text{if} \ n \equiv 1 \pmod{8}\\
(0,0,0,1,1,1) & \ \text{if} \ n \equiv 3 \pmod{8}\\
(0,1,1,0,0,1) & \ \text{if} \ n \equiv 5 \pmod{8}\\
(1,1,1,1,1,1) & \ \text{if} \ n \equiv 7 \pmod{8}\\
\end{array},
\end{cases}
\end{eqnarray*}
and hence there are four cases to consider when computing the zeta functions of each of the curves~(\ref{eq:6coefficients}).

For $i_5 = 1$ the genus of all of the above curves is $50$\footnote[2]{See~\url{6CoefficientsGenus.m}.} and
therefore the brute-force computation of their zeta functions is non-trivial.
A curve-specific analysis may yield the zeta functions more efficiently, 
but since our algorithm is arguably more interesting than the explicit formulae, we leave this as an open problem.

\subsection{Computing $F_2(n,t_1,t_2,t_3,t_4,t_5,t_6,t_7)$}\label{subsec:F7}
Let $\vec{f} = (T_7,T_6,T_5,T_4,T_3,T_2,T_1)$. To determine $V( \vec{i} \cdot \vec{f})$ for $64 \le i \le 127$, we use Lemma~\ref{lem:T} part (7).
In particular, setting $\alpha = a_{0}^2$ and $\beta = a_{0}$ for $r_0 = 0$, and $\alpha = a_{0}^2 + a_0$ and $\beta = 1$ for $r_0 = 1$, 
and evaluating mod $2$ gives the following:
\begin{eqnarray*}
T_7(a_{0}^2 + a_{0} + r_0) &=& r_0T_3(a_{0}) + T_1(a_{0}^5)T_1(a_{0}^3)T_1(a_{0}) + r_0(T_2(a_{0}^3)T_1(a_{0}^3) + T_2(a_{0}^3)T_1(a_{0}) + T_2(a_{0})T_1(a_{0}^3))\\
&+& T_2(a_{0}^3)T_1(a_{0}^5) + T_2(a_{0}^3)T_1(a_{0}) + T_2(a_{0})T_1(a_{0}^5) + T_2(a_{0})T_1(a_{0}) + r_0(T_2(a_{0})\\
&+& T_2(a_{0}^3))\binom{n}{3} + r_0T_2(a_{0})+ r_0T_2(a_{0}^5) + T_1(a_{0}^9)T_1(a_{0}^3) + T_1(a_{0}^9)T_1(a_{0}) + T_1(a_{0}^7)T_1(a_{0}^5)\\
&+& r_0T_1(a_{0}^7)T_1(a_{0}^3) + (r_0 + 1)T_1(a_{0}^7)T_1(a_{0}) + \Big(r_0 + 1 + r_0\binom{n}{2}\Big)T_1(a_{0}^5)T_1(a_{0}^3)\\
&+& \Big(1 + r_0\binom{n}{2}\Big)T_1(a_{0}^5)T_1(a_{0}) + \Big(r_0 + 1 + r_0\binom{n}{2} + r_0\binom{n}{3}\Big)T_1(a_{0}^3)T_1(a_{0})\\
&+& T_1\Big(a_{0}^{13} + (r_0+1)a_{0}^{11} + a_{0}^9 + r_0(a_{0}^7 + a_{0}^5 + a_{0}^3) + a_{0} + r_0(a_{0}^9 + a_{0}^3 + a_{0})\binom{n}{2}\\
&+& r_0a_{0}^7\binom{n}{3} + r_0(a_{0}^5 + a_{0})\binom{n}{4} + r_0(a_{0}^3 + a_0)\binom{n}{5} + r_0\binom{n}{7}\Big).
%T_7(a_{0}^2 + a_{0}) &=& T_1(a_{0}^5)T_1(a_{0}^3)T_1(a_{0}) + T_2(a_{0}^3)T_1(a_{0}^5) + T_2(a_{0}^3)T_1(a_{0}) + T_2(a_{0})T_1(a_{0}^5) + %T_2(a_{0})T_1(a_{0})\\
%&&+ T_1(a_{0}^9)T_1(a_{0}^3) + T_1(a_{0}^9)T_1(a_{0}) +  T_1(a_{0}^7)T_1(a_{0}^5) + T_1(a_{0}^7)T_1(a_{0}) + T_1(a_{0}^5)T_1(a_{0}^3)\\
%&& + T_1(a_{0}^5)T_1(a_{0}) +T_1(a_{0}^3)T_1(a_{0}) + T_1(a_{0}^{13} + a_{0}^{11} + a_{0}^9 + a_{0}),\\
%T_7(a_{0}^2 + a_{0} + 1) &=& T_1(a_{0}^5)T_1(a_{0}^3)T_1(a_{0}) + T_2(a_{0}^3)T_1(a_{0}^5) + T_2(a_{0}^3)T_1(a_{0}^3) + T_2(a_{0})T_1(a_{0}^5) + %T_2(a_{0})T_1(a_{0}^3)\\
%&& + T_2(a_{0})T_1(a_{0}) + T_3(a_{0}) + T_2(a_{0}^5) +\binom{n}{3}T_2(a_{0}^3) +\Big(\binom{n}{3} +1\Big)T_2(a_0) +  T_1(a_{0}^9)T_1(a_{0}^3)\\
%&& + T_1(a_{0}^9)T_1(a_{0}) + T_1(a_{0}^7)T_1(a_{0}^5) + T_1(a_{0}^7)T_1(a_{0}^3) + T_1(a_{0}^5)T_1(a_{0}) +\binom{n}{2} (T_1(a_{0}^5)T_1(a_{0}^3)\\
%&& + T_1(a_{0}^5)T_1(a_{0})) + \Big(\binom{n}{3} + \binom{n}{2}\Big)T_1(a_{0}^3)T_1(a_{0}) + T_1\Big( a_{0}^{13} + \Big(\binom{n}{2} + 1\Big)a_{0}^9 \\
%&&+ \Big(\binom{n}{3}  + 1 \Big)a_{0}^7   + \Big(\binom{n}{4} + 1\Big)a_{0}^5 + \Big(\binom{n}{5}  + \binom{n}{2} + 1\Big)a_{0}^3\\
%&& + \Big( \binom{n}{5} +  \binom{n}{4} + \binom{n}{2} + 1 \Big)a_{0} + \binom{n}{7} \Big).
\end{eqnarray*}
As in the six coefficient case, this can be reduced to a $T_1$ expression using the substitutions $a_0 = a_{1}^2 + a_1 + r_1$, 
$a_{0}^3 = a_{2}^2 + a_2 + r_2$ and $a_{0}^5 = a_{3}^2 + a_3 + r_3$, where $r_1,r_2,r_3 \in \F_2$ are the traces of $a_0,a_{0}^3$ 
and $a_{0}^5$ respectively. This results in 
\begin{eqnarray*}
T_7(a_{0}^2 + a_0 + r_0) &=& T_1\Big(r_0(a_{3}^3 + a_{3}) + (r_0r_1 + r_0r_2 + r_1 + r_3 + r_0\binom{n}{3})(a_{2}^3 + a_2) + r_0(a_{1}^5 + a_{1}^3)\\
&+& (r_0r_1 + r_0r_2 + r_1 + r_3 + r_0\binom{n}{3})(a_{1}^3 + a_1) + a_{0}^{13} + (r_0 + 1)a_{0}^{11}\\
&+& \Big(r_1 + r_2 + 1 + r_0\binom{n}{2}\Big)a_{0}^9 + \Big(r_0 + r_1 + r_3 + r_0r_1 + r_0r_2 + r_0\binom{n}{3}  \Big)a_{0}^7\\
&+& r_1 + r_0r_2 + r_0r_3 + r_1r_2 + r_1r_3 + r_2r_3 + r_0r_1r_2 + r_0r_2r_3 + r_1r_2r_3\\
&+& \Big(r_1 + r_0r_3 + r_1r_2 + r_1r_3 + r_2r_3 + r_0r_1r_2 + r_0r_1r_3 + r_0r_2r_3\Big)\binom{n}{2}\\
%\end{eqnarray*}
%\begin{eqnarray*}
&+& \Big( r_0r_1 + r_0r_1r_2 \Big)\binom{n}{3} + \Big(r_0r_1 + r_0r_2 \Big)\binom{n}{2}\binom{n}{3} + \Big( r_0r_1 + r_0r_3\Big)\binom{n}{4}\\
&+& \Big( r_0r_1 + r_0r_2 \Big)\binom{n}{5} + r_0\binom{n}{7}\Big).
%T_7(a_{0}^2 + a_0) &=& T_1\big( (r_0 + r_2)(a_{2}^3 + a_2 + a_{1}^3 + a_1 + a_{0}^7) + a_{0}^{13} + a_{0}^{11}  + (r_0 + r_1 + 1)a_{0}^9\\
%&+& \Big(\binom{n}{2} + 1\Big)(r_0 r_1 + r_0 r_2 + r_1 r_2 + r_0) + r_0 r_1 r_2  \big)\\
%T_7(a_{0}^2 + a_0 + 1) &=& T_1\big( a_{3}^3 + a_3  + \binom{n}{3}(a_{2}^3 + a_2 + a_{1}^3 + a_1) + (r_0 + r_2)(a_{1}^3 + a_1) 
%+ (r_1 + r_2)(a_{2}^3 + a_2)\\
%&+& (r_0 + r_1)(a_{1}^3 + a_{1}) + a_{1}^5 + a_{1}^3 + a_{0}^{13} + (r_0 + r_1 + 1 + \binom{n}{2})a_{0}^9 + (r_1 + r_2 + 1 + \binom{n}{3})a_{0}^7\\
%&+& \binom{n}{7} + (r_0 + r_1)\binom{n}{5} + (r_0 + r_2)\binom{n}{4} + (r_0 + r_0r_1)\binom{n}{3} + (r_0 + r_1)\binom{n}{2}\binom{n}{3}\\
%&+& (r_0 + r_2)\binom{n}{2} + r_0 + r_1 + r_2 + r_0 r_2 + r_0 r_1 r_2 \big)
\end{eqnarray*}
For $64 \le i \le 127$ let $\vec{i} = (i_6,i_5,i_4,i_3,i_2,i_1,i_0)$. 
Applying Transform 3, the curves we are interested in are given by the following intersections:
\begin{eqnarray*}
a_{4}^2 + a_{4} &=& 
i_6\Big(r_0(a_{3}^3 + a_{3}) + (r_0r_1 + r_0r_2 + r_1 + r_3 + r_0\binom{n}{3})(a_{2}^3 + a_2) + r_0(a_{1}^5 + a_{1}^3)\\
&+& (r_0r_1 + r_0r_2 + r_1 + r_3 + r_0\binom{n}{3})(a_{1}^3 + a_1) + a_{0}^{13} + (r_0 + 1)a_{0}^{11}\\
&+& \Big(r_1 + r_2 + 1 + r_0\binom{n}{2}\Big)a_{0}^9 + \Big(r_0 + r_1 + r_3 + r_0r_1 + r_0r_2 + r_0\binom{n}{3}  \Big)a_{0}^7\\
&+& r_1 + r_0r_2 + r_0r_3 + r_1r_2 + r_1r_3 + r_2r_3 + r_0r_1r_2 + r_0r_2r_3 + r_1r_2r_3\\
&+& \Big(r_1 + r_0r_3 + r_1r_2 + r_1r_3 + r_2r_3 + r_0r_1r_2 + r_0r_1r_3 + r_0r_2r_3\Big)\binom{n}{2}\\
&+& \Big( r_0r_1 + r_0r_1r_2 \Big)\binom{n}{3} + \Big(r_0r_1 + r_0r_2 \Big)\binom{n}{2}\binom{n}{3} + \Big( r_0r_1 + r_0r_3\Big)\binom{n}{4}\\
\end{eqnarray*}
\begin{eqnarray}
\nonumber &+& \Big( r_0r_1 + r_0r_2 \Big)\binom{n}{5} + r_0\binom{n}{7}\Big)\\
\nonumber &+& i_5\Big( a_{3}^3 + a_3 + (r_1 + r_2)(a_{2}^3 + a_2 + a_{1}^3 + a_1 + a_{0}^7) + a_{1}^5 + a_{1}^3 + a_{0}^{11} + a_{0}^7
+ r_0r_1 + r_0r_2 \\
\nonumber &+& r_1r_2 + r_2r_3 + \big(r_0(a_{2}^3 + a_2 + a_{1}^3 + a_1 + a_{0}^7 + r_1 + r_2 + r_1r_2 +1) +r_1 + r_2 + r_3\big)\binom{n}{2}\\
\nonumber &+& (r_0r_1 + r_0r_3 + r_1)\binom{n}{3} + r_0(r_1 + r_2)\binom{n}{4} + r_0\binom{n}{6}\Big)\\
\nonumber &+& i_4\Big(r_0(a_{2}^3 + a_{2} + a_{1}^3 + a_{1}) + a_{0}^9 + r_0a_{0}^7 + (r_1 + r_2)a_{0}^5 + r_1 + r_2 + r_1r_2 + r_0r_1r_2\\
\nonumber &+& (r_0a_{0}^5 + r_0r_2)\binom{n}{2} + (r_0r_1 + r_0r_2)\binom{n}{3} + r_0\binom{n}{5}\Big)\\
\nonumber &+& i_3\Big(a_{2}^3 + a_2 + a_{1}^3 + a_1 + a_{0}^7 + a_{0}^5 + (r_0 + 1)(r_1 + r_2)\binom{n}{2} + r_0r_1 + r_0r_2 + r_1r_2 + r_2 + r_0\binom{n}{4}\Big)\\
\nonumber &+& i_2\Big(a_{0}^5 + a_{0} + r_0\Big(a_{0}^3 + a_0 + \binom{n}{3} \Big)\Big) + i_1\Big(a_{0}^3 + a_{0} + r_0\binom{n}{2}\Big) + i_0r_0,\\
\nonumber a_0 &=& a_{1}^2 + a_1 + r_1,\\
\label{eq:7coefficients} a_{0}^3 &=& a_{2}^2 + a_2 + r_2.\\
\nonumber a_{0}^5 &=& a_{3}^2 + a_3 + r_3.
\end{eqnarray}
Corollary~\ref{cor:period} implies that mod $2$ one has
\begin{eqnarray*}
\Big(\binom{n}{7},\binom{n}{6},\binom{n}{5},\binom{n}{4},\binom{n}{3},\binom{n}{2},\binom{n}{1}\Big) \equiv \begin{cases}
\begin{array}{lr}
(0,0,0,0,0,0,1) & \ \text{if} \ n \equiv 1 \pmod{8}\\
(0,0,0,0,1,1,1) & \ \text{if} \ n \equiv 3 \pmod{8}\\
(0,0,1,1,0,0,1) & \ \text{if} \ n \equiv 5 \pmod{8}\\
(1,1,1,1,1,1,1) & \ \text{if} \ n \equiv 7 \pmod{8}\\
\end{array},
\end{cases}
\end{eqnarray*}
and hence there are four cases to consider when computing the zeta functions of each of the curves~(\ref{eq:7coefficients}).

For $i_6 = 1$, the genus of each of the above curves is $58$\footnote[2]{See~\url{7CoefficientsGenus.m}.}.
Therefore, we again leave it as an open problem to determine their zeta functions.

%Note that there are alternative ways to linearise the expressions for $T_{l}(a_{0}^2 + a_{0} + r_0)$, which may result
%in curves of different genera. However, once the L-polynomials have been computed and Transform 2 is applied, they all must result in the same
%expressions for $F_2(n,t_1,\ldots,t_l)$.

%-----------------------------------------------------------------
%-----------------------------------------------------------------

\subsection{Reducing the prescribed coefficients problem to the prescribed traces problem}\label{sec:transformPCPPTP}

In this section we give formulae for all $I_2(n,t_1,t_2,t_3,t_4)$ in terms of the $F_2(n,t_1,t_2,t_3,t_4)$.
%While to the author's knowledge an algorithm to compute these transforms has not yet been written down in full generality, 
%the approaches in the literature are expected to be extendable to any particular case. Existing examples include the transform 
%between: $I_2(n,t_1,t_2)$ and $F_2(n,t_1,t_2)$~\cite{cattell}, which was the first paper to develop the transform; $I_2(n,t_1,t_2,t_3)$ 
%and $F_2(n,t_1,t_2,t_3)$, for $n$ odd~\cite{fitzyucas} and for $n$ even~\cite{yucasmullen};
%$I_{2^r}(n,t_1,t_2)$ and $F_{2^r}(n,t_1,t_2)$~\cite{RMR}; $I_3(n,t_1,*,t_3)$ and $F_3(n,t_1,*,t_3)$~\cite{Sharma};
%and $I_{2^r}(n,0,0,0)$ and $F_{2^r}(n,0,0,0)$~\cite{AGGMY}.
In the work~\cite{Koma}, Koma and Panario gave explicit formulae for the transform between $I_2(n,t_1,\ldots,t_l)$ and $F_2(n,t_1,\ldots,t_l)$ 
for $l = 4$ and $l = 5$; the extension to the $l = 6$ and $l = 7$ cases was also explained and a sketch of how to compute the 
transform for arbitrary $l \ge 8$ was provided.
However, there is a small error in the application of the multinomial theorem which unfortunately invalidates several lemmas, propositions and 
some of the formulae. Since in this work we have provided formulae for $F_2(n,t_1,t_2,t_3,t_4)$ we now provide the correct formulae for
$I_2(n,t_1,t_2,t_3,t_4)$.
%For $\beta \in \F_{2^n}$ let $f = \text{Min}(\beta)$ denote the minimum polynomial of $\beta$ over $\F_2$, which has degree $n/d$ for some $d \mid n$.
%Note that $T_i(\beta)$ is the coefficient of $x^{n-i}$ in $f^d$~\cite[Lemma 2]{cattell}, so abusing notation slightly we also write $T_i(\beta)$ as 
%$T_i(f^d)$. We use the following extension of~\cite[Prop. 1]{yucasmullen}, all parts of which follow by simply expanding $f^d$ using the multinomial
%theorem and comparing coefficients (cf. Lemma~\ref{lemma:multinomial}).
The following lemma is simply a subcase of Lemma~\ref{lemma:multinomial}.

\begin{lemma}\label{lemma:multinomial2}
For each integer $d \ge 1$ and $f(x) \in \F_2[x]$,
\begin{enumerate}
\item $T_1(f^d) = d T_1(f)$
\item $T_2(f^d) = \binom{d}{2} T_1(f) + d T_2(f)$
\item $T_3(f^d) = \binom{d}{3} T_1(f) + d T_3(f)$
\item $T_4(f^d) = \binom{d}{4} T_1(f) + \binom{d}{2} T_2(f) + dT_4(f) + (d-2)\binom{d}{2} T_1(f)T_2(f)$
\end{enumerate}
\end{lemma}
It is the final term of Lemma~\ref{lemma:multinomial2} part (4) that is missing from~\cite[Prop. 2.1]{Koma}, which has coefficient $1$ for $d \equiv 3 \pmod{4}$.

Recall from Corollary~\ref{cor:period} that the featured binomial coefficients have maximal period $8$ and hence we need to consider the residue
classes of $d$ mod $8$. 
For brevity we write $d \equiv a \, (8)$ to denote $d \equiv a \pmod{8}$ and similarly write  $d \equiv a \, (4)$ to denote $d \equiv a \pmod{4}$.
As in~\cite{cattell} let ${\bf Irr}(n)$ denote the set of all irreducible polynomials of degree $n$ over $\F_2$, 
and let $a \cdot {\bf Irr}(n)$ denote the multiset consisting of $a$ copies of ${\bf Irr}(n)$. 
Following~\cite{cattell} and~\cite{yucasmullen}, applying Lemma~\ref{lemma:multinomial2} gives the following result.

\begin{lemma}\label{lemma:4coefftransform0}
For $n \ge 4$ we have
\begin{eqnarray}
\nonumber F_2(n,t_1,t_2,t_3,t_4) &=& \Big\vert \bigcup_{\beta \in \F_{2^n}, T_1(\beta) = t_1, T_2(\beta) = t_2, T_3(\beta) = t_3, T_4(\beta) = t_4} Min(\beta) \Big\vert\\
\nonumber &=& \Big\vert \bigcup_{d \mid n} \frac{n}{d} \big\{ f \in {\bf Irr}(\frac{n}{d}):  dT_1(f) = t_1, \binom{d}{2}T_1(f) + d T_2(f) = t_2, 
\binom{d}{3}T_1(f) + d T_3(f) = t_3\\
\nonumber &&  \binom{d}{4} T_1(f) + \binom{d}{2} T_2(f) + dT_4(f) + (d-2)\binom{d}{2} T_1(f)T_2(f) = t_4 \big\} \Big\vert \\                                 
\nonumber                    &=& \Big\vert \bigcup_{d \mid n, \ d \equiv 0 \, (8)} \frac{n}{d} \big\{ f \in {\bf Irr}(\frac{n}{d}):  
                                 0 = t_1, 0 = t_2, 0 = t_3, 0 = t_4 \big\} \Big\vert \\
\nonumber        	            &+& \Big\vert \bigcup_{d \mid n, \ d \equiv 1 \, (8)} \frac{n}{d} \big\{ f \in {\bf Irr}(\frac{n}{d}):  
                                 T_1(f) = t_1, T_2(f) = t_2, T_3(f) = t_3, T_4(f) = t_4 \big\} \Big\vert \\ 
\nonumber        	            &+& \Big\vert \bigcup_{d \mid n, \ d \equiv 2 \, (8)} \frac{n}{d} \big\{ f \in {\bf Irr}(\frac{n}{d}):  
                                 0 = t_1, T_1(f) = t_2, 0 = t_3, T_2(f) = t_4 \big\} \Big\vert \\
\nonumber        	            &+& \Big\vert \bigcup_{d \mid n, \ d \equiv 3 \, (8)} \frac{n}{d} \big\{ f \in {\bf Irr}(\frac{n}{d}):  
                                 T_1(f) = t_1, T_1(f) + T_2(f) = t_2, T_1(f) + T_3(f) = t_3,\\
\nonumber &&                    T_2(f) + T_4(f) + T_1(f)T_2(f) = t_4 \big\} \Big\vert \\
\nonumber        	            &+& \Big\vert \bigcup_{d \mid n, \ d \equiv 4 \, (8)} \frac{n}{d} \big\{ f \in {\bf Irr}(\frac{n}{d}):  
                                 0 = t_1, 0 = t_2, 0 = t_3, T_1(f) = t_4 \big\} \Big\vert \\ 
\nonumber        	            &+& \Big\vert \bigcup_{d \mid n, \ d \equiv 5 \, (8)} \frac{n}{d} \big\{ f \in {\bf Irr}(\frac{n}{d}):  
                                 T_1(f) = t_1, T_2(f) = t_2, T_3(f) = t_3, T_1(f) + T_4(f) = t_4 \big\} \Big\vert \\ 
\nonumber        	            &+& \Big\vert \bigcup_{d \mid n, \ d \equiv 6 \, (8)} \frac{n}{d} \big\{ f \in {\bf Irr}(\frac{n}{d}):  
                                 0 = t_1, T_1(f) = t_2, 0 = t_3, T_1(f) + T_2(f) = t_4 \big\} \Big\vert \\                                  
\nonumber        	            &+& \Big\vert \bigcup_{d \mid n, \ d \equiv 7 \, (8)} \frac{n}{d} \big\{ f \in {\bf Irr}(\frac{n}{d}):  
                                 T_1(f) = t_1, T_1(f) + T_2(f) = t_2, T_1(f) + T_3(f) = t_3,\\
\nonumber &&                    T_1(f) + T_2(f) + T_4(f) + T_1(f)T_2(f) = t_4 \big\} \Big\vert.                 
\end{eqnarray}
\end{lemma}
Evaluating Lemma~\ref{lemma:4coefftransform0} at each set of traces and applying the same arguments as given in~\cite{Koma} 
and~\cite[Chap. 5]{Komathesis} {\em mutatis mutandis} leads to Theorem~\ref{thm:transform0binary4}. The impact of the extra term of 
Lemma~\ref{lemma:multinomial2} part (4) affects only parts $(2)$, $(4)$, $(6)$ and $(8)$ -- switching the $d \equiv 3$ and $d \equiv 7$ terms --
but we include all of them for the sake of completeness. 

\begin{theorem}\label{thm:transform0binary4}
For $n \ge 4$ we have
\begin{enumerate}[label={(\arabic*)}]
\item $\displaystyle\begin{aligned}[t]
n I_2(n,1,1,1,0) &=& \sum_{\substack{d \mid n \\ d \equiv 1 \, (8)}} \mu(d) F_2(n/d,1,1,1,0) + \sum_{\substack{d \mid n \\ d \equiv 3 \, (8)}} \mu(d) F_2(n/d,1,0,0,0)\\
&+&  \sum_{\substack{d \mid n \\ d \equiv 5 \, (8)}} \mu(d) F_2(n/d,1,1,1,1) + \sum_{\substack{d \mid n \\ d \equiv 7 \, (8)}} \mu(d) F_2(n/d,1,0,0,1),
\end{aligned}$
\item $\displaystyle\begin{aligned}[t]
n I_2(n,1,0,0,0) &=& \sum_{\substack{d \mid n \\ d \equiv 1 \, (8)}} \mu(d) F_2(n/d,1,0,0,0) + \sum_{\substack{d \mid n \\ d \equiv 3 \, (8)}} \mu(d) F_2(n/d,1,1,1,0)\\
&+&  \sum_{\substack{d \mid n \\ d \equiv 5 \, (8)}} \mu(d) F_2(n/d,1,0,0,1) + \sum_{\substack{d \mid n \\ d \equiv 7 \, (8)}} \mu(d) F_2(n/d,1,1,1,1),
\end{aligned}$
\item $\displaystyle\begin{aligned}[t]
n I_2(n,1,1,1,1) &=& \sum_{\substack{d \mid n \\ d \equiv 1 \, (8)}} \mu(d) F_2(n/d,1,1,1,1) + \sum_{\substack{d \mid n \\ d \equiv 3 \, (8)}} \mu(d) F_2(n/d,1,0,0,1)\\
&+&  \sum_{\substack{d \mid n \\ d \equiv 5 \, (8)}} \mu(d) F_2(n/d,1,1,1,0) + \sum_{\substack{d \mid n \\ d \equiv 7 \, (8)}} \mu(d) F_2(n/d,1,0,0,0),\\
\end{aligned}$
\item $\displaystyle\begin{aligned}[t]
n I_2(n,1,0,0,1) &=& \sum_{\substack{d \mid n \\ d \equiv 1 \, (8)}} \mu(d) F_2(n/d,1,0,0,1) + \sum_{\substack{d \mid n \\ d \equiv 3 \, (8)}} \mu(d) F_2(n/d,1,1,1,1)\\
&+&  \sum_{\substack{d \mid n \\ d \equiv 5 \, (8)}} \mu(d) F_2(n/d,1,0,0,0) + \sum_{\substack{d \mid n \\ d \equiv 7 \, (8)}} \mu(d) F_2(n/d,1,1,1,0),\\
\end{aligned}$
\item $\displaystyle\begin{aligned}[t]
n I_2(n,1,1,0,0) &=& \sum_{\substack{d \mid n \\ d \equiv 1 \, (8)}} \mu(d) F_2(n/d,1,1,0,0) + \sum_{\substack{d \mid n \\ d \equiv 3 \, (8)}} \mu(d) F_2(n/d,1,0,1,0)\\
&+&  \sum_{\substack{d \mid n \\ d \equiv 5 \, (8)}} \mu(d) F_2(n/d,1,1,0,1) + \sum_{\substack{d \mid n \\ d \equiv 7 \, (8)}} \mu(d) F_2(n/d,1,0,1,1),\\
\end{aligned}$
\item $\displaystyle\begin{aligned}[t]
n I_2(n,1,0,1,0) &=& \sum_{\substack{d \mid n \\ d \equiv 1 \, (8)}} \mu(d) F_2(n/d,1,0,1,0) + \sum_{\substack{d \mid n \\ d \equiv 3 \, (8)}} \mu(d) F_2(n/d,1,1,0,0)\\
&+&  \sum_{\substack{d \mid n \\ d \equiv 5 \, (8)}} \mu(d) F_2(n/d,1,0,1,1) + \sum_{\substack{d \mid n \\ d \equiv 7 \, (8)}} \mu(d) F_2(n/d,1,1,0,1),\\
\end{aligned}$
\item $\displaystyle\begin{aligned}[t]
n I_2(n,1,1,0,1) &=& \sum_{\substack{d \mid n \\ d \equiv 1 \, (8)}} \mu(d) F_2(n/d,1,1,0,1) + \sum_{\substack{d \mid n \\ d \equiv 3 \, (8)}} \mu(d) F_2(n/d,1,0,1,1)\\
&+&  \sum_{\substack{d \mid n \\ d \equiv 5 \, (8)}} \mu(d) F_2(n/d,1,1,0,0) + \sum_{\substack{d \mid n \\ d \equiv 7 \, (8)}} \mu(d) F_2(n/d,1,0,1,0),\\
\end{aligned}$
\item $\displaystyle\begin{aligned}[t]
n I_2(n,1,0,1,1) &=& \sum_{\substack{d \mid n \\ d \equiv 1 \, (8)}} \mu(d) F_2(n/d,1,0,1,1) + \sum_{\substack{d \mid n \\ d \equiv 3 \, (8)}} \mu(d) F_2(n/d,1,1,0,1)\\
&+&  \sum_{\substack{d \mid n \\ d \equiv 5 \, (8)}} \mu(d) F_2(n/d,1,0,1,0) + \sum_{\substack{d \mid n \\ d \equiv 7 \, (8)}} \mu(d) F_2(n/d,1,1,0,0),\\
\end{aligned}$
\item $\displaystyle\begin{aligned}[t]
n I_2(n,0,0,1,0) &=& \sum_{\substack{d \mid n \\ d \ \text{odd}}} \mu(d) F_2(n/d,0,0,1,0),\\
\end{aligned}$
\item $\displaystyle\begin{aligned}[t]
n I_2(n,0,0,1,1) &=& \sum_{\substack{d \mid n \\ d \ \text{odd}}} \mu(d) F_2(n/d,0,0,1,1),\\
\end{aligned}$
\item $\displaystyle\begin{aligned}[t]
n I_2(n,0,1,1,1) &=& \sum_{\substack{d \mid n \\ d \equiv 1 \, (4)}} \mu(d) F_2(n/d,0,1,1,1) + \sum_{\substack{d \mid n \\ d \equiv 3 \, (4)}} \mu(d) F_2(n/d,0,1,1,0),\\
\end{aligned}$
\item $\displaystyle\begin{aligned}[t]
n I_2(n,0,1,1,0) &=& \sum_{\substack{d \mid n \\ d \equiv 1 \, (4)}} \mu(d) F_2(n/d,0,1,1,0) + \sum_{\substack{d \mid n \\ d \equiv 3 \, (4)}} \mu(d) F_2(n/d,0,1,1,1),\\
\end{aligned}$
\item $\displaystyle\begin{aligned}[t]
n I_2(n,0,0,0,0) &=& \sum_{\substack{d \mid n \\ d \ \text{odd}}} \mu(d) F_2(n/d,0,0,0,0) - \sum_{\substack{d \mid n, \ \frac{n}{d} \ \text{even} \\ d \ \text{odd}}} \mu(d) F_2(n/2d,0,0),\\
\end{aligned}$
\item $\displaystyle\begin{aligned}[t]
n I_2(n,0,0,0,1) &=& \sum_{\substack{d \mid n \\ d \ \text{odd}}} \mu(d) F_2(n/d,0,0,0,1) - \sum_{\substack{d \mid n, \ \frac{n}{d} \ \text{even} \\ d \ \text{odd}}} \mu(d) F_2(n/2d,0,1),\\
\end{aligned}$
\item $\displaystyle\begin{aligned}[t]
n I_2(n,0,1,0,0) &=& \sum_{\substack{d \mid n \\ d \equiv 1 \, (4)}} \mu(d) F_2(n/d,0,1,0,0) - \sum_{\substack{d \mid n, \ \frac{n}{d} \ \text{even} \\ d \equiv 1 \, (4)}} \mu(d) F_2(n/2d,1,0)\\
&+& \sum_{\substack{d \mid n \\ d \equiv 3 \, (4)}} \mu(d) F_2(n/d,0,1,0,1) - \sum_{\substack{d \mid n, \ \frac{n}{d} \ \text{even} \\ d \equiv 3 \, (4)}} \mu(d) F_2(n/2d,1,1),\\
\end{aligned}$
\item $\displaystyle\begin{aligned}[t]
n I_2(n,0,1,0,1) &=& \sum_{\substack{d \mid n \\ d \equiv 1 \, (4)}} \mu(d) F_2(n/d,0,1,0,1) - \sum_{\substack{d \mid n, \ \frac{n}{d} \ \text{even} \\ d \equiv 1 \, (4)}} \mu(d) F_2(n/2d,1,1)\\
&+& \sum_{\substack{d \mid n \\ d \equiv 3 \, (4)}} \mu(d) F_2(n/d,0,1,0,0) - \sum_{\substack{d \mid n, \ \frac{n}{d} \ \text{even} \\ d \equiv 3 \, (4)}} \mu(d) F_2(n/2d,1,0).\\
\end{aligned}$
\end{enumerate}
\end{theorem}

\subsubsection{Formulae for $l \ge 5$.}
The formulae for five coefficients follow from an argument analogous to that given in Lemma~\ref{lemma:4coefftransform0}, 
using the identity
\begin{equation}\label{eq:5coeffsT}
T_5(f^d) = \binom{d}{5}T_1(f) + dT_5(f) + (d-2)\binom{d}{2}\big( T_1(f)T_2(f) + T_1(f)T_3(f)\big).
\end{equation}
We omit the details since they are not part of the primary contribution of this work and follow easily. 
Note that Eq.~(\ref{eq:5coeffsT}) is at odds with Eq.~(3) of~\cite{Koma} which claims that for any $l \ge 1$ and $d \ge 1$ the following holds mod $2$: 
\begin{equation*}\label{eq:badmultinomial}
T_l(f^d) = \sum_{k \mid l} \binom{d}{k} T_{l/k}(f),
\end{equation*}
which contradicts Lemma~\ref{lemma:multinomial}. If combined with Lemma~\ref{lemma:multinomial}, the approach sketched in~\cite[\S4.2]{Koma} for 
arbitrary $l$, which uses the generalised M\"obius inversion of Miers and Ruskey~\cite[Theorem 3.2]{miersruskey2}, can no doubt be developed into a 
general algorithm for computing the required transforms. However, we leave describing such an algorithm, for $q=2$ and for arbitrary $q$ and $l$, 
as an open problem.

%------------------------------------------------------------------------------------------

\section{Curves and Explicit Formulae for $q = l = 3$}\label{sec:ternaryzetas}

In this section we determine curves and explicit formulae for all $F_3(n,t_1,t_2,t_3)$ for $n \ge 3$ but coprime to $3$ and for $F_3(n,0,0,0)$ for all 
$n \ge 3$. We do this using the indirect and direct methods, and one other method, in order to highlight some redundancy 
in the formulae arising from the indirect method, \ie linear relations between powers of the roots of the featured characteristic polynomials of Frobenius.
Since the Galois groups of all the featured polynomials are soluble, eliminating such redundancy 
is feasible, by identifying cancellations between their roots for various residues of $n$ mod $9$.
On the other hand the direct method, although producing curves with harder-to-compute zeta functions, seems to have no such redundancy. 
This implies that there is a trade-off between the ease of computation and the compactness of the formulae. 

The only previous result on such counts over $\F_3$ are due to Sharma~\ea~\cite{Sharma}, who gave approximations to $F_3(n,t_1,*,t_3)$ using
analogous simplifications to those used in~\cite{Koma,Komathesis}. Note that~\cite{Sharma} contains the transforms to $I_3(n,t_1,*,t_3)$ and
hence one can use the results of this section to compute these counts.

The reason we do not compute formulae for four coefficients is that the terms $T_4(\alpha)$ and $T_4(\beta)$ in Lemma~\ref{lem:T} part (4) 
do not cancel mod $3$ and the method of introducing new variables in order to linearise $T_4(a_{0}^3 - a_{0} + r_0)$ fails as a result.

\noindent In order to express $F_3(n,t_1,t_2,t_3)$ compactly, we define the following eight polynomials:
\begin{eqnarray*}
\epsilon_{2,1} &=& X^2 - 3X + 3,\\
\epsilon_{2,2} &=& X^2 + 3X + 3,\\
\epsilon_{2,3} &=& X^2 + 3,\\
\epsilon_6 &=& X^6 + 3X^5 + 9X^4 + 15X^3 + 27X^2 + 27X + 27,\\
\epsilon_{12,1} &=& X^{12} - 3X^{11} + 3X^9 + 9X^8 - 45X^6 + 81X^4 + 81X^3 - 729X + 729,\\
\epsilon_{12,2} &=& X^{12} - 3X^{11} + 12X^9 - 18X^8 - 27X^7 + 117X^6 - 81X^5 - 162X^4 + 324X^3 - 729X + 729,\\
\epsilon_{12,3} &=& X^{12} - 3X^{11} + 9X^{10} - 15X^9 + 36X^8 - 54X^7 + 117X^6 - 162X^5 + 324X^4 - 405X^3 + 729X^2\\
&&- 729X + 729,\\
\epsilon_{12,4} &=& X^{12} + 6X^{11} + 18X^{10} + 39X^9 + 63X^8 + 81X^7 + 117X^6 + 243X^5 + 567X^4 + 1053X^3 + 1458X^2\\
&& + 1458X + 729.
\end{eqnarray*}
Observe that by Theorem~\ref{thm:SS}, $\epsilon_{6}, \epsilon_{12,1}, \epsilon_{12,2}, \epsilon_{12,3}$ and $\epsilon_{12,4}$ are not the 
characteristic polynomials of the Frobenius endomorphism of supersingular abelian varieties.
As in~\S\ref{sec:binaryzetas} we use Lemma~\ref{lem:T} parts (1) to (3). Also as 
in~\S\ref{sec:binaryzetas}, it is more efficient to compute linearised forms for $T_{2}(a)$ and $T_3(a)$ and to combine them as appropriate for 
each $i \in \{1,\ldots,3^3-1\}$. 

\subsection{Indirect method}

Setting $\vec{f} = (T_3,T_2,T_1)$, by the transform~(\ref{eq:Sminus1}) we have
\begin{equation}\label{eq:transform1char3}
{\tiny
\begin{bmatrix}
F_3(n,0,0,0)\\
F_3(n,1,0,0)\\
F_3(n,2,0,0)\\
F_3(n,0,1,0)\\
F_3(n,1,1,0)\\
F_3(n,2,1,0)\\
F_3(n,0,2,0)\\
F_3(n,1,2,0)\\
F_3(n,2,2,0)\\
F_3(n,0,0,1)\\
F_3(n,1,0,1)\\
F_3(n,2,0,1)\\
F_3(n,0,1,1)\\
F_3(n,1,1,1)\\
F_3(n,2,1,1)\\
F_3(n,0,2,1)\\
F_3(n,1,2,1)\\
F_3(n,2,2,1)\\
F_3(n,0,0,2)\\
F_3(n,1,0,2)\\
F_3(n,2,0,2)\\
F_3(n,0,1,2)\\
F_3(n,1,1,2)\\
F_3(n,2,1,2)\\
F_3(n,0,2,2)\\
F_3(n,1,2,2)\\
F_3(n,2,2,2)\\
\end{bmatrix}
=
\begin{bmatrix}
N(\vec{0})\\
N(\vec{1})\\
N(\vec{2})\\
N(\vec{3})\\
N(\vec{4})\\
N(\vec{5})\\
N(\vec{6})\\
N(\vec{7})\\
N(\vec{8})\\
N(\vec{9})\\
N(\vec{10})\\
N(\vec{11})\\
N(\vec{12})\\
N(\vec{13})\\
N(\vec{14})\\
N(\vec{15})\\
N(\vec{16})\\
N(\vec{17})\\
N(\vec{18})\\
N(\vec{19})\\
N(\vec{20})\\
N(\vec{21})\\
N(\vec{22})\\
N(\vec{23})\\
N(\vec{24})\\
N(\vec{25})\\
N(\vec{26})\\
\end{bmatrix}
=
\frac{1}{9}
\begin{bmatrix*}[r]
9 & - & - & - & - & - & - & - & - & - & - & - & - & - & - & - & - & - & - & - & - & - & - & - & - & - & -\\ 
0 & 1 & 0 & - & 1 & 0 & - & 1 & 0 & - & 1 & 0 & - & 1 & 0 & - & 1 & 0 & - & 1 & 0 & - & 1 & 0 & - & 1 & 0\\
0 & 0 & 1 & - & 0 & 1 & - & 0 & 1 & - & 0 & 1 & - & 0 & 1 & - & 0 & 1 & - & 0 & 1 & - & 0 & 1 & - & 0 & 1\\
0 & - & - & 1 & 1 & 1 & 0 & 0 & 0 & - & - & - & 1 & 1 & 1 & 0 & 0 & 0 & - & - & - & 1 & 1 & 1 & 0 & 0 & 0\\
0 & 1 & 0 & 1 & 0 & - & 0 & - & 1 & - & 1 & 0 & 1 & 0 & - & 0 & - & 1 & - & 1 & 0 & 1 & 0 & - & 0 & - & 1\\
0 & 0 & 1 & 1 & - & 0 & 0 & 1 & - & - & 0 & 1 & 1 & - & 0 & 0 & 1 & - & - & 0 & 1 & 1 & - & 0 & 0 & 1 & -\\
0 & - & - & 0 & 0 & 0 & 1 & 1 & 1 & - & - & - & 0 & 0 & 0 & 1 & 1 & 1 & - & - & - & 0 & 0 & 0 & 1 & 1 & 1\\
0 & 1 & 0 & 0 & - & 1 & 1 & 0 & - & - & 1 & 0 & 0 & - & 1 & 1 & 0 & - & - & 1 & 0 & 0 & - & 1 & 1 & 0 & -\\
0 & 0 & 1 & 0 & 1 & - & 1 & - & 0 & - & 0 & 1 & 0 & 1 & - & 1 & - & 0 & - & 0 & 1 & 0 & 1 & - & 1 & - & 0\\
0 & - & - & - & - & - & - & - & - & 1 & 1 & 1 & 1 & 1 & 1 & 1 & 1 & 1 & 0 & 0 & 0 & 0 & 0 & 0 & 0 & 0 & 0\\
0 & 1 & 0 & - & 1 & 0 & - & 1 & 0 & 1 & 0 & - & 1 & 0 & - & 1 & 0 & - & 0 & - & 1 & 0 & - & 1 & 0 & - & 1\\ 
0 & 0 & 1 & - & 0 & 1 & - & 0 & 1 & 1 & - & 0 & 1 & - & 0 & 1 & - & 0 & 0 & 1 & - & 0 & 1 & - & 0 & 1 & -\\
0 & - & - & 1 & 1 & 1 & 0 & 0 & 0 & 1 & 1 & 1 & 0 & 0 & 0 & - & - & - & 0 & 0 & 0 & - & - & - & 1 & 1 & 1\\ 
0 & 1 & 0 & 1 & 0 & - & 0 & - & 1 & 1 & 0 & - & 0 & - & 1 & - & 1 & 0 & 0 & - & 1 & - & 1 & 0 & 1 & 0 & -\\ 
0 & 0 & 1 & 1 & - & 0 & 0 & 1 & - & 1 & - & 0 & 0 & 1 & - & - & 0 & 1 & 0 & 1 & - & - & 0 & 1 & 1 & - & 0\\ 
0 & - & - & 0 & 0 & 0 & 1 & 1 & 1 & 1 & 1 & 1 & - & - & - & 0 & 0 & 0 & 0 & 0 & 0 & 1 & 1 & 1 & - & - & -\\
0 & 1 & 0 & 0 & - & 1 & 1 & 0 & - & 1 & 0 & - & - & 1 & 0 & 0 & - & 1 & 0 & - & 1 & 1 & 0 & - & - & 1 & 0\\
0 & 0 & 1 & 0 & 1 & - & 1 & - & 0 & 1 & - & 0 & - & 0 & 1 & 0 & 1 & - & 0 & 1 & - & 1 & - & 0 & - & 0 & 1\\
0 & - & - & - & - & - & - & - & - & 0 & 0 & 0 & 0 & 0 & 0 & 0 & 0 & 0 & 1 & 1 & 1 & 1 & 1 & 1 & 1 & 1 & 1\\
0 & 1 & 0 & - & 1 & 0 & - & 1 & 0 & 0 & - & 1 & 0 & - & 1 & 0 & - & 1 & 1 & 0 & - & 1 & 0 & - & 1 & 0 & -\\
0 & 0 & 1 & - & 0 & 1 & - & 0 & 1 & 0 & 1 & - & 0 & 1 & - & 0 & 1 & - & 1 & - & 0 & 1 & - & 0 & 1 & - & 0\\ 
0 & - & - & 1 & 1 & 1 & 0 & 0 & 0 & 0 & 0 & 0 & - & - & - & 1 & 1 & 1 & 1 & 1 & 1 & 0 & 0 & 0 & - & - & -\\
0 & 1 & 0 & 1 & 0 & - & 0 & - & 1 & 0 & - & 1 & - & 1 & 0 & 1 & 0 & - & 1 & 0 & - & 0 & - & 1 & - & 1 & 0\\
0 & 0 & 1 & 1 & - & 0 & 0 & 1 & - & 0 & 1 & - & - & 0 & 1 & 1 & - & 0 & 1 & - & 0 & 0 & 1 & - & - & 0 & 1\\ 
0 & - & - & 0 & 0 & 0 & 1 & 1 & 1 & 0 & 0 & 0 & 1 & 1 & 1 & - & - & - & 1 & 1 & 1 & - & - & - & 0 & 0 & 0\\ 
0 & 1 & 0 & 0 & - & 1 & 1 & 0 & - & 0 & - & 1 & 1 & 0 & - & - & 1 & 0 & 1 & 0 & - & - & 1 & 0 & 0 & - & 1\\
0 & 0 & 1 & 0 & 1 & - & 1 & - & 0 & 0 & 1 & - & 1 & - & 0 & - & 0 & 1 & 1 & - & 0 & - & 0 & 1 & 0 & 1 & -\\
\end{bmatrix*}
\begin{bmatrix}
V_1(\vec{0} \cdot \vec{f})\\
V_1(\vec{1} \cdot \vec{f})\\
V_1(\vec{2} \cdot \vec{f})\\
V_1(\vec{3} \cdot \vec{f})\\
V_1(\vec{4} \cdot \vec{f})\\
V_1(\vec{5} \cdot \vec{f})\\
V_1(\vec{6} \cdot \vec{f})\\
V_1(\vec{7} \cdot \vec{f})\\
V_1(\vec{8} \cdot \vec{f})\\
V_1(\vec{9} \cdot \vec{f})\\
V_1(\vec{10} \cdot \vec{f})\\
V_1(\vec{11} \cdot \vec{f})\\
V_1(\vec{12} \cdot \vec{f})\\
V_1(\vec{13} \cdot \vec{f})\\
V_1(\vec{14} \cdot \vec{f})\\
V_1(\vec{15} \cdot \vec{f})\\
V_1(\vec{16} \cdot \vec{f})\\
V_1(\vec{17} \cdot \vec{f})\\
V_1(\vec{18} \cdot \vec{f})\\
V_1(\vec{19} \cdot \vec{f})\\
V_1(\vec{20} \cdot \vec{f})\\
V_1(\vec{21} \cdot \vec{f})\\
V_1(\vec{22} \cdot \vec{f})\\
V_1(\vec{23} \cdot \vec{f})\\
V_1(\vec{24} \cdot \vec{f})\\
V_1(\vec{25} \cdot \vec{f})\\
V_1(\vec{26} \cdot \vec{f})\\
\end{bmatrix}.}
\end{equation}
By definition we have $V_1(\vec{0} \cdot \vec{f}) = 3^n$, while $V_1(\vec{1} \cdot \vec{f}) = V_1(T_1) = \#\{a \in \F_{3^n} \mid T_1(a) = 1\} = 3^{n-1}$, 
as does $V_1(\vec{2} \cdot \vec{f}) = V_1(2T_1) = \#\{a \in \F_{3^n} \mid 2T_1(a) = 1\}$. 
Note that for $r_0 \in \F_3$ one has $T_l(r_0) = \binom{n}{l}r_0$. To determine $V_1( \vec{i} \cdot \vec{f})$ for $3 \le i \le 26$, setting 
$\alpha = a_{0}^3 - a_{0}$ and $\beta = -r_0$, and using Lemma~\ref{lem:T} parts (1) to (3) mod $3$ gives the 
following\footnote[2]{See~\url{Ternary3Coefficients.mw}.}:
\begin{eqnarray}
\label{eq:ternary1} T_1(a_{0}^3 - a_0 + r_0) &=& T_1(r_0),\\
\label{eq:ternary2} T_2(a_{0}^3 - a_0 + r_0) &=& T_1\Big(a_{0}^4 - a_{0}^2 - r_0\binom{n}{2}/n\Big),\\
\label{eq:ternary3} T_3(a_{0}^3 - a_0 + r_0) &=& T_1\Big(a_{0}^7 - a_{0}^5 + r_0(n+1)(a_{0}^4 - a_{0}^2) + r_0\binom{n}{3}/n\Big).
\end{eqnarray}
For $3 \le i \le 26$ let $\vec{i} = (i_2,i_1,i_0)$. The curves we are interested in for $a$ of trace $r_0$ are
\begin{eqnarray}\label{eq:curvesternary} 
a_{1}^3 - a_{1} + 1/n &=& i_2\Big(a_{0}^7 - a_{0}^5 - r_0(n+1)(a_{0}^4 - a_{0}^2) 
- r_0\binom{n}{3}/n\Big) + i_1\Big(a_{0}^4 - a_{0}^2 + r_0\binom{n}{2}/n\Big) + i_0 r_0.
\end{eqnarray}
These curves have genus $3$ if $i_2 = 0$ and genus $6$ if $i_2 \neq 0$. Corollary~\ref{cor:period} implies that mod $3$ one has
\begin{eqnarray*}
\Big(\binom{n}{3},\binom{n}{2},\binom{n}{1}\Big) \equiv \begin{cases}
\begin{array}{lr}
(0,0,1) & \ \text{if} \ n \equiv 1 \pmod{9}\\
(0,1,-1) & \ \text{if} \ n \equiv 2 \pmod{9}\\
(1,0,1) & \ \text{if} \ n \equiv 4 \pmod{9}\\
(1,1,-1) & \ \text{if} \ n \equiv 5 \pmod{9}\\
(-1,0,1) & \ \text{if} \ n \equiv 7 \pmod{9}\\
(-1,1,-1) & \ \text{if} \ n \equiv 8 \pmod{9}\\
\end{array},
\end{cases}
\end{eqnarray*}
and hence there are six cases to consider when computing the zeta functions of each of the relevant curves.
Using Magma to compute the zeta functions of the curves~(\ref{eq:curvesternary}) and applying~(\ref{eq:transform1char3}) gives the following theorem\footnote[3]{See~\url{F3(n,t1,t2,t3).m} 
for generation and counting code, and~\url{F3(n,t1,t2,t3)verify.mw} for verification of $F_3(n,0,0,0)$, which can easily be adapted for the other cases.},
where $\vec{v} = (\rho_n(\epsilon_{2,1}),\rho_n(\epsilon_{2,2}),\rho_n(\epsilon_{2,3}),\rho_n(\epsilon_{6}),\rho_n(\epsilon_{12,1}),\rho_n(\epsilon_{12,2}),\rho_n(\epsilon_{12,3}),\rho_n(\epsilon_{12,4}))$.

\begin{theorem}\label{thm:char3_3traces}
For $n \ge 3$ we have
{\small
\begin{flalign*}
F_3(n,0,0,0) &= 3^{n-3} - \frac{1}{81}
(-21, -15, -18, -8, -14, -14, -14, -8)\cdot \vec{v} 
\ \ \text{if} \ n \equiv 1,2,4,5,7,8 \pmod{9}\\
F_3(n,1,0,0) &= 3^{n-3} - \frac{1}{81} \cdot \begin{cases}
\begin{array}{lr}
(-3, 3, 0, 4, -2, -2, -2, 4)\cdot\vec{v} & \ \ \text{if} \ n \equiv 1 \pmod{9}\\
( 3, 0, -3, 4, -2, 4, -2, -2)\cdot\vec{v} & \ \ \text{if} \ n \equiv 2 \pmod{9}\\
(-3, 3, 0, -2, 1, 1, 1, -2)\cdot\vec{v} & \ \ \text{if} \ n \equiv 4,7 \pmod{9}\\
( 3, 0, -3, -2, 1, -2, 1, 1)\cdot\vec{v} & \ \ \text{if} \ n \equiv 5,8 \pmod{9}\\
\end{array}
\end{cases}\\
F_3(n,2,0,0) &= 3^{n-3} - \frac{1}{81} \cdot \begin{cases}
\begin{array}{lr}
( -3, 3, 0, 4, -2, -2, -2, 4)\cdot\vec{v} & \ \text{if} \ n \equiv 1 \pmod{9}\\
(0, -3, 3, 4, -2, -2, 4, -2)\cdot\vec{v} & \ \text{if} \ n \equiv 2 \pmod{9}\\
( -3, 3, 0, -2, 1, 1, 1, -2 )\cdot\vec{v} & \ \text{if} \ n \equiv 4,7 \pmod{9}\\
( 0, -3, 3, -2, 1, 1, -2, 1)\cdot\vec{v} & \ \ \text{if} \ n \equiv 5,8 \pmod{9}\\
\end{array}
\end{cases}\\
F_3(n,0,1,0) &= 3^{n-3} - \frac{1}{81} \cdot \begin{cases}
\begin{array}{lr}
( 12, 9, 6, 4, -2, 4, -2, -2)\cdot\vec{v} & \ \text{if} \ n \equiv 1,4,7 \pmod{9}\\
( 9, 6, 12, -8, -14, 4, 10, 4)\cdot\vec{v} & \ \text{if} \ n \equiv 2,5,8 \pmod{9}\\
\end{array}
\end{cases}\\
F_3(n,1,1,0) &= 3^{n-3} - \frac{1}{81} \cdot \begin{cases}
\begin{array}{lr}
(3, 0, -3, -2, 1, -2, 1, 1)\cdot\vec{v} & \ \text{if} \ n \equiv 1,7 \pmod{9}\\
( -3, 3, 0, 4, -2, -2, -2, 4 )\cdot\vec{v} & \ \text{if} \ n \equiv 2 \pmod{9}\\
(3, 0, -3, 4, -2, 4, -2, -2)\cdot\vec{v} & \ \text{if} \ n \equiv 4 \pmod{9}\\
( -3, 3, 0, -2, 1, 1, 1, -2)\cdot\vec{v} & \ \text{if} \ n \equiv 5,8 \pmod{9}\\
\end{array}
\end{cases}\\
F_3(n,2,1,0) &= 3^{n-3} - \frac{1}{81} \cdot \begin{cases}
\begin{array}{lr}
( 3, 0, -3, -2, 1, -2, 1, 1 )\cdot\vec{v} & \ \text{if} \ n \equiv 1,5,7,8 \pmod{9}\\
( 3, 0, -3, 4, -2, 4, -2, -2)\cdot\vec{v} & \ \text{if} \ n \equiv 2,4 \pmod{9}\\
\end{array}
\end{cases}\\
F_3(n,0,2,0) &= 3^{n-3} - \frac{1}{81} \cdot \begin{cases}
\begin{array}{lr}
( 9, 6, 12, 4, -2, -2, 4, -2)\cdot\vec{v} & \ \text{if} \ n \equiv 1,4,7 \pmod{9}\\
( 12, 9, 6, -8, -14, 10, 4, 4)\cdot\vec{v} & \ \text{if} \ n \equiv 2,5,8 \pmod{9}\\
\end{array}
\end{cases}\\
F_3(n,1,2,0) &= 3^{n-3} - \frac{1}{81} \cdot \begin{cases}
\begin{array}{lr}
( 0, -3, 3, -2, 1, 1, -2, 1 )\cdot\vec{v} & \ \text{if} \ n \equiv 1,4,5,8 \pmod{9}\\
( 0, -3, 3, 4, -2, -2, 4, -2)\cdot\vec{v} & \ \text{if} \ n \equiv 2,7 \pmod{9}\\
\end{array}
\end{cases}\\
F_3(n,2,2,0) &= 3^{n-3} - \frac{1}{81} \cdot \begin{cases}
\begin{array}{lr}
( 0, -3, 3, -2, 1, 1, -2, 1)\cdot\vec{v} & \ \text{if} \ n \equiv 1,4 \pmod{9}\\
( -3, 3, 0, 4, -2, -2, -2, 4)\cdot\vec{v} & \ \text{if} \ n \equiv 2 \pmod{9}\\
( -3, 3, 0, -2, 1, 1, 1, -2)\cdot\vec{v} & \ \text{if} \ n \equiv 5,8 \pmod{9}\\
( 0, -3, 3, 4, -2, -2, 4, -2 )\cdot\vec{v} & \ \text{if} \ n \equiv 7 \pmod{9}\\
\end{array}
\end{cases}\\
F_3(n,0,0,1) &= 3^{n-3} - \frac{1}{81}
( -21, -15, -18, 4, 7, 7, 7, 4)\cdot\vec{v} 
\ \ \text{if} \ n \equiv 1,2,4,5,7,8 \pmod{9}\\
\end{flalign*}
}
{\small
\begin{flalign*}
F_3(n,1,0,1) &= 3^{n-3} - \frac{1}{81} \cdot \begin{cases}
\begin{array}{lr}
(-3, 3, 0, -2, 1, 1, 1, -2 )\cdot\vec{v} & \ \text{if} \ n \equiv 1,7 \pmod{9}\\
(3, 0, -3, -2, 1, -2, 1, 1)\cdot\vec{v} & \ \text{if} \ n \equiv 2,5 \pmod{9}\\
( -3, 3, 0, 4, -2, -2, -2, 4)\cdot\vec{v} & \ \text{if} \ n \equiv 4 \pmod{9}\\
( 3, 0, -3, 4, -2, 4, -2, -2 )\cdot\vec{v} & \ \text{if} \ n \equiv 8 \pmod{9}\\
\end{array}
\end{cases}\\
F_3(n,2,0,1) &= 3^{n-3} - \frac{1}{81} \cdot \begin{cases}
\begin{array}{lr}
( -3, 3, 0, -2, 1, 1, 1, -2)\cdot\vec{v} & \ \text{if} \ n \equiv 1,4 \pmod{9}\\
( 0, -3, 3, -2, 1, 1, -2, 1 )\cdot\vec{v} & \ \text{if} \ n \equiv 2,8 \pmod{9}\\
( 0, -3, 3, 4, -2, -2, 4, -2 )\cdot\vec{v} & \ \text{if} \ n \equiv 5 \pmod{9}\\
(-3, 3, 0, 4, -2, -2, -2, 4)\cdot\vec{v} & \ \text{if} \ n \equiv 7 \pmod{9}\\
\end{array}
\end{cases}\\
F_3(n,0,1,1) &= 3^{n-3} - \frac{1}{81} \cdot \begin{cases}
\begin{array}{lr}
( 12, 9, 6, -2, 1, -2, 1, 1 )\cdot\vec{v} & \ \text{if} \ n \equiv 1,4,7 \pmod{9}\\
(9, 6, 12, 4, 7, -2, -5, -2)\cdot\vec{v} & \ \text{if} \ n \equiv 2,5,8 \pmod{9}\\
\end{array}
\end{cases}\\
F_3(n,1,1,1) &= 3^{n-3} - \frac{1}{81} \cdot \begin{cases}
\begin{array}{lr}
( 3, 0, -3, -2, 1, -2, 1, 1)\cdot\vec{v} & \ \text{if} \ n \equiv 1,4 \pmod{9}\\
(-3, 3, 0, -2, 1, 1, 1, -2 )\cdot\vec{v} & \ \text{if} \ n \equiv 2,5 \pmod{9}\\
(3, 0, -3, 4, -2, 4, -2, -2)\cdot\vec{v} & \ \text{if} \ n \equiv 7 \pmod{9}\\
(-3, 3, 0, 4, -2, -2, -2, 4)\cdot\vec{v} & \ \text{if} \ n \equiv 8 \pmod{9}\\
\end{array}
\end{cases}\\
F_3(n,2,1,1) &= 3^{n-3} - \frac{1}{81} \cdot \begin{cases}
\begin{array}{lr}
(3, 0, -3, 4, -2, 4, -2, -2)\cdot\vec{v} & \ \text{if} \ n \equiv 1,5 \pmod{9}\\
( 3, 0, -3, -2, 1, -2, 1, 1)\cdot\vec{v} & \ \text{if} \ n \equiv 2,4,7,8 \pmod{9}\\
\end{array}
\end{cases}\\
F_3(n,0,2,1) &= 3^{n-3} - \frac{1}{81} \cdot \begin{cases}
\begin{array}{lr}
( 9, 6, 12, -2, 1, 1, -2, 1)\cdot\vec{v} & \ \text{if} \ n \equiv 1,4,7 \pmod{9}\\
(12, 9, 6, 4, 7, -5, -2, -2)\cdot\vec{v} & \ \text{if} \ n \equiv 2,5,8 \pmod{9}\\
\end{array}
\end{cases}\\
F_3(n,1,2,1) &= 3^{n-3} - \frac{1}{81} \cdot \begin{cases}
\begin{array}{lr}
(0, -3, 3, 4, -2, -2, 4, -2)\cdot\vec{v} & \ \text{if} \ n \equiv 1,8 \pmod{9}\\
( 0, -3, 3, -2, 1, 1, -2, 1)\cdot\vec{v} & \ \text{if} \ n \equiv 2,4,5,7 \pmod{9}\\
\end{array}
\end{cases}\\
F_3(n,2,2,1) &= 3^{n-3} - \frac{1}{81} \cdot \begin{cases}
\begin{array}{lr}
( 0, -3, 3, -2, 1, 1, -2, 1 )\cdot\vec{v} & \ \text{if} \ n \equiv 1,7 \pmod{9}\\
( -3, 3, 0, -2, 1, 1, 1, -2 )\cdot\vec{v} & \ \text{if} \ n \equiv 2,8 \pmod{9}\\
( 0, -3, 3, 4, -2, -2, 4, -2 )\cdot\vec{v} & \ \text{if} \ n \equiv 4 \pmod{9}\\
( -3, 3, 0, 4, -2, -2, -2, 4 )\cdot\vec{v} & \ \text{if} \ n \equiv 5 \pmod{9}\\
\end{array}
\end{cases}\\
F_3(n,0,0,2) &= 3^{n-3} - \frac{1}{81}
( -21, -15, -18, 4, 7, 7, 7, 4 )\cdot \vec{v} 
\ \ \text{if} \ n \equiv 1,2,4,5,7,8 \pmod{9}\\
F_3(n,1,0,2) &= 3^{n-3} - \frac{1}{81} \cdot \begin{cases}
\begin{array}{lr}
( -3, 3, 0, -2, 1, 1, 1, -2)\cdot\vec{v} & \ \text{if} \ n \equiv 1,4 \pmod{9}\\
( 3, 0, -3, -2, 1, -2, 1, 1 )\cdot\vec{v} & \ \text{if} \ n \equiv 2,8 \pmod{9}\\
( 3, 0, -3, 4, -2, 4, -2, -2 )\cdot\vec{v} & \ \text{if} \ n \equiv 5 \pmod{9}\\
( -3, 3, 0, 4, -2, -2, -2, 4 )\cdot\vec{v} & \ \text{if} \ n \equiv 7 \pmod{9}\\
\end{array}
\end{cases}\\
F_3(n,2,0,2) &= 3^{n-3} - \frac{1}{81} \cdot \begin{cases}
\begin{array}{lr}
( -3, 3, 0, -2, 1, 1, 1, -2 )\cdot\vec{v} & \ \text{if} \ n \equiv 1,7 \pmod{9}\\
( 0, -3, 3, -2, 1, 1, -2, 1)\cdot\vec{v} & \ \text{if} \ n \equiv 2,5 \pmod{9}\\
( -3, 3, 0, 4, -2, -2, -2, 4)\cdot\vec{v} & \ \text{if} \ n \equiv 4 \pmod{9}\\
( 0, -3, 3, 4, -2, -2, 4, -2 )\cdot\vec{v} & \ \text{if} \ n \equiv 8 \pmod{9}\\
\end{array}
\end{cases}\\
F_3(n,0,1,2) &= 3^{n-3} - \frac{1}{81} \cdot \begin{cases}
\begin{array}{lr}
( 12, 9, 6, -2, 1, -2, 1, 1 )\cdot\vec{v} & \ \text{if} \ n \equiv 1,4,7 \pmod{9}\\
( 9, 6, 12, 4, 7, -2, -5, -2)\cdot\vec{v} & \ \text{if} \ n \equiv 2,5,8 \pmod{9}\\
\end{array}
\end{cases}\\
F_3(n,1,1,2) &= 3^{n-3} - \frac{1}{81} \cdot \begin{cases}
\begin{array}{lr}
( 3, 0, -3, 4, -2, 4, -2, -2)\cdot\vec{v} & \ \text{if} \ n \equiv 1 \pmod{9}\\
(-3, 3, 0, -2, 1, 1, 1, -2)\cdot\vec{v} & \ \text{if} \ n \equiv 2,8 \pmod{9}\\
(3, 0, -3, -2, 1, -2, 1, 1)\cdot\vec{v} & \ \text{if} \ n \equiv 4,7 \pmod{9}\\
( -3, 3, 0, 4, -2, -2, -2, 4 )\cdot\vec{v} & \ \text{if} \ n \equiv 5 \pmod{9}\\
\end{array}
\end{cases}\\
F_3(n,2,1,2) &= 3^{n-3} - \frac{1}{81} \cdot \begin{cases}
\begin{array}{lr}
( 3, 0, -3, -2, 1, -2, 1, 1 )\cdot\vec{v} & \ \text{if} \ n \equiv 1,2,4,5 \pmod{9}\\
( 3, 0, -3, 4, -2, 4, -2, -2 )\cdot\vec{v} & \ \text{if} \ n \equiv 7,8 \pmod{9}\\
\end{array}
\end{cases}\\
F_3(n,0,2,2) &= 3^{n-3} - \frac{1}{81} \cdot \begin{cases}
\begin{array}{lr}
( 9, 6, 12, -2, 1, 1, -2, 1 )\cdot\vec{v} & \ \text{if} \ n \equiv 1,4,7 \pmod{9}\\
(12, 9, 6, 4, 7, -5, -2, -2)\cdot\vec{v} & \ \text{if} \ n \equiv 2,5,8 \pmod{9}\\
\end{array}
\end{cases}\\
\end{flalign*}
}
{\small
\begin{flalign*}
F_3(n,1,2,2) &= 3^{n-3} - \frac{1}{81} \cdot \begin{cases}
\begin{array}{lr}
( 0, -3, 3, -2, 1, 1, -2, 1 )\cdot\vec{v} & \ \text{if} \ n \equiv 1,2,7,8 \pmod{9}\\
( 0, -3, 3, 4, -2, -2, 4, -2 )\cdot\vec{v} & \ \text{if} \ n \equiv 4,5 \pmod{9}\\
\end{array}
\end{cases}\\
F_3(n,2,2,2) &= 3^{n-3} - \frac{1}{81} \cdot \begin{cases}
\begin{array}{lr}
( 0, -3, 3, 4, -2, -2, 4, -2 )\cdot\vec{v} & \ \text{if} \ n \equiv 1 \pmod{9}\\
( -3, 3, 0, -2, 1, 1, 1, -2)\cdot\vec{v} & \ \text{if} \ n \equiv 2,5 \pmod{9}\\
( 0, -3, 3, -2, 1, 1, -2, 1 )\cdot\vec{v} & \ \text{if} \ n \equiv 4,7 \pmod{9}\\
( -3, 3, 0, 4, -2, -2, -2, 4 )\cdot\vec{v} & \ \text{if} \ n \equiv 8 \pmod{9}\\
\end{array}
\end{cases}\\
\end{flalign*}
}
\end{theorem}

\subsubsection{Less redundant formulae.}\label{sec:lessredundant}
An alternative and more representationally efficient way to evaluate these formulae is to write:
\begin{eqnarray}
\nonumber F_3(n,t_1,t_2,t_3) &=&  \frac{1}{3}\#\{a_0 \in \F_{3^n} \mid T_1\big(a_{0}^4 - a_{0}^2 
- t_1\binom{n}{2}/n^2\big) = t_2, \\
\label{eq:altF3t1t2t3} && \ \ T_1\big( a_{0}^7 - a_{0}^5 + t_1(n+1)(a_{0}^4 - a_{0}^2)/n 
+ t_1\binom{n}{3}/n^2\big) = t_3\},
\end{eqnarray}
and then for each $t_1$ apply the indirect method to compute all nine $F_3(n,t_1,t_2,t_3)$ simultaneously.
In particular, we instead let $\vec{f} = (a_{0}^7 - a_{0}^5 + t_1(n+1)(a_{0}^4 - a_{0}^2)/n 
+ t_1\binom{n}{3}/n^2, a_{0}^4 - a_{0}^2 - t_1\binom{n}{2}/n^2)$.
Since there are now fewer summands, the total degree of the expressions is $84$ in the worst case: for $n \ge 3$ and coprime to $3$ we have
\[
F_3(n,0,0,0) = 3^{n-3} + \frac{1}{27}\big(3\rho_n(\epsilon_{2,1}) + \rho_n(\epsilon_{2,2}) + 2\rho_n(\epsilon_{2,3}) + 2\rho_n(\epsilon_{12,1}) + 2\rho_n(\epsilon_{12,2})+ 2\rho_n(\epsilon_{12,3}) \big).
\]
Comparing this with Theorem~\ref{thm:char3_3traces} implies the identity, valid for all $n \equiv 1,2,4,5,7,8 \pmod{9}$:
\begin{equation*}\label{rootidentity}
3\big(\rho_n(\epsilon_{2,1}) + \rho_n(\epsilon_{2,2}) + \rho_n(\epsilon_{2,3})\big) + 2\big(\rho_n(\epsilon_{6}) + \rho_n(\epsilon_{12,1}) + \rho_n(\epsilon_{12,2}) + \rho_n(\epsilon_{12,3}) +\rho_n(\epsilon_{12,4})\big) = 0.
\end{equation*}
This is therefore an example of a linear relation between the powers of the roots 
of the featured characteristic polynomials of Frobenius.
Furthermore, since the roots of $\epsilon_{2,1},\epsilon_{2,2},\epsilon_{2,3}$ are $12$-th roots of unity, one can check that 
\begin{equation*}
\rho_n(\epsilon_{2,1}) + \rho_n(\epsilon_{2,2}) + \rho_n(\epsilon_{2,3}) = \begin{cases} 
\ 6 \cdot 3^{n/2} \ \text{if} \  n \equiv 0 \pmod{12}\\
-6 \cdot 3^{n/2} \ \text{if} \  n \equiv 6 \pmod{12}\\
0 \ \text{otherwise}.
\end{cases}
\end{equation*}
We therefore deduce that for all $n \equiv 1,2,4,5,7,8 \pmod{9}$, we have 
\[
\rho_n(\epsilon_{6}) + \rho_n(\epsilon_{12,1}) + \rho_n(\epsilon_{12,2}) + \rho_n(\epsilon_{12,3}) +\rho_n(\epsilon_{12,4}) = 0,
\]
which does not seem obvious from the polynomials themselves.

\subsection{Direct method}\label{subsec:directF3}

For $n \ge 3$ and coprime to $3$ applying Equations~(\ref{eq:ternary1}) to~(\ref{eq:ternary3}) we have
\begin{eqnarray}
\nonumber F_3(n,t_1,t_2,t_3) &=& \#\{a \in \F_{3^n} \mid T_1(a) = t_1, \ T_2(a) = t_2, \ T_3(a) = t_3 \}\\
\nonumber &=& \frac{1}{3}\#\{a_0 \in \F_{3^n} \mid T_2(a_{0}^3 - a_{0} + t_1/\overline{n}) = t_2, \ T_3(a_{0}^3 - a_{0} + t_1/n) = t_3\}\\
\nonumber &=& \frac{1}{3^3}\#\{(a_0,a_1,a_2) \in (\F_{3^n})^3 \mid a_{1}^3 - a_{1} + t_2/n = a_{0}^4 - a_{0}^2 
- t_1\binom{n}{2}/n^2,\\
\label{eq:directF3t1t2t3} && a_{2}^3 - a_{2} + t_3/n = a_{0}^7 - a_{0}^5 + t_1(n+1)(a_{0}^4 - a_{0}^2)/n +
 t_1\binom{n}{3}/n^2\}.
\end{eqnarray}
These curves are all absolutely irreducible and of genus $21$, which is much less than half of the total degree of each expression in Theorem~\ref{thm:char3_3traces} and precisely half the degree of the worst case of the alternative indirect method of~\S\ref{sec:lessredundant}.
Based on some preliminary experiments, Magma can compute the zeta functions of the curves~(\ref{eq:directF3t1t2t3}) in a 
matter of days. 
%Note that for the all-zero case for all $n \ge 3$ we have
%\[
%F_3(n,0,0,0) = \frac{1}{3^3}\#\{(a_0,a_1,a_2) \in (\F_{3^n})^3 \mid a_{1}^3 - a_{1} = a_{0}^4 - a_{0}^2, \ a_{2}^3 - a_{2} = a_{0}^7 - a_{0}^5\}.
%\]
However, in order to compute these functions more efficiently, since there are only two linear conditions one can apply~\cite[Lemma 6]{AGGMY}. 
In particular, we have
\[
F_3(n,t_1,t_2,t_3) = \frac{1}{9}\big( V( (0,1)\cdot \vec{f}) + \sum_{\alpha \in \F_3} V( (1,\alpha) \cdot \vec{f}) - 3^n),
\]
where $\vec{f} = (a_{0}^7 - a_{0}^5 + t_1(n+1)(a_{0}^4 - a_{0}^2)/n 
+ t_1\binom{n}{3}/n^2, a_{0}^4 - a_{0}^2 - t_1\binom{n}{2}/n^2)$.
Since this uses the zero count $V(\vec{i} \cdot \vec{f})$ the resulting formula for $F_3(n,0,0,0)$ is valid for all $n \ge 3$.
This leads to the following refinement of Theorem~\ref{thm:char3_3traces}.

\begin{theorem}\label{thm:char3_3traces_new}
For $n \ge 3$ we have
{\small
\begin{flalign*}
F_3(n,0,0,0) &= 3^{n-3} - \frac{1}{27}
(0, 2, 1, 2, 0, 0, 0, 2)\cdot \vec{v} 
\ \ \text{for all} \ n \ge 3\\
F_3(n,1,0,0) &= 3^{n-3} - \frac{1}{27} \cdot \begin{cases}
\begin{array}{lr}
( 0, 2, 1, 2, 0, 0, 0, 2)\cdot\vec{v} & \ \ \text{if} \ n \equiv 1 \pmod{9}\\
( 2, 1, 0, 2, 0, 2, 0, 0)\cdot\vec{v} & \ \ \text{if} \ n \equiv 2 \pmod{9}\\
( 0, 2, 1, 0, 1, 1, 1, 0)\cdot\vec{v} & \ \ \text{if} \ n \equiv 4,7 \pmod{9}\\
( 2, 1, 0, 0, 1, 0, 1, 1)\cdot\vec{v} & \ \ \text{if} \ n \equiv 5,8 \pmod{9}\\
\end{array}
\end{cases}\\
F_3(n,2,0,0) &= 3^{n-3} - \frac{1}{27} \cdot \begin{cases}
\begin{array}{lr}
(  0, 2, 1, 2, 0, 0, 0, 2)\cdot\vec{v} & \ \text{if} \ n \equiv 1 \pmod{9}\\
(  1, 0, 2, 2, 0, 0, 2, 0)\cdot\vec{v} & \ \text{if} \ n \equiv 2 \pmod{9}\\
( 0, 2, 1, 0, 1, 1, 1, 0 )\cdot\vec{v} & \ \text{if} \ n \equiv 4,7 \pmod{9}\\
(  1, 0, 2, 0, 1, 1, 0, 1)\cdot\vec{v} & \ \ \text{if} \ n \equiv 5,8 \pmod{9}\\
\end{array}
\end{cases}\\
F_3(n,0,1,0) &= 3^{n-3} - \frac{1}{27} \cdot \begin{cases}
\begin{array}{lr}
(  2, 1, 0, 2, 0, 2, 0, 0 )\cdot\vec{v} & \ \text{if} \ n \equiv 1,4,7 \pmod{9}\\
(  1, 0, 2, 2, 0, 0, 2, 0 )\cdot\vec{v} & \ \text{if} \ n \equiv 2,5,8 \pmod{9}\\
\end{array}
\end{cases}\\
F_3(n,1,1,0) &= 3^{n-3} - \frac{1}{27} \cdot \begin{cases}
\begin{array}{lr}
( 2, 1, 0, 0, 1, 0, 1, 1 )\cdot\vec{v} & \ \text{if} \ n \equiv 1,7 \pmod{9}\\
( 0, 2, 1, 2, 0, 0, 0, 2)\cdot\vec{v} & \ \text{if} \ n \equiv 2 \pmod{9}\\
( 2, 1, 0, 2, 0, 2, 0, 0 )\cdot\vec{v} & \ \text{if} \ n \equiv 4 \pmod{9}\\
( 0, 2, 1, 0, 1, 1, 1, 0)\cdot\vec{v} & \ \text{if} \ n \equiv 5,8 \pmod{9}\\
\end{array}
\end{cases}\\
F_3(n,2,1,0) &= 3^{n-3} - \frac{1}{27} \cdot \begin{cases}
\begin{array}{lr}
( 2, 1, 0, 0, 1, 0, 1, 1)\cdot\vec{v} & \ \text{if} \ n \equiv 1,5,7,8 \pmod{9}\\
( 2, 1, 0, 2, 0, 2, 0, 0 )\cdot\vec{v} & \ \text{if} \ n \equiv 2,4 \pmod{9}\\
\end{array}
\end{cases}\\
F_3(n,0,2,0) &= 3^{n-3} - \frac{1}{27} \cdot \begin{cases}
\begin{array}{lr}
(  1, 0, 2, 2, 0, 0, 2, 0 )\cdot\vec{v} & \ \text{if} \ n \equiv 1,4,7 \pmod{9}\\
(  2, 1, 0, 2, 0, 2, 0, 0)\cdot\vec{v} & \ \text{if} \ n \equiv 2,5,8 \pmod{9}\\
\end{array}
\end{cases}\\
F_3(n,1,2,0) &= 3^{n-3} - \frac{1}{27} \cdot \begin{cases}
\begin{array}{lr}
( 1, 0, 2, 0, 1, 1, 0, 1   )\cdot\vec{v} & \ \text{if} \ n \equiv 1,4,5,8 \pmod{9}\\
(  1, 0, 2, 2, 0, 0, 2, 0  )\cdot\vec{v} & \ \text{if} \ n \equiv 2,7 \pmod{9}\\
\end{array}
\end{cases}\\
F_3(n,2,2,0) &= 3^{n-3} - \frac{1}{27} \cdot \begin{cases}
\begin{array}{lr}
(1, 0, 2, 0, 1, 1, 0, 1 )\cdot\vec{v} & \ \text{if} \ n \equiv 1,4 \pmod{9}\\
( 0, 2, 1, 2, 0, 0, 0, 2  )\cdot\vec{v} & \ \text{if} \ n \equiv 2 \pmod{9}\\
(  0, 2, 1, 0, 1, 1, 1, 0)\cdot\vec{v} & \ \text{if} \ n \equiv 5,8 \pmod{9}\\
( 1, 0, 2, 2, 0, 0, 2, 0 )\cdot\vec{v} & \ \text{if} \ n \equiv 7 \pmod{9}\\
\end{array}
\end{cases}\\
F_3(n,0,0,1) &= 3^{n-3} - \frac{1}{27}
( 0, 2, 1, 0, 1, 1, 1, 0)\cdot\vec{v} 
\ \ \text{if} \ n \equiv 1,2,4,5,7,8 \pmod{9}\\
F_3(n,1,0,1) &= 3^{n-3} - \frac{1}{27} \cdot \begin{cases}
\begin{array}{lr}
( 0, 2, 1, 0, 1, 1, 1, 0 )\cdot\vec{v} & \ \text{if} \ n \equiv 1,7 \pmod{9}\\
(  2, 1, 0, 0, 1, 0, 1, 1)\cdot\vec{v} & \ \text{if} \ n \equiv 2,5 \pmod{9}\\
(  0, 2, 1, 2, 0, 0, 0, 2)\cdot\vec{v} & \ \text{if} \ n \equiv 4 \pmod{9}\\
( 2, 1, 0, 2, 0, 2, 0, 0  )\cdot\vec{v} & \ \text{if} \ n \equiv 8 \pmod{9}\\
\end{array}
\end{cases}\\
F_3(n,2,0,1) &= 3^{n-3} - \frac{1}{27} \cdot \begin{cases}
\begin{array}{lr}
( 0, 2, 1, 0, 1, 1, 1, 0  )\cdot\vec{v} & \ \text{if} \ n \equiv 1,4 \pmod{9}\\
(  1, 0, 2, 0, 1, 1, 0, 1 )\cdot\vec{v} & \ \text{if} \ n \equiv 2,8 \pmod{9}\\
(  1, 0, 2, 2, 0, 0, 2, 0  )\cdot\vec{v} & \ \text{if} \ n \equiv 5 \pmod{9}\\
(  0, 2, 1, 2, 0, 0, 0, 2)\cdot\vec{v} & \ \text{if} \ n \equiv 7 \pmod{9}\\
\end{array}
\end{cases}\\
\end{flalign*}
}
{\small
\begin{flalign*}
F_3(n,0,1,1) &= 3^{n-3} - \frac{1}{27} \cdot \begin{cases}
\begin{array}{lr}
(  2, 1, 0, 0, 1, 0, 1, 1 )\cdot\vec{v} & \ \text{if} \ n \equiv 1,4,7 \pmod{9}\\
(  1, 0, 2, 0, 1, 1, 0, 1 )\cdot\vec{v} & \ \text{if} \ n \equiv 2,5,8 \pmod{9}\\
\end{array}
\end{cases}\\
F_3(n,1,1,1) &= 3^{n-3} - \frac{1}{27} \cdot \begin{cases}
\begin{array}{lr}
( 2, 1, 0, 0, 1, 0, 1, 1  )\cdot\vec{v} & \ \text{if} \ n \equiv 1,4 \pmod{9}\\
( 0, 2, 1, 0, 1, 1, 1, 0 )\cdot\vec{v} & \ \text{if} \ n \equiv 2,5 \pmod{9}\\
(  2, 1, 0, 2, 0, 2, 0, 0)\cdot\vec{v} & \ \text{if} \ n \equiv 7 \pmod{9}\\
( 0, 2, 1, 2, 0, 0, 0, 2 )\cdot\vec{v} & \ \text{if} \ n \equiv 8 \pmod{9}\\
\end{array}
\end{cases}\\
F_3(n,2,1,1) &= 3^{n-3} - \frac{1}{27} \cdot \begin{cases}
\begin{array}{lr}
(  2, 1, 0, 2, 0, 2, 0, 0)\cdot\vec{v} & \ \text{if} \ n \equiv 1,5 \pmod{9}\\
( 2, 1, 0, 0, 1, 0, 1, 1 )\cdot\vec{v} & \ \text{if} \ n \equiv 2,4,7,8 \pmod{9}\\
\end{array}
\end{cases}\\
F_3(n,0,2,1) &= 3^{n-3} - \frac{1}{27} \cdot \begin{cases}
\begin{array}{lr}
( 1, 0, 2, 0, 1, 1, 0, 1 )\cdot\vec{v} & \ \text{if} \ n \equiv 1,4,7 \pmod{9}\\
( 2, 1, 0, 0, 1, 0, 1, 1 )\cdot\vec{v} & \ \text{if} \ n \equiv 2,5,8 \pmod{9}\\
\end{array}
\end{cases}\\
F_3(n,1,2,1) &= 3^{n-3} - \frac{1}{27} \cdot \begin{cases}
\begin{array}{lr}
(  1, 0, 2, 2, 0, 0, 2, 0 )\cdot\vec{v} & \ \text{if} \ n \equiv 1,8 \pmod{9}\\
(  1, 0, 2, 0, 1, 1, 0, 1 )\cdot\vec{v} & \ \text{if} \ n \equiv 2,4,5,7 \pmod{9}\\
\end{array}
\end{cases}\\
F_3(n,2,2,1) &= 3^{n-3} - \frac{1}{27} \cdot \begin{cases}
\begin{array}{lr}
(  1, 0, 2, 0, 1, 1, 0, 1 )\cdot\vec{v} & \ \text{if} \ n \equiv 1,7 \pmod{9}\\
( 0, 2, 1, 0, 1, 1, 1, 0  )\cdot\vec{v} & \ \text{if} \ n \equiv 2,8 \pmod{9}\\
(  1, 0, 2, 2, 0, 0, 2, 0)\cdot\vec{v} & \ \text{if} \ n \equiv 4 \pmod{9}\\
(  0, 2, 1, 2, 0, 0, 0, 2 )\cdot\vec{v} & \ \text{if} \ n \equiv 5 \pmod{9}\\
\end{array}
\end{cases}\\
F_3(n,0,0,2) &= 3^{n-3} - \frac{1}{27}
( 0, 2, 1, 0, 1, 1, 1, 0 )\cdot \vec{v} 
\ \ \text{if} \ n \equiv 1,2,4,5,7,8 \pmod{9}\\
F_3(n,1,0,2) &= 3^{n-3} - \frac{1}{27} \cdot \begin{cases}
\begin{array}{lr}
(  0, 2, 1, 0, 1, 1, 1, 0 )\cdot\vec{v} & \ \text{if} \ n \equiv 1,4 \pmod{9}\\
(  2, 1, 0, 0, 1, 0, 1, 1)\cdot\vec{v} & \ \text{if} \ n \equiv 2,8 \pmod{9}\\
( 2, 1, 0, 2, 0, 2, 0, 0  )\cdot\vec{v} & \ \text{if} \ n \equiv 5 \pmod{9}\\
( 0, 2, 1, 2, 0, 0, 0, 2)\cdot\vec{v} & \ \text{if} \ n \equiv 7 \pmod{9}\\
\end{array}
\end{cases}\\
F_3(n,2,0,2) &= 3^{n-3} - \frac{1}{27} \cdot \begin{cases}
\begin{array}{lr}
(  0, 2, 1, 0, 1, 1, 1, 0  )\cdot\vec{v} & \ \text{if} \ n \equiv 1,7 \pmod{9}\\
(1, 0, 2, 0, 1, 1, 0, 1  )\cdot\vec{v} & \ \text{if} \ n \equiv 2,5 \pmod{9}\\
( 0, 2, 1, 2, 0, 0, 0, 2  )\cdot\vec{v} & \ \text{if} \ n \equiv 4 \pmod{9}\\
(  1, 0, 2, 2, 0, 0, 2, 0)\cdot\vec{v} & \ \text{if} \ n \equiv 8 \pmod{9}\\
\end{array}
\end{cases}\\
F_3(n,0,1,2) &= 3^{n-3} - \frac{1}{27} \cdot \begin{cases}
\begin{array}{lr}
( 2, 1, 0, 0, 1, 0, 1, 1 )\cdot\vec{v} & \ \text{if} \ n \equiv 1,4,7 \pmod{9}\\
(  1, 0, 2, 0, 1, 1, 0, 1)\cdot\vec{v} & \ \text{if} \ n \equiv 2,5,8 \pmod{9}\\
\end{array}
\end{cases}\\
F_3(n,1,1,2) &= 3^{n-3} - \frac{1}{27} \cdot \begin{cases}
\begin{array}{lr}
(  2, 1, 0, 2, 0, 2, 0, 0 )\cdot\vec{v} & \ \text{if} \ n \equiv 1 \pmod{9}\\
( 0, 2, 1, 0, 1, 1, 1, 0 )\cdot\vec{v} & \ \text{if} \ n \equiv 2,8 \pmod{9}\\
(  2, 1, 0, 0, 1, 0, 1, 1)\cdot\vec{v} & \ \text{if} \ n \equiv 4,7 \pmod{9}\\
([ 0, 2, 1, 2, 0, 0, 0, 2 )\cdot\vec{v} & \ \text{if} \ n \equiv 5 \pmod{9}\\
\end{array}
\end{cases}\\
F_3(n,2,1,2) &= 3^{n-3} - \frac{1}{27} \cdot \begin{cases}
\begin{array}{lr}
(   2, 1, 0, 0, 1, 0, 1, 1 )\cdot\vec{v} & \ \text{if} \ n \equiv 1,2,4,5 \pmod{9}\\
( 2, 1, 0, 2, 0, 2, 0, 0  )\cdot\vec{v} & \ \text{if} \ n \equiv 7,8 \pmod{9}\\
\end{array}
\end{cases}\\
F_3(n,0,2,2) &= 3^{n-3} - \frac{1}{27} \cdot \begin{cases}
\begin{array}{lr}
(  1, 0, 2, 0, 1, 1, 0, 1  )\cdot\vec{v} & \ \text{if} \ n \equiv 1,4,7 \pmod{9}\\
(  2, 1, 0, 0, 1, 0, 1, 1 )\cdot\vec{v} & \ \text{if} \ n \equiv 2,5,8 \pmod{9}\\
\end{array}
\end{cases}\\
F_3(n,1,2,2) &= 3^{n-3} - \frac{1}{27} \cdot \begin{cases}
\begin{array}{lr}
( 1, 0, 2, 0, 1, 1, 0, 1  )\cdot\vec{v} & \ \text{if} \ n \equiv 1,2,7,8 \pmod{9}\\
( 1, 0, 2, 2, 0, 0, 2, 0 )\cdot\vec{v} & \ \text{if} \ n \equiv 4,5 \pmod{9}\\
\end{array}
\end{cases}\\
F_3(n,2,2,2) &= 3^{n-3} - \frac{1}{27} \cdot \begin{cases}
\begin{array}{lr}
(   1, 0, 2, 2, 0, 0, 2, 0)\cdot\vec{v} & \ \text{if} \ n \equiv 1 \pmod{9}\\
(  0, 2, 1, 0, 1, 1, 1, 0)\cdot\vec{v} & \ \text{if} \ n \equiv 2,5 \pmod{9}\\
(  1, 0, 2, 0, 1, 1, 0, 1 )\cdot\vec{v} & \ \text{if} \ n \equiv 4,7 \pmod{9}\\
(  0, 2, 1, 2, 0, 0, 0, 2  )\cdot\vec{v} & \ \text{if} \ n \equiv 8 \pmod{9}\\
\end{array}
\end{cases}\\
\end{flalign*}
}
\end{theorem}
Observe that the total degrees of the featured polynomials in Theorem~\ref{thm:char3_3traces_new} is always $42$, which means these are the 
characteristic polynomials of Frobenius arising from the direct method and should have optimal representational efficiency.

If one compares for instance the formulae for $F_3(n,0,1,2)$ in Theorems~\ref{thm:char3_3traces} and~\ref{thm:char3_3traces_new}
one can further deduce that for all $n \equiv 2,5,8 \pmod{9}$, we have 
$\rho_n(\epsilon_{6}) + \rho_n(\epsilon_{12,1}) = 0$ and $\rho_n(\epsilon_{12,2}) + \rho_n(\epsilon_{12,3}) +\rho_n(\epsilon_{12,4}) = 0$, which again are seemingly not
obvious from the polynomials alone. The three observed identities and possibly any others arising by such a comparison would probably 
allow one to transform Theorem~\ref{thm:char3_3traces} into Theorem~\ref{thm:char3_3traces_new} without carrying out any zeta 
function computations.
Conversely, although our goal was to compute formulae for $F_3(n,t_1,t_2,t_3)$, by using various approaches to do so one can also deduce identities 
between the roots of the featured characteristic polynomials of Frobenius for residues of $n \bmod{9}$ without computing any of the roots 
explicitly.

%-----------------------------------------------------------------
%-----------------------------------------------------------------

\section{Final Remarks and Open Problems}\label{sec:summary}

We have presented the first algorithmic approaches to solving the prescribed traces problem and therefore the prescribed coefficients problem.
While our main algorithm for $l < p$ is extremely simple to state and very efficient, the $l \ge p$ case in full generality remains open.
There are many properties and consequences of our approaches that have yet to be determined or explored. 
We now present some open problems, in what may be considered to be an approximately increasing order of interest.

\vspace{3mm}

\noindent {\bf Problem 1:} Provide an algorithm which for any $q = p^r$ and $l$ outputs the transforms between $I_q(n,t_1,\ldots,t_l)$ and $F_q(n,t_1,\ldots,t_l)$, and deduce its complexity.

\vspace{3mm}

\noindent {\bf Problem 2:}  Compute the characteristic polynomials of Frobenius of the curves arising for $F_2(n,t_1,\ldots,t_6)$ and 
$F_2(n,t_1,\ldots,t_7)$, for $n$ odd. Although it is feasible to compute these by brute force point counting, it would be preferable to employ a more 
elegant approach, perhaps by computing the quotient curves arising from various curve automorphisms and then applying a theorem of Kani and 
Rosen~\cite{kanirosen} to infer the decomposition of their Jacobians.

\vspace{3mm}

\noindent {\bf Problem 3:}  Compute formulae for $F_2(n,t_1,\ldots,t_l)$ with $4 \le l \le 7$ for all $n \ge l$.

\vspace{3mm}

\noindent {\bf Problem 4:}  More generally, for the main algorithm provide a method for solving the $n \equiv 0 \pmod{p}$ cases, for any $q \ge 2$.
Note that if solved then the formulae for $F_q(n,t_1,\ldots,t_l)$ for each of the $p$ residue classes of $n$ can be unified via Fourier analysis using the complex $p$-th roots of unity, as in~\cite{AGGMY}.
Similarly, if $l \ge p$ then as per Corollary~\ref{cor:period} the complex $p^{1+ \lfloor \log_{p} l \rfloor}$-th roots of unity can be employed.
However, it may be preferable to retain the residue class distinctions for simplicity (cf.~\cite[Prop. 4 and 5]{AGGMY}).

\vspace{3mm}

\noindent {\bf Problem 5:} Determine if it is possible to express $T_l(\alpha - \beta)$ in terms of lower degree traces, over $\Z$, for any or all
 $l > 7$, so that one may extend the $q = 2$ approach.

\vspace{3mm}

\noindent {\bf Problem 6:} For the indirect method in the main algorithm when $q > p$, remove the reliance on Lemma~\ref{thm:extensionreduction} 
so that the $\frac{q-1}{p-1}$ factor in the sum in~(\ref{eq:maintheorem}) in Theorem~\ref{thm:maintheorem} may be removed.

\vspace{3mm}

\noindent {\bf Problem 7:} When $l < p$ provide an algorithm to compute the zeta function of the curve~(\ref{eq:main_direct}) arising
from the direct method, or indeed of the fibre product~(\ref{eq:newmain}), which is more efficient for a single $F_q(n,t_1,\ldots,t_l)$ than 
executing the indirect method. This is also desirable since the direct method seems to produce the most compact formulae.

\vspace{3mm}

\noindent {\bf Problem 8:} As is evident from the examples we have given for $q = 2$, there are often several ways to associate an affine curve 
$C_{\vec{r},\overline{n}}$ to a given counting problem $F_q(n,t_{l_0},\ldots,t_{l_{m-1}})$ with differing numbers of auxiliary variables 
$(a_0,\ldots,a_{s-1})$ with corresponding linear traces $\vec{r} = (r_0,\ldots,r_{s-1})$. Not only does this affect the complexity of determining 
the relevant
characteristic values, but each such curve may lead to different (but necessarily numerically identical) formulae. One can associate a generating, 
or zeta function to each prescribed traces problem, by letting 
$N_n = q^m \cdot F_q(n,t_{l_0},\ldots,t_{l_{m-1}}) + 1$ and as usual defining
\[
Z( F_q(n,t_{l_0},\ldots,t_{l_{m-1}}); t) = \exp{\left(  \sum_{n=1}^{\infty}  \frac{N_n}{n} t^n \right)}.
\]
If $C_{\vec{r},\overline{n}}$ is absolutely irreducible for each $\vec{r}$ and assuming Problem 4 has been solved then for all $\overline{n} \in \{0,\ldots,p-1\}$ one has:
\begin{equation}\label{eq:zetafunction}
Z( F_q(n,t_{l_0},\ldots,t_{l_{m-1}}); t) = \frac{L_{\overline{n}}(t)^{1/q^s}}{(1-t)(1-qt)}, \ \text{for all} \ \ n \equiv \overline{n}\pmod{p} \ 
\ \text{with} \ \ n \ge l_{m-1},
\end{equation}
where $L_{\overline{n}}(t) = c_0 + c_1t + \cdots + c_{2g}t^{2g} \in \Z[t]$ with $c_0 = 1$ is the product of the characteristic polynomials of Frobenius
of the $q^s$ specialisations of $C_{\vec{r},\overline{n}}$. 
Since $L_{\overline{n}}(t)$ and $s$ in Eq.~(\ref{eq:zetafunction}) are non-canonical, this is less than aesthetically pleasing in comparison 
to the rational zeta functions of smooth projective curves. Is there a canonical way to associate such a curve, which also avoids all redundancy?
Perhaps one can associate a motivic zeta function to each $F_q(n,t_{l_0},\ldots,t_{l_{m-1}})$ \`a la~\cite{duSautoy1,duSautoy2}, which resolve 
a similar non-canonicality issue when associating varieties to counting problems in the theory of $p$-groups and nilpotent groups?

\vspace{3mm}

\noindent {\bf Problem 9:} Regarding the general $l \ge p$ case, another natural question is whether or not it is possible to obviate the 
failure of Newton's identities by working $p$-adically and in Galois rings, \`a la Fan and Han's refinement~\cite{fanhan} of Han's work 
on Cohen's problem~\cite{han}? See also the exposition of Cohen~\cite{cohenexpo}. 
If so, this could lead to an algorithm for solving the prescribed traces problem for any number of coefficients in any positions.
Assuming this can be done, for the sake of generality we make the following conjecture.
\begin{conjecture}\label{conj:generalconj}
For every $q = p^r$, $1 \le l_0 < \cdots < l_{m-1}$, $(t_{l_0},\ldots,t_{l_{m-1}}) \in (\F_q)^m$ and 
$\overline{n} \in \{0,\ldots,p^{1 + \lfloor \log_p{l_{m-1}} \rfloor}-1\}$, 
there exists $\omega_1,\ldots,\omega_N \in \overline{\Z}$, all of norm $\sqrt{q}$, $\upsilon_1,\ldots,\upsilon_N \in \Z$ and an integer $s \ge 0$ such that for 
all $n \equiv \overline{n} \pmod{p^{1 + \lfloor \log_p{l_{m-1}} \rfloor}}$ with $n \ge l_{m-1}$ one has
\begin{equation*}\label{eq:mainconj}
F_q(n,t_{l_0},\ldots,t_{l_{m-1}}) = \frac{1}{q^m}\Big( q^n + \frac{1}{q^s}\sum_{i=1}^N \upsilon_{i} \alpha_{i}^n \Big) = q^{n-m} + O(q^{n/2}).
\end{equation*}
\end{conjecture}
Note that the sum is missing the factor $\frac{q-1}{p-1}$ from Theorem~\ref{thm:maintheorem}, which will be avoidable should Problem 6 be solved.
Also note that it may always be possible to set $s = 0$.

\vspace{3mm}

\noindent {\bf Problem 10:}
Should sufficiently general versions of Problems $1$, $4$ and $9$ be resolved positively and bounds 
on the genera of the resulting curves be sufficiently small, then can one prove interesting existence results -- in the best case 
with up to $\lfloor (1/2 - \epsilon)n \rfloor$ coefficients prescribed for any $\epsilon > 0$ -- when $n$ is sufficiently large for the main 
term to dominate the error term in Conjecture~\ref{eq:mainconj}?

\section*{Acknowledgements}
This work was supported by the Swiss National Science Foundation via grant number 200021-156420, while the author was based in the Laboratory for Cryptologic Algorithms, \'Ecole polytechnique f\'ed\'erale de Lausanne, Switzerland. Some of the ideas were developed 
while the author was visiting Steven Galbraith at The University of Auckland in April 2016; I would like to thank him for his excellent 
hospitality and for some enlightening discussions on intersections. I would also like to thank: Wouter Castryck and Jan Tuitman for explaining to me 
the state-of-the-art in $p$-adic point counting techniques for curves; Daniel Panario for his useful comments on an early draft of this work; 
Alan Lauder for answering my questions regarding his and Wan's algorithm; Claus Diem and Benjamin Wesolowski for discussions; and Omran Ahmadi 
for providing the translation of Lemma~\ref{thm:extensionreduction}. Finally, I would very much like to thank Thorsten Kleinjung for numerous 
very helpful discussions and suggestions, and for acting as a sounding board throughout much of this work.

\bibliographystyle{plain}
\bibliography{4andmoreNew}

\end{document}